\documentclass[11pt,a4paper]{article}

\usepackage[english]{babel}
\usepackage[T1]{fontenc}
\usepackage[utf8]{inputenc}
\usepackage[
  margin=2.5cm,
  includefoot,
  footskip=30pt,
]{geometry}
\usepackage{lmodern}
\usepackage{amsmath}
\usepackage{amssymb}
\usepackage{amsthm}
\usepackage{booktabs}
\usepackage{array}
\usepackage{graphicx}
\usepackage{cite}
\usepackage[dvipsnames]{xcolor}
\usepackage{tikz-cd}
\usetikzlibrary{decorations.markings,angles}
\usepackage{stmaryrd} \SetSymbolFont{stmry}{bold}{U}{stmry}{m}{n}
\usepackage{mathrsfs}
\usepackage{extarrows}
\usepackage{mathtools}
\usepackage{tensor}
\usepackage{enumerate}
\usepackage{adjustbox}
\usepackage[mathcal]{eucal}
\usepackage{hyperref}
\hypersetup{
	colorlinks,
	urlcolor={red!55!black},
	citecolor={green!60!black},
	linkcolor={red!55!black}
}
\usepackage{bm}
\usepackage{tabularx}
\usepackage{cleveref}
\usepackage[strict]{changepage}
\usepackage{tensor}

\usetikzlibrary{calc}
\tikzset{
    vertex/.style = {
        circle,
        fill=black,
        outer sep=3pt,
        inner sep=1.5pt
    }
}
\tikzset{
    frvertex/.style = {
       inner sep=2pt, outer sep=3pt, fill=blue, draw=none}
    }

\mathchardef\mhyphen="2D
\def\on{\operatorname}

\makeatletter
\providecommand{\leftsquigarrow}{%
  \mathrel{\mathpalette\reflect@squig\relax}%
}
\newcommand{\reflect@squig}[2]{%
  \reflectbox{$\m@th#1\rightsquigarrow$}%
}
\makeatother

\definecolor{ao}{rgb}{0.0, 0.5, 0.0}

\newtheorem{theorem}{Theorem}[section]
\newtheorem{lemma}[theorem]{Lemma}

\newtheorem{proposition}[theorem]{Proposition}
\newtheorem{corollary}[theorem]{Corollary}
\newtheorem{introthm}{Theorem}

\newtheorem{introprop}{Proposition}

\theoremstyle{definition}
\newtheorem{construction}[theorem]{Construction}
\newtheorem{definition}[theorem]{Definition}

\newtheorem{remark}[theorem]{Remark}
\newtheorem{example}[theorem]{Example}

\newcommand\noloc{%
  \nobreak
  \mspace{6mu plus 1mu}
  {:}
  \nonscript\mkern-\thinmuskip
  \mathpunct{}
  \mspace{2mu}
}

\newcommand\cocolon{%
  \nobreak
  \mspace{6mu plus 1mu}
  {:}
  \nonscript\mkern-\thinmuskip
  \mathpunct{}
  \mspace{2mu}
}

\newcommand{\rgraph}{{\bf G}}
\newcommand{\srgraph}{{\bf H}}

\newcommand{\cwt}{\gamma^\circlearrowright}
\newcommand{\ccwt}{\gamma^\circlearrowleft}
\newcommand{\indL}{\on{ind}^L}

\newcommand{\indR}{\on{ind}^R}

\newcommand{\C}{\mathcal{C}}
\newcommand{\D}{\mathcal{D}}

\newcommand{\V}{\mathcal{V}}
\newcommand{\N}{\mathcal{N}}
\newcommand{\T}{\mathcal{T}}

\newcommand{\ADE}{I}

\newcommand{\glsec}{\mathcal{H}}
\newcommand{\losec}{\mathcal{L}}

\title{Induction in perverse schobers and cluster tilting theory}
\author{Merlin Christ}
\date{\today}

\begin{document}
\maketitle

\abstract{We exhibit gluing properties of cluster tilting subcategories in exact $\infty$-categories within the framework of perverse schobers on surfaces with boundary. These results are based on a study of the restriction functors from global sections of perverse schobers to local sections and their adjoint induction functors. New examples include cluster tilting subcategories in higher rank topological Fukaya categories related to higher Teichm\"uller theory, and cluster tilting subcategories arising from marked surfaces with punctures.}

\tableofcontents

\section{Introduction}

Cluster algebras are a class of commutative algebras equipped with special generators called clusters, that are related with each other via a combinatorial process called mutation \cite{FZ02}. Cluster tilting objects in triangulated or extriangulated categories are special generators that can be mutated in a similar way to clusters in a cluster algebra \cite{IY08}. Cluster tilting objects are a central concept in the additive categorification of cluster algebras, as they correspond to the clusters in the cluster algebra. Typical examples of categories with cluster tilting objects include the triangulated cluster categories \cite{BMRRT06,Kel05,Ami09}, classes of Frobenius exact $1$-categories \cite{GLS06,JKS16}, and more recently the extriangulated Higgs categories \cite{Wu21}. The existence of cluster tilting subcategories can have strong structural implications for the ambient categories \cite{JM22,KL23}. 
  
In this paper, we exhibit a connection between cluster tilting theory and the notion of a perverse schober on a surface parametrized by a spanning ribbon graph, introduced in \cite{Chr22}. Such perverse schobers are constructible sheaves of stable $\infty$-categories on a ribbon graph, satisfying local properties, categorifying the local perversity properties of perverse sheaves. The notion of a perverse
schober, proposed by Kapranov--Schechtman \cite{KS14}, remains conjectural for general stratified spaces.  We show that within the framework of parametrized perverse schobers on surfaces, cluster tilting subcategories can be glued from local cluster tilting subcategories. 

To go from a local section, meaning an element in the stalk at a vertex $v$ of the ribbon graph, to a global section, we consider \textit{induction functors}. There is left induction $\indL_v$ and right induction $\indR_v$, defined as the left, respectively, right adjoint of the restriction functor $\on{ev}_v$ from global sections to local sections at $v$. We emphasize that this type of induction is internal to a given perverse schober and thus distinct from the question of induction/extension of perverse schobers along inclusions of stratified spaces.  We undertake a detailed study of these induction functors, describing them in terms of the combinatorial geometry of the marked surface  with spanning graph, and more specifically in terms of curve and web trajectories in the surface. We use this to obtain the following gluing result for cluster tilting subcategories:

\begin{introthm}\label{introthm:1}
Let $\rgraph$ be a spanning ribbon graph of a marked surface without $1$-valent vertices. Let $\mathcal{F}$ be a $\rgraph$-parametrized perverse schober. Given for every vertex $v\in \rgraph_0$ a cluster tilting subcategory $\mathcal{T}_v\subset \mathcal{F}(v)$, the additive closure of the union of left inductions
\[ \bigcup_{v\in \rgraph_0} \indL_v(\mathcal{T}_v)\]
defines a cluster tilting subcategory of the Frobenius exact $\infty$-category $\mathcal{H}(\rgraph,\mathcal{F})$ of global sections of $\mathcal{F}$. 

The additive closure of the union of right inductions $\bigcup_{v\in \rgraph_0} \indR_v(\mathcal{T}_v)$ also defines a cluster tilting subcategory.
\end{introthm}

Under mild assumptions, the above lifts the well-known amalgamation procedure for ice quivers of Fock--Goncharov \cite{FG06a} to cluster tilting objects, see also further below. 

We remark that both $\mathcal{H}(\rgraph,\mathcal{F})$ and $\mathcal{F}(v)$ are understood to be equipped with $\infty$-categorical Frobenius exact structures (which induce Frobenius extriangulated structures on the homotopy $1$-categories in the sense of \cite{NP19}). The origin of these $\infty$-categorical exact structures was discussed in \cite{Chr22b}, they arise as the relative exact structure induced by the boundary restriction functors of the perverse schober. The framework of perverse schobers allows to conveniently glue these exact $\infty$-categories, ensuring that they remain Frobenius.

The articles \cite{Chr22b,Chr25b} discuss gluing results for Higgs categories of marked surfaces. The present work differs from \cite{Chr22b,Chr25b} in that we consider ambient exact $\infty$-categories glued via a perverse schober, and present a general recipe for obtaining global cluster tilting subcategories given local ones. 

The main new examples to which we apply \Cref{introthm:1} include cluster tilting objects in higher rank topological Fukaya categories, whose corresponding cluster algebras conjecturally arise from higher Teichm\"uller theory. These examples will be explored in \cite{Chr25b}. Other examples include categorifications of cluster algebras arising from marked surfaces with punctures.

We proceed in \Cref{introsec:induction} by discussing the description of the induction functors. In \Cref{introsec:gluingCT}, we further discuss \Cref{introthm:1} and some related results. Examples are discussed in \Cref{introsec:examples}. In \Cref{introsec:outlook}, we mention some open questions.

\subsection{Induction and restriction functors for perverse schobers}\label{introsec:induction}

Let ${\bf S}$ be a marked surface with boundary and spanning ribbon graph $\rgraph$. The exit path category $\on{Exit}(\rgraph)$ is given by the $1$-category whose objects are the vertices and edges of $\rgraph$, and whose morphisms go from vertices to edges according to incidence. A $\rgraph$-parametrized perverse schober refers to a functor $\on{Exit}(\rgraph)\to \on{St}$ with target the $\infty$-category of stable $\infty$-categories, satisfying local conditions. Such a functor encodes a constructible sheaf of stable $\infty$-categories on $\rgraph$. Towards a discussion of the restriction and induction functors, we first inspect these local conditions for perverse schobers.\\ 

\noindent {\bf Induction on an $n$-gon}

Fix a vertex $v$ of $\rgraph$ of valency $m$, with incident edges $e_1,\dots,e_m$. In the case $m=1$, near $v$, a functor $\on{Exit}(\rgraph)\to \on{St}$ amounts to an exact functor $\mathcal{F}(v\to e)\colon \mathcal{F}(v)\to \mathcal{F}(e)$ between stable $\infty$-categories. The local perversity condition is in this case that the functor $\mathcal{F}(v\to e)$ is spherical. For valency $m\geq 2$, the local perversity condition takes a different form: there are $m$ functors $\mathcal{F}(v\to e_i)\colon \mathcal{F}(v)\to \mathcal{F}(e_i)$, which are each assumed to admit right adjoints $\mathcal{F}(v\to e_i)^R$. Up to some additional technical conditions, see \Cref{def:schobernspider}, the perversity condition at $v$ states that
\begin{itemize}
\item the functor $\mathcal{F}(v\to e_i)^R$ is fully faithful,
\item the functor $\mathcal{F}(v\to e_{i})\circ \mathcal{F}(v\to e_{i+1})^R\colon \mathcal{F}(e_{i+1})\to \mathcal{F}(e_{i})$ is an equivalence, and
\item the functor $\mathcal{F}(v\to e_i)\circ \mathcal{F}(v\to e_j)^R$ vanishes for $j\not=i,i+1$. 
\end{itemize} 

The functor $\mathcal{F}(v\to e_{i+1})$ can be seen as restriction of local sections supported near $v$ to local sections supported only at the edge $e_{i+1}$. Its right adjoint $\mathcal{F}(v\to e_{i+1})^R$ thus describes induction from the edge to the vertex. It is productive to reinterpret the above algebraic conditions geometrically. Namely, we attach to each induction functor $\mathcal{F}(v\to e_{i+1})^R$ an embedded curve $\gamma^\circlearrowleft$ (in blue) considered up to suitable homotopy in the $m$-gon ($m=5$ for illustration), as follows.

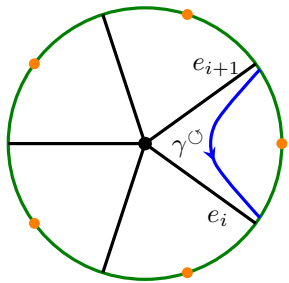
\begin{center}
    \begin{tikzpicture}[decoration={markings, 
	mark= at position 0.6 with {\arrow{stealth}}}, scale=0.6]

            \draw[ao, very thick] (0,0) circle(3);    
            \fill[black] (0,0) circle(0.15);
        \foreach \n in {0,...,4}
        {
			\fill[orange] (360/5*\n:3) circle(0.12);
            \draw[very thick, black] (0,0) -- (360/5*\n+36:3);
		 }
  \draw[color=blue][very thick][postaction={decorate}] plot [smooth] coordinates { (33:3) (360/5*4.72+36:1.6) (360/5*4.28+36:1.6)   (360/5*4+3+36:3)};
  \node (1) at (47:2.3) {$e_{i+1}$};
  \node (2) at (-46:2.3) {$e_{i}$};
  \node () at (0:0.95){$\gamma^\circlearrowleft$};
      \end{tikzpicture}
 \end{center}

The support of the curve encodes the support of the sections in the image of the functor $\mathcal{F}(v\to e_{i})^R$: the restriction to $e_j$ is non-trivial if and only if the curve intersects the $j$-th boundary component of the $m$-gon. The direction of the curve encodes the direction of the so-called transport of the perverse schober along the curve, which simply amounts to the equivalence 
\[ \mathcal{F}^\rightarrow(\gamma^\circlearrowleft)\coloneqq \mathcal{F}(v\to e_{i})\circ\mathcal{F}(v\to e_{i+1})^R\colon \mathcal{F}(e_{i+1})\to \mathcal{F}(e_{i})\,.\]

\noindent {\bf Induction for general surfaces}

The above simple local geometric interpretation generalizes to other induction functors, such as induction from local sections at edges or vertices to global sections. This uses more general trajectories, denoted $\ccwt_e$ or $\ccwt_v$, with respect to a line field on  ${\bf S}\backslash \rgraph_0$ induced by the ribbon graph $\rgraph$, and corresponding transport equivalences. In general, the trajectories will however not be curves: when inducing from a vertex $v$, the trajectory $\ccwt_v$ branches at $v$. We call such trajectories web trajectories, as opposed to curve trajectories. See \Cref{fig:trajectories} for an illustration.

The geometry of a curve or web trajectory controls the construction of the induction functors, which we formulate in terms of a gluing of local sections of the perverse schober following the trajectory. As a consequence, the support of the global section obtained from induction is given by the trajectory. For instance, for the right induction $\indR_v$ from a vertex $v$ to global sections, defined as the right adjoint to the restriction functor $\on{ev}_v$ to $v$, we obtain the following concrete description of its evaluations. 

\begin{introprop}[See \Cref{prop:evv}]\label{introprop:splitting}
Assume that $\rgraph$ has no vertices of valency $1$. Let $v,v'$ be vertices of $\rgraph$ and let $f$ be an edge of $\rgraph$. 
\begin{enumerate}[(1)]
\item There is an equivalence in $\on{Fun}(\mathcal{F}(v),\mathcal{F}(f))$
\[ \on{ev}_f\circ \indR_v\simeq \bigoplus_{h\in H(v),\,c:h\overset{\circlearrowleft}{\shortrightarrow}f}\mathcal{F}^\rightarrow(c)\circ \mathcal{F}(h)
\,,\]
where the sum runs over all halfedges $h$ incident to $v$ and counterclockwise trajectories $c$ from $h$ to $f$. This counts the number of times that the web trajectory $\ccwt_v$ passes by $f$. 
\item There is an equivalence in $\on{Fun}(\mathcal{F}(v),\mathcal{F}(v'))$
\[ \on{ev}_{v'}\circ \indR_v\simeq  \begin{cases} \on{id}_{\mathcal{F}(v)}\oplus \bigoplus_{h,h'\in H(v),c:h\overset{\circlearrowleft}{\shortrightarrow}h',c\not = \ast} \mathcal{F}(h')^L\circ \mathcal{F}^\rightarrow(c)\circ \mathcal{F}(h) & \text{if }v'=v\\ \bigoplus_{h\in H(v),h'\in H(v'),c:h\overset{\circlearrowleft}{\shortrightarrow}h'} \mathcal{F}(h')^R\circ \mathcal{F}^\rightarrow(c)\circ \mathcal{F}(h) & \text{if }v'\not=v\,.\end{cases}\]
The above sums count the number of times that $\ccwt_v$ passes by the vertex $v'$.
\item The counit $\on{ev}_v\indR_v\to \on{id}_{\mathcal{F}(v)}$ is a split projection.
\end{enumerate}
\end{introprop}

A main feature of the above result is that the composite of restriction with induction splits into a direct sum of functors. Further, these functors are themselves composed of induction and restriction functors (in the opposite order), together with transport equivalences between them. To have such splitting results, the restriction to ribbon graphs without $1$-valent vertices is necessary. Given a spherical functor $F\colon \V\to \N$ with right adjoint $G$, the counit $FG\to \on{id}_\N$ in general does not split.

The story is entirely analogous for left adjoint induction instead of right adjoint induction, replacing counterclockwise trajectories with clockwise trajectories. The left and right induction functors in fact differ from each other by composition with certain spherical twist functors, corresponding geometrically to the rotation of the boundary of the marked surface, see \Cref{prop:twist_indL_indR}.

The construction of the induction functors is inspired by constructions of Lagrangians in symplectic manifolds equipped with Lefschetz fibrations to surfaces. Induction should correspond to parallel transport constructions of Lagrangians in the fiber along trajectories. Curve trajectories themselves also form Lagrangians in the surface. For the partially wrapped Fukaya category of the surface, the induction functors of the corresponding perverse schobers produce the objects given by the curve trajectories.

\subsection{Gluing cluster tilting subcategories}\label{introsec:gluingCT}

We fix a marked surface ${\bf S}$ with a spanning ribbon graph $\rgraph$ without $1$-valent vertices and a $\rgraph$-parametrized perverse schober $\mathcal{F}$.\\ 

\noindent {\bf $\infty$-categorical  Frobenius exact structures from perverse schobers}

Consider a vertex $v$ of $\rgraph$ of valency $m\geq 2$ with incident edges $e_1,\dots, e_m$. The stalk of $\mathcal{F}$ at $v$ is the stable $\infty$-category $\mathcal{F}(v)$. It comes with $m$ functors $\mathcal{F}(v\to e_i)\colon \mathcal{F}(v)\to \mathcal{F}(e_i)$. These induce an $\infty$-categorical exact structure on $\mathcal{F}(v)$, where a fiber and cofiber sequence is declared to be exact if its image under each of the functors $\mathcal{F}(v\to e_i)$ is split exact. By general principles, the injective objects with respect to this exact structure are given by the images of objects in any stalk $\mathcal{F}(e_i)$ under the fully faithful left adjoint $\mathcal{F}(v\to e_i)^L\colon \mathcal{F}(e_i)\to \mathcal{F}(v)$ of $\mathcal{F}(v\to e_i)$. Dually, the projective objects are given by the images of objects in any $\mathcal{F}(e_i)$ under the fully faithful right adjoint $\mathcal{F}(v\to e_i)^R$. The axioms of perverse schobers on the $n$-spider imply that these injective and projectives objects coincide, with which one can show that $\mathcal{F}(v)$ is Frobenius exact.

The Frobenius exact structure on the $\infty$-category of global section $\mathcal{H}(\rgraph,\mathcal{F})$ arises in a similar way, using the restriction functors to the external edges of $\rgraph$. We consider cluster tilting subcategories always with respect to these $\infty$-categorical exact structures.\\

\noindent {\bf Gluing rigid objects}

Proving that an additive subcategory is cluster tilting involves verifying two properties: Firstly, the subcategory is rigid, meaning there are no exact extensions between objects. Secondly, every object in the ambient exact $\infty$-category admits an exact $2$-term resolution and coresolution by objects in the subcategory. Both of these properties glue separately along perverse schobers, see \Cref{prop:glue2termres,prop:rigidglue}. The gluing of $2$-term resolutions along induction is based on a rather formal argument about limits. Checking that rigidity is preserved under gluing relies on the geometric description of the induction functors and the results of \Cref{introprop:splitting}. An intuitive reason that rigidity is preserved under induction is as follows: Trajectories cannot cross, since the flow is injective. Typically, the exact extensions between objects count the number of crossings between the corresponding curves or webs. Hence, without crossings, rigidity is satisfied. See \Cref{fig:induction} for an example of a collection of non-crossing trajectories corresponding to a cluster tilting object.\\

\noindent {\bf Ice quivers and amalgamation}

Any cluster tilting subcategory $\mathcal{T}_v\subset \mathcal{F}(v)$ must contain all injective-projective objects, and hence the full subcategory $\mathcal{F}(e_i)\subset \mathcal{T}_v$ for all $1\leq i\leq m$. In the case that $\mathcal{F}(v)$ contains a cluster tilting objects $T_v$ with finite dimensional discrete endomorphism algebra, one can extract an ice quiver from this algebra. Under mild assumptions, see \Cref{subsec:amalgamation}, this ice quiver has $m$ frozen components. Each component corresponds to injective-projective objects arising from the summands of an additive generator of $\mathcal{F}(e_i)$. Given a collection of cluster tilting objects $T_v\subset \mathcal{F}(v)$, we thus obtain a collection of ice quivers whose frozen components can be paired according to the incidence of the edges of $\rgraph$ with these vertices. This allows to consider the amalgamation of these ice quivers. Amalgamation was introduced in \cite{FG06a}, and roughly speaking corresponds to the process of gluing ice quivers along their frozen parts, and then unfreezing the frozen part along which was glued. In \Cref{subsec:amalgamation}, we relate the gluing of cluster tilting objects via \Cref{introthm:1} to the amalgamation of their corresponding ice quivers.\\

\noindent {\bf Further remarks}

The gluing of cluster tilting subcategories has previously been studied in certain settings with semiorthogonal decompositions/recollements, see for instance \cite{Vas21,Vas20}. In the examples related with higher Teichm\"uller theory, the local cluster tilting subcategories lie in the exact $\infty$-category $\on{Fun}(\Delta^1,\C_\ADE)$, which describes derived representations of the $A_2$-quiver in the $1$-cluster category $\C_\ADE$ of the Dynkin quiver $\ADE$. Note that $\on{Fun}(\Delta^1,\C_\ADE)$ has a recollement into two copies of $\C_\ADE$. The existence of these cluster tilting subcategories is nevertheless non-trivial, see Keller--Liu \cite{KL25}, and, to the authors knowledge, does not follow from previous results. It would be understand these cluster tilting subcategories conceptually via this recollement, and to determine for which other stable $\infty$-categories $\D$ the exact $\infty$-category $\on{Fun}(\Delta^1,\D)$ admits cluster tilting objects.

The cluster categories of \cite{Ami09} are triangulated $2$-Calabi--Yau and the Higgs categories of \cite{Wu21} are extriangulated $2$-Calabi--Yau. In the setting of perverse schobers, one can obtain relative right Calabi--Yau structures on the $k$-linear stable $\infty$-categories of global sections using the gluing of relative Calabi--Yau structures, see \cite{BD19,Chr23}. These give rise to extriangulated $2$-Calabi--Yau properties, see \cite{Chr22b}. However, no Calabi--Yau conditions are necessary for \Cref{introthm:1}.

\newpage

\subsection{Examples}\label{introsec:examples}

\noindent{\bf $1$-periodic topological Fukaya categories}

The $1$-periodic topological Fukaya category of the marked surface arises as the global sections of a perverse schober with generic stalk given by the perfect derived category of $1$-periodic chain complexes $\D^{\on{perf}}(k[t_1^\pm])$ and no singularities. In \cite{Chr22b}, it is shown that the cluster tilting objects in this exact $\infty$-category are in bijection with ideal triangulations of the marked surface. It is further shown that the $1$-periodic topological Fukaya category is equivalent, as an exact $\infty$-category, to the Higgs category \cite{Wu21} arising from the relative Ginzburg algebras of \cite{Chr22}. 

\Cref{introthm:1} constructs all these cluster tilting objects: Every ideal triangulation is dual to a spanning ribbon graph with only $3$-valent vertices. The local exact $\infty$-category describing the perverse schober at the $3$-valent vertices is given by $\on{Fun}(\Delta^1,\D^{\on{perf}}(k[t_1^\pm])$ and admits a unique cluster tilting object with endomorphism ice quiver as follows: 
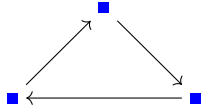
\begin{center}
\begin{tikzpicture}[scale=1.2]
\draw node[frvertex] (1) at (0,1) {};
\draw node[frvertex] (2) at (-1,0) {};
\draw node[frvertex] (3) at (1,0) {};

\draw[->] (2)--(1);
\draw[->] (3)--(2);
\draw[->] (1)--(3); 
\end{tikzpicture}
\end{center}
Gluing yields the corresponding cluster tilting object corresponding to the triangulation.\\

\noindent {\bf Topological Fukaya categories for higher Teichm\"uller theory}

The $1$-periodic perfect derived category $\D^{\on{perf}}(k[t_1^\pm])$ is equivalent to the cosingularity category $\D^{\on{perf}}(\Pi_2(A_1))/\D^{\on{fin}}(\Pi_2(A_1))$ of the $2$-Calabi-Yau completion $\Pi_2(A_1)$ of the  $A_1$-quiver in the sense of \cite{Kel11}. Replacing the $A_1$-quiver by a different ADE quiver $I$, we obtain a $2$-periodic stable $\infty$-category $\C_I\coloneqq \D^{\on{perf}}(\Pi_2(I))/\D^{\on{fin}}(\Pi_2(I))$ with finitely many equivalence classes of indecomposable objects. Note that $\C_I$ also describes the $1$-Calabi--Yau cluster category of type $I$, as well as the category of matrix factorizations of a type $I$ surface singularity.

We can obtain the $\C_I$-valued topological Fukaya category of a marked surface ${\bf S}$ in the sense of \cite{DK18} as the global sections of a perverse schober without singularities and generic stalk $\C_I$ and trivial monodromy, see \cite{Chr23}.

In the case of ${\bf S}=\Delta$ a $3$-gon, this topological Fukaya category is equivalent to $\on{Fun}(\Delta^1,\C_I)$. The work of Keller--Liu \cite{KL25} shows that $\on{Fun}(\Delta^1,\C_I)$ is equivalent to a canonical Higgs category. As a consequence $\on{Fun}(\Delta^1,\C_I)$ carries cluster tilting objects. Keller--Liu obtain a different cluster tilting object for each orientation of the Dynkin diagram $\ADE$. Conjecturally (and manifestly in type $A_n$), the ice quivers of these cluster tilting objects recover ice quivers underlying cluster algebras arising from the type $\ADE$ higher Teichm\"uller theory of a triangle in the sense of \cite{FG06,Le19,GS19}. Applying \Cref{introthm:1}, we obtain cluster tilting objects in the $\C_I$-valued topological Fukaya categories for arbitrary marked surfaces. The corresponding ice quivers arise via amalgamation and thus, in general conjecturally, coincide with ice quivers used for the cluster algebras arising from the higher Teichm\"uller theory of arbitrary surfaces in type $\ADE$. This class of examples, and its relation to the Higgs categories and cosingularity categories of $3$-Calabi--Yau categories, will be discussed in more detail in \cite{Chr25b}.\\

\noindent {\bf Marked surfaces with punctures}

A marked surface with punctures refers to a marked surface in which marked points in the interior (the punctures) are allowed. For every marked surface with punctures, there is an associated cluster algebra, whose clusters are in bijection with tagged triangulations of the surface \cite{FST08}. These generalize the cluster algebras associated with marked surfaces without punctures. In \Cref{subsec:puncturedsurf}, we discuss how to obtain additive categorifications of these cluster algebras, extending results of \cite{Chr22b} on marked surfaces without punctures. For this, we add singularities at the punctures to the non-singular perverse schobers from \cite{Chr22b}, the corresponding spherical adjunction is between the $1$-periodic and the $2$-periodic derived $\infty$-category. 

Most tagged triangulations give rise to a ice quivers arising from amalgamation, used for the cluster seeds. We construct corresponding cluster tilting objects in the $\infty$-category of global sections in \Cref{prop:puncturedsurfaceCTO}. We expect the Higgs categories associated with punctured marked surface to appear as a special case of this construction.

\subsection{Outlook}\label{introsec:outlook}

Finally, we list some open problems and topics for further research.

The full understanding of the cluster algebras of surface allowed to understand the corresponding cluster tilting theory of cluster categories of surfaces, in particular classifying the cluster tilting objects in terms of ideal triangulations. Based on this, one can understand the silting theory, or equivalently tilting theory of simple minded collections, in corresponding $3$-Calabi--Yau categories, from which the $2$-Calabi--Yau cluster category arises by passage to the cosingularity category of \cite{Ami11}. Based on this relation, \cite{KQ20} proved an identification of the space of Bridgeland stability condition on the proper $3$-Calabi--Yau category with framed quadratic differentials. This shows that cluster tilting theory can be used as a bridge between tilting and silting theory as well as the combinatorics of cluster algebras. The results of this paper construct, and interpret geometrically, cluster tilting objects in 'higher rank' generalizations of the surface cluster categories. The study of the corresponding cluster algebras remains an active field of research, see for instance \cite{IY23}. It would be very interesting to use these cluster tilting objects to transfer results between the tilting theory and the study of these cluster algebras.

In this paper, we treat marked surfaces with punctures only in type $A_1$ (meaning the corresponding Lie group is $\on{SL}_2$). A solution for the other Dynkin types remains an intriguing open problem. As input for higher Teichm\"uller theory can more generally serve any punctured marked surface with empty boundary. It would be interesting to extend the results of this paper to surfaces without boundary. 

Besides higher Teichm\"uller theory, an interesting class of cluster algebras which admit seeds arising from amalgamation arise from double Bott-Samelson varieties \cite{SW21}. Further cluster seeds arising from amalgamation associated with general marked surfaces together with a choice of Dynkin type and boundary decorations by braid words appear in \cite[Section 8]{GK21}. We hope to return to the question of the gluing of the Higgs categories and corresponding cluster tilting objects in this context in future work. An additional challenge may be that one cannot expect the $\infty$-categories underlying the corresponding Higgs categories to be stable in general, though they still carry $\infty$-categorical exact structures.

\subsection*{Acknowledgements}

I thank Bernhard Keller for many helpful discussions and Gustavo Jasso for helpful comments. The author is a member of the Hausdorff Center for Mathematics at the University of Bonn (DFG GZ 2047/1, project ID 390685813). This project has received funding from the European Union’s Horizon 2020 research and innovation programme under the Marie Skłodowska-Curie grant agreement No 101034255.

\section{Higher categorical preliminaries}

We freely use the language of stable $\infty$-categories, as developed in \cite{HTT,HA}. We refer to \cite{Ker,Cis} for introductory treatments of the theory of $\infty$-categories.

We denote the $\infty$-category of (small) $\infty$-categories by $\on{Cat}_\infty$. We denote the $\infty$-category of (small) stable $\infty$-categories and exact functors by $\on{St}$. In most of the paper, the perverse schobers will take values in $\on{St}$.  Given a stable $\infty$-category $\C$, we denote by \[ \on{Ext}^i_\C(\mhyphen,\mhyphen)\coloneqq \pi_0 \on{Map}_{\C}(\mhyphen,\mhyphen[n])\colon \C^{\on{op}}\times \C\to \on{Ab}\]
the $i$-th extension group functor. 

Let $k$ be a field. The $\infty$-category of $k$-linear (small) stable $\infty$-categories is defined as $\on{LinCat}_k^{\on{St}}\coloneqq \on{Mod}_{\on{Mod}_k^{\on{c}}}(\on{St})$, where $\on{Mod}_k^{\on{c}}\simeq \D^{\on{perf}}(k)$ is the symmetric monoidal stable $\infty$-category of compact $k$-modules.

\subsection{Limits of \texorpdfstring{$\infty$}{infinity}-categories via sections of the Grothendieck construction}\label{subsec:Grothendieckconstr}

The forgetful functors $\on{LinCat}_k^{\on{St}}\to\on{St}\to \on{Cat}_\infty$ preserve and reflect limits. 
In the following we explain how to describe limits in $\on{Cat}_\infty$.

 We will not distinguish in notation between $1$-categories and their nerve in $\on{Set}_\Delta$. Let $Z$ be a $1$-category. A diagram $Z\to \on{Cat}_\infty$ is said to be strictly commuting if it arises from a functor of $1$-categories $Z\to \on{Set}_\Delta$ by composition with the localization functor $\on{Set}_\Delta \to \on{Cat}_\infty$. Limits of strictly commuting diagrams can be explicitly described as follows.

Let $f\colon Z\to \on{Set}_\Delta$ be a diagram taking values in $\infty$-categories. The Grothendieck construction of $f$ refers to the coCartesian fibration $p\colon \Gamma(f)\to Z$ defined in \cite[\href{https://kerodon.net/tag/025X}{Def.~025X}]{Ker}. Here, $\Gamma(f)$ is an explicitly defined simplicial set. This coCartesian fibration is classified by the functor $F\colon Z\xrightarrow{f} \on{Set}_\Delta\to \on{Cat}_\infty$ arising from $f$. 
The objects of $\Gamma(f)$ (i.e.~$0$-simplicies) are given by pairs of objects $(z,X)$ with $z\in Z$ and $X\in f(z)$. The fibration $p$ maps the pair $(z,X)$ to $z\in Z$. A morphism (i.e.~$1$-simplex) $(z,X)\to (z',X')$ in $\Gamma(f)$ amounts to a morphism $\alpha\colon z\to z'$ in $Z$ together with a morphism $f(\alpha)(X)\to X'$ in the $\infty$-category $f(z')$. 

The $\infty$-category of sections of the Grothendieck construction $p\colon \Gamma(f)\to Z$ is defined as the pullback $\on{Fun}(Z,\Gamma(f))\times_{\on{Fun}(Z,Z)}\{\on{id}_Z\}$. This $\infty$-category of sections also describes the lax limit of $F$ in the $(\infty,2)$-category of $\infty$-categories.

A section $s:Z\to \Gamma(f)$ of $p$ is called coCartesian if for every morphism $\alpha\colon z\to z'$ in $Z$ the morphism $s(\alpha)\colon (z,X)\to (z',X')$ is $p$-coCartesian in the sense of \cite[\href{https://kerodon.net/tag/01T5}{Tag 01T5}]{Ker}. Unraveling the definition, one finds that an edge is coCartesian if and only if it describes an equivalence $f(\alpha)(X)\simeq X'$. The full subcategory of the $\infty$-category of sections of $p$ consisting of coCartesian sections describes the the limit of the functor $F\colon Z\to \on{Cat}_\infty$, see \cite[\href{https://kerodon.net/tag/05RX}{Prop.~05RX}]{Ker}. 

Suppose that $f\colon Z\to \on{Set}_\Delta$ takes values in $\infty$-categories admitting finite limits and colimits and finite limits and finite colimits preserving functors.
Then finite limits and colimits in the $\infty$-category of coCartesian sections of $\Gamma(f)$ exist and are computed pointwise in $Z$, meaning that the evaluation functor $\on{Fun}(Z,\Gamma(f))\times_{\on{Fun}(Z,Z)}\{\on{id}_Z\}\to f(z)$, for any $z\in Z$, preserves finite limits and colimits and furthermore together these functor reflect finite limits and colimits. This follows from \cite[5.1.2.3, 4.3.1.10, 4.3.1.16]{HTT}. 

\subsection{Exact \texorpdfstring{$\infty$}{infinity}-categories and cluster tilting subcategories}

We recall some aspects of the theory of exact $\infty$-categories. 

\begin{definition}[$\!\!$\cite{Bar15}]
An exact $\infty$-category is a triple $(\mathcal{C},\mathcal{C}_{\dagger},\mathcal{C}^{\dagger})$, where $\mathcal{C}$ is an additive $\infty$-category and $\mathcal{C}_{\dagger},\mathcal{C}^{\dagger}\subset \mathcal{C}$ are subcategories (called subcategories of inflations and deflations), satisfying that 
\begin{enumerate}[(1)]
\item every morphism $0\rightarrow X$ in $\mathcal{C}$ lies in $\mathcal{C}_{\dagger}$ and every morphism $X\rightarrow 0$ in $\mathcal{C}$ lies in $\mathcal{C}^\dagger$.
\item pushouts in $\mathcal{C}$ along morphisms in $\mathcal{C}_{\dagger}$ exist and lie in $\mathcal{C}_{\dagger}$. Dually, pullbacks in $\mathcal{C}$ along morphisms in $\mathcal{C}^{\dagger}$ exist and lie in $\mathcal{C}^{\dagger}$. 
\item  Given a commutative square in $\mathcal{C}$ of the form
\[
\begin{tikzcd}
X \arrow[r, "a"] \arrow[d, "b"] & Y \arrow[d, "c"] \\
X' \arrow[r, "d"]               & Y'              
\end{tikzcd}
\]
the following are equivalent.
\begin{itemize}
\item The square is pullback, $c\in \mathcal{C}_{\dagger}$ and $d\in \mathcal{C}^{\dagger}$.
\item The square is pushout, $b\in \mathcal{C}_{\dagger}$ and $a\in \mathcal{C}^{\dagger}$.
\end{itemize}
\end{enumerate}
We typically abuse notation and simply refer to $\mathcal{C}$ as the exact $\infty$-category.
\end{definition}

\begin{definition}
An exact sequence $X\rightarrow Y\rightarrow Z$ in an exact $\infty$-category $\C$ consists of a fiber and cofiber sequence in $\mathcal{C}$
\[
\begin{tikzcd}
X \arrow[r, "a"] \arrow[dr, "\square", phantom] \arrow[d] & Y \arrow[d, "b"] \\
0 \arrow[r]                & Z               
\end{tikzcd}
\]
with $a$ an inflation and $b$ a deflation. 

A functor between exact $\infty$-categories is called exact if it maps exact sequences to exact sequences.
\end{definition}

Given an exact $\infty$-category $\C$ and $X,Y\in\C$, the exact extensions $X\to Y[1]$ form a subgroup of the abelian group $\on{Ext}^1_\C(X,Y)$, which we will denote by $\on{Ext}^{1,\on{ex}}_\C(X,Y)$. 

\begin{remark}
Every exact $\infty$-category arises as an extension closed subcategory of a stable $\infty$-category \cite{Kle20}. Thus, if $\C$ is equipped with the structure of an exact $\infty$-category, its homotopy $1$-category $\on{ho}\C$ inherits the structure of an extriangulated category \cite{Kle20,NP20}. The exact extensions $\mathbb{E}_{\on{ho}\C}$ in the extriangulated structure parametrize exact sequences in $\C$. 
\end{remark}

\begin{definition}\label{def:frob}
Let $\C$ be an exact $\infty$-category.
\begin{enumerate}[1)]
\item An object $P\in \mathcal{C}$ is called projective if every exact sequence $X\rightarrow Y \rightarrow P$ splits. An object $I\in \mathcal{C}$ is called injective if every exact sequence $I\rightarrow Y \rightarrow Z$ splits.
\item We say that $\mathcal{C}$ has enough projectives if for each object $X\in \mathcal{C}$ there exists an exact sequence $X\rightarrow P\rightarrow Y$ with $P$ projective. Similarly, we say that $\mathcal{C}$ has enough injectives if for each object $Y\in \mathcal{C}$ there exists an exact sequence $Y\rightarrow I \rightarrow X$ with $I$ injective.
\item We say that $\mathcal{C}$ is Frobenius if $\mathcal{C}$ has enough projectives and injectives and the classes of projective and injective objects coincide. 
\end{enumerate}
\end{definition}

\begin{example}\label{ex:inducedexstr}
Let $F\colon \C\to \D$ be an exact functor (in the stable sense) between stable $\infty$-categories. Then there exists an exact structure on $\C$ where a fiber and cofiber sequence in $\mathcal{C}$ is exact if and only if its image under $G$ splits. We will refer to it as the exact structure on $\C$ induced by $F$.
\end{example}

\begin{proposition}[$\!\!$\cite{Chr22b}]\label{prop:Frob2CYexact}
Let $F\colon \C\to \D$ be a spherical functor between stable $\infty$-categories. Then the exact structure on $\C$ induced by $F$ is Frobenius. The injective-projective objects are given by the essential image of the right adjoint $G$ of $F$.
\end{proposition}
 
A subcategory $\T\subset \C$ of an additive $\infty$-category $\C$ is called an additive subcategory if it is closed under finite direct sums and direct summands. 

\begin{definition}\label{def:clustertilting}
Let $\T$ be an additive subcategory of an exact $\infty$-category $\C$.
\begin{enumerate}[(1)]
\item We call $\T$ rigid if $\on{Ext}_\C^{1,\on{ex}}(T,T')\simeq 0$ for all $T,T'\in \T$. 
\item We say that $\T$ has the right $2$-term resolution property if for every object $X\in \C$ there exists an exact sequence $X\to T_0\to T_1$ in $\C$ with $T_0,T_1\in \T$.
\item We say that $\T$ has the left $2$-term resolution property if for every object $X\in \C$ there exists an exact sequence $T_0\to T_1\to X$ in $\C$ with $T_0,T_1\in \T$.
\item We say that $\T$ has the two-sided $2$-term resolution property if it has the left and the right $2$-term resolution property.
\item We call $\T$ a cluster tilting subcategory if $T$ is rigid and has the two-sided $2$-term resolution property. 
\end{enumerate}
\end{definition}

\begin{definition}
Let $\C$ be an exact $\infty$-category. Let $T\in \C$ be an object and denote by $\T=\on{Add}(T)$ its additive closure, meaning the smallest additive subcategory containing $T$. 
\begin{enumerate}[(1)]
\item We call $T$ basic if $T$ is the direct sum of finitely many indecomposable objects which are pairwise non-isomorphic.
\item We call $T$ a cluster tilting object if $T$ is basic and $\T\subset \C$ is cluster tilting. 
\end{enumerate}
\end{definition}

\begin{remark}
In the literature, cluster tilting subcategories of $\C$ are often defined as additive subcategories $\T$ satisfying that for any object $X\in \C$, $X$ lies in $\T$ if and only if $\on{Ext}^{1,\on{ex}}_{\C}(T,X)\simeq \on{Ext}^{\on{ex},1}_{\C}(X,T)\simeq 0$ for all $T\in \T$. In typical situations, this definition coincides with \Cref{def:clustertilting}, see \cite[Lem.~5.16]{Chr22b}.
\end{remark}

\begin{remark}\label{rem:stablecategory}
Given a Frobenius exact $\infty$-category $\D$, there is a corresponding (so-called) stable category (which is a stable $\infty$-category) $\underline{D}$, obtained by a localization removing the injective--projective objects. We refer to \cite{Chr22} for details. The homotopy category of $\underline{\D}$ can be identified with the stable category of the extriangulated homotopy category of $\D$, in the sense of \cite[Prop.~3.30, Cor.~7.4]{NP19}. For $X,Y\in \underline{\D}$, the abelian group $\on{Ext}^0_{\underline{\D}}(X,Y)$ is thus given by the quotient of $\on{Ext}^0_\D(X,Y)$ by the subgroup of morphisms which factor through an injective-projective. Further, the abelian group $\on{Ext}^1_{\underline{\D}}(X,Y)$ thus agrees with the abelian group of exact extensions between $X$ and $Y$ in $\D$.

As an immediate consequence of this, we see that $\T\subset \D$ is a cluster tilting subcategory if and only if it contains all injective-projective objects and its image in $\underline{\D}$ is a cluster tilting subcategory. If the discrete endomorphism algebra of a cluster tilting object $T$ is described by an ice quiver, then the discrete endomorphism algebra of the image of $T$ in $\underline{\D}$ is described by the arising non-ice quiver obtained by removing the frozen vertices.
\end{remark}

\section{Perverse schobers on marked surfaces}

\subsection{Surfaces, ribbon graphs and induced trajectories}\label{subsec:surfaces}

\begin{definition}\label{def:surf}
\begin{enumerate}[(1)]
\item By a surface, we will mean a smooth, oriented surface ${\bf S}$ with non-empty boundary $\partial {\bf S}$ and interior ${\bf S}^\circ$. We will also assume that ${\bf S}$ is connected and compact, unless stated otherwise. Note that if ${\bf S}$ is compact, the boundary $\partial {\bf S}$ consists of a disjoint union of circles.
\item A marked surface $({\bf S},M)$ consists of a surface ${\bf S}$ and a non-empty finite set $M\subset \partial {\bf S}$ of marked points, lying on the boundary of ${\bf S}$. We also require that each connected component of $\partial {\bf S}$ contains at least one marked point. We typically write ${\bf S}$ for $({\bf S},M)$.
\end{enumerate}
\end{definition}

\begin{definition}~
\begin{enumerate}[(1)]
\item A graph ${\rgraph}$ consists of two finite sets ${\rgraph}_0$ of vertices and $\on{H}_{\rgraph}$ of halfedges together with an involution $\tau\colon \on{H}_{\rgraph} \rightarrow \on{H}_{\rgraph}$ and a map $\sigma\colon \on{H}_{\rgraph}\rightarrow \rgraph_0$.
\item Let $\rgraph$ be a graph. An edge of $\rgraph$ is defined to be an orbit of $\tau$, we denote the set of edges by $\rgraph_1$. An edge is called internal if the orbit contains two elements and called external if the orbit contains a single element. 
We denote the set of external edges of $\rgraph$ by $\rgraph^\partial_1$. If $h$ is a halfedge, we denote by $e(h)\in \rgraph_1$ the edge which contains it.
\item A ribbon graph consists of a graph $\rgraph$ together with a choice of a cyclic order on the set $\on{H}_\rgraph(v)$ of halfedges incident to each vertex $v$.
\end{enumerate}
We typically assume graphs to be connected.
\end{definition}

\begin{definition}
Let $\rgraph$ be a graph. We define the exit path category $\on{Exit}(\rgraph)$ of $\rgraph$ to be the nerve of the $1$-category with
\begin{itemize}
\item objects the vertices and edges of $\rgraph$ and
\item the non-identity morphism are given by halfedges $h\colon v\rightarrow e$, where $v$ is a vertex incident to $h$ and $e$ is the edge containing $h$.
\end{itemize}
The geometric realization $|\rgraph|$ of $\rgraph$ is defined as the geometric realization $|\on{Exit}(\rgraph)|$ of $\on{Exit}(\rgraph)$. 
\end{definition}

\begin{remark}
A graph $\rgraph$ whose geometric realization $|\rgraph|$ is embedded into an oriented surface ${\bf S}$ inherits a canonical ribbon graph structure by requiring the halfedges at any vertex to be ordered in the counterclockwise direction.
\end{remark}

\begin{definition}\label{def:sg}
Let ${\bf S}$ be a marked surface. A spanning graph for ${\bf S}$ consists of a graph $\rgraph$ together with an embedding $i:|\rgraph|\subset {\bf S}\backslash M$ 
satisfying that 
\begin{itemize}
\item $i$ is a homotopy equivalence,
\item $i$ restricts to a homotopy equivalence $\partial |\rgraph|\rightarrow \partial {\bf S}\backslash M$.
\end{itemize}
We consider a spanning graph of ${\bf S}$ as endowed with the canonical ribbon graph structure arising from the embedding into ${\bf S}$.

Note that, since we assumed that every boundary component of ${\bf S}$ contains a marked point, spanning graphs of ${\bf S}$ cannot have loops, meaning edges whose two halfedges lie at the same vertex.
\end{definition}

\begin{definition}
Let $\Sigma$ be a possibly non-compact surface. A line field $\nu$ on $\Sigma$ is a section of the projectivized tangent bundle $\mathbb{P}T\Sigma$.
\end{definition}

A spanning graph $\rgraph$ of a marked surface ${\bf S}$ determines a canonical homotopy class of a line field on ${\bf S}\backslash \rgraph_0$ with the properties that the edges of $\rgraph$ define flow-lines, there are no other flow-lines asymptotic to the vertices of $\rgraph$, and whose winding number around an $n$-valent vertex is given by $-m$. We denote a choice of this line field by $\nu_\rgraph$. See \Cref{fig:exlinefield} for an illustration.

\begin{figure}[ht!]
\begin{center}
    \begin{tikzpicture}

            \draw[ao, very thick] (0,0) circle(3);    
            \fill[black] (0,0) circle(0.1);
        \foreach \n in {0,...,4}
        {
            \foreach \m in {0,0.5,1,1.5,2,2.5}
            {	
            	\coordinate (x) at (360/5*\n:3);
            	\coordinate (y) at (360/5*\n+360/5:3);
            	\coordinate (v) at (360/5*\n+4+6*\m:3);
            	\coordinate (w) at (360/5*\n+360/5-4-6*\m:3);
                \coordinate (u) at (360/5*\n+36:\m+0.15);
                \draw[gray] (v) .. controls (u) and (u) .. (w);
            }
            \coordinate (u) at (360/5*\n+36:3.3);
			\fill[orange] (360/5*\n+36:3) circle(0.1);
            \draw[very thick, black] (0,0) -- (360/5*\n:3);
		 }

    \end{tikzpicture}
\caption{The line field $\nu_\rgraph$ for a spanning graph $\rgraph$ of the $5$-gon.}\label{fig:exlinefield}\end{center}
\end{figure}

\begin{definition}
Let $\Sigma$ be a possibly non-compact surface equipped with a line field $\nu$. We call a curve $\gamma\colon U\to \Sigma$, with $U\subset [0,1]$, a trajectory if $\dot{\gamma}_t=\nu_{\gamma(t)}\in \mathbb{P}T\Sigma$ for all $t\in U$. 
\end{definition}

We fix a marked surface ${\bf S}$ with spanning graph $\rgraph$. We next define clockwise and counterclockwise trajectories in the non-compact surface ${\bf S}\backslash \rgraph_0$ equipped with the line field $\nu_\rgraph$. For this, we first discuss some conventions on directions of halfedges.

Each halfedge $h$ of an internal edge $e$ determines an orientation of the edge $e$, pointing in the direction away from the vertex of $h$. This fits with how we consider halfedges as exiting a vertices towards its edge.

Recall that an external edge $e$ has a unique halfedge $h$. Opposite to the case of internal edges, we say that the direction of $e$ induced by $h$ is the direction pointing towards the vertex of $h$, meaning inwards.

\begin{definition}\label{def:cwtrajectory}
Let ${\bf S}$ be a marked surface with a spanning graph $\rgraph$. Let $h$ be a halfedge of $\rgraph$, lying at a vertex $v$, and which is part of the edge $e$. We consider $e$ as oriented by $h$ as above. 
\begin{itemize}
\item Choose any point $p$ close to $e$ which lies on the right of the directed edge $e$. We define the clockwise halfedge trajectory $\cwt_h\colon [0,1]\to {\bf S}\backslash \rgraph_0$ of $h$ as any trajectory satisfying $\cwt_h(0)=p$,  $\cwt_h(1)\in \partial {\bf S}$ and such that $\cwt_h$ points at $p$ in the direction of $e$.
\item Choose a point $q$ close to $e$ which lies on the left of $e$. We define the counterclockwise halfedge trajectory $\ccwt_h\colon [0,1]\to {\bf S}\backslash \rgraph_0$ of $h$ as any trajectory satisfying $\ccwt_h(0)=q$,  $\ccwt_h(1)\in \partial {\bf S}$ and pointing at $q$ in the direction of $e$.
\end{itemize}
\end{definition}

We note that clockwise trajectories turn right at every vertex of the ribbon graph, whereas counterclockwise trajectories turn left.

\begin{definition}
Let ${\bf S}$ be a marked surface with a spanning graph $\rgraph$. Consider the topological space 
\[ U_n=([0,1]\times \{1,\dots,n\})/((0,1)\sim (0,i))_{2\leq i\leq n}\,,\] i.e.~an $n$-spider.
\begin{enumerate}[(1)]
\item Let $e$ be an external edge of $\rgraph$. We define the clockwise trajectory $\cwt_e\colon [0,1]\to {\bf S}$ as $\cwt_h$ with $h\in e$, choosing $\cwt_h$ such that $\cwt_e(0)\in \partial {\bf S}$. We define the counterclockwise trajectory $\ccwt_v$ similarly. 
\item Let $e$ be an internal edge of $\rgraph$. We define the clockwise trajectory $\cwt_e\colon U_2\to {\bf S}$ of $e$ as the composite of the two clockwise halfedge trajectories $\cwt_{h},\cwt_{h'}$ with $h,h'\in e$ (after perturbing their starting points to that they coincide and lie on $e$, instead of next to $e$).  We define the counterclockwise trajectory $\ccwt_e\colon U_2\to {\bf S}$ similarly. 
\end{enumerate}
We also call $\cwt_e,\ccwt_e$ curve trajectories.
\begin{enumerate}
\item[(3)] Let $v$ be an $n$-valent vertex of $\rgraph$ with incident halfedges $h_1,\dots,h_n$.  We define the clockwise web trajectory $\cwt_v\colon U_n\to {\bf S}$ as the composite of the $n$ trajectories $\cwt_{h_i}$ with a small $n$-spider emanating from $v$. We define the counterclockwise web trajectory $\ccwt_v\colon U_n\to {\bf S}$ similarly.
\end{enumerate}
\end{definition}

\begin{figure}[ht]
\begin{center}
\begin{tikzpicture}[scale=1.2]
    
            \draw[ao, very thick] (-4,-2)--(2,-2)--(2,2)--(-4,2)--(-4,-2); 
            \draw[very thick] (-2,0)--(0,0) (-2,-2)--(-2,0)--(-2,2) (-2,0)--(-4,0) (1,-2)--(0,0)--(1,2) ;
            \fill[black] (-2,0) circle(0.1);
            \fill[blue] (0,0) circle(0.1);
            \fill[orange] (-4,-2) circle(0.1);
            \fill[orange] (2,0) circle(0.1);
            \fill[orange] (-4,2) circle(0.1);
            \fill[orange] (-0.5,-2) circle(0.1);
            \fill[orange] (-0.5,2) circle(0.1);
\foreach \n in {1,2,3,4}{
\draw[gray] (-2-\n*0.3,2) .. controls (-2-\n*0.3,\n*0.3) .. (-4,\n*0.3);
\draw[gray] (-2-\n*0.3,-2) .. controls (-2-\n*0.3,-\n*0.3) .. (-4,-\n*0.3);
        }     
        
\foreach \n in {1,2,3}{
\draw[gray] (1+\n*0.3,2) .. controls (\n*0.3,0) .. (1+\n*0.3,-2);
}

\foreach \n in {1,2,3}{
\draw[gray] plot[smooth] coordinates {(-1.9+\n*0.3,2) (-1.7+\n*0.25,\n*0.3+0.1)  (-\n*0.15,\n*0.3+0.1) (1-\n*0.3,2)};
}

\foreach \n in {1,2,3}{
\draw[gray] plot[smooth] coordinates {(-1.9+\n*0.3,-2) (-1.7+\n*0.25,-\n*0.3+-0.1)  (-\n*0.15,-\n*0.3-0.1) (1-\n*0.3,-2)};
}

\draw[blue, very thick] plot[smooth] coordinates {(-1.7,2) (-1.55,0.4) (-0.5,0.14) (0,0)};
\draw[blue, very thick] (0,0) .. controls (-0.1,-0.15) .. (0.85,-2);
\draw[blue, very thick] (0,0) .. controls (0.1,-0.15) .. (1.15,2);
\draw[Turquoise, very thick] plot[smooth] coordinates {(-1.85,2) (-1.7,0.3) (-0.6,-0.2) (-0.05,-0.5) (0.7,-2)};

\node () at (0.4,-0.2){$\cwt_v$};
\node () at (-1.6,-0.3){$\cwt_e$};
    \end{tikzpicture}
\caption{The clockwise web trajectory $\cwt_v$ starting at a trivalent vertex $v$ and the clockwise curve trajectory $\cwt_e$ starting at an incident edge $e$.} \label{fig:trajectories}
\end{center}
\end{figure}
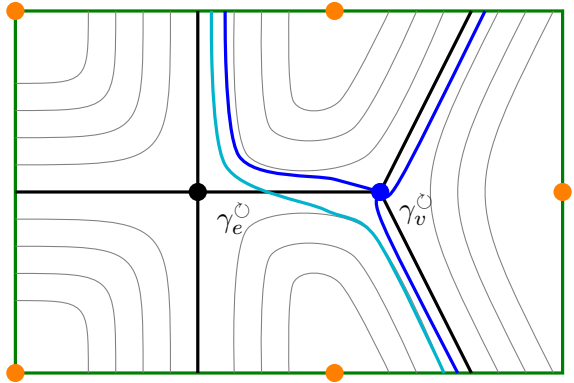

\begin{remark}
The web trajectories considered in the paper give very simple instances of possible general notions of webs. Not only should a notion of web allow for multiple brach points, they may also have $2$-dimensional faces. See \cite{HKS24} for instances of webs (there called spectral networks) in higher rank ($\D^{\on{perf}}(A_n)$-valued) topological Fukaya categories of surfaces. 
\end{remark}

We next introduce notation for counts of trajectories (up to homotopy) starting and ending at given halfedges or edges. 

\begin{definition}\label{def:trajectory}~
\begin{enumerate}[(1)]
\item Let $h$ be a halfedge of $\rgraph$ and $e$ an edge of $\rgraph$. Consider a trajectory $c$ in ${\bf S}\backslash \rgraph_0$ starting near $e$. We write $c\colon h\overset{\circlearrowright}{\shortrightarrow} e'$ if $c$ starts to the right of $e$ (with respect to the orientation induced by $h$) and ends near $e'$. If $h\in e'$, we allow $c$ to be a constant curve.
\item Let $e,e'$ be two edges of $\rgraph$. Let $h_1,h_2$ be the two halfedges of $e$. Given a trajectory $c$ in ${\bf S}\backslash \rgraph_0$, we write $c\colon e\xrightarrow{\circlearrowright} e'$ if either $c\colon h_1\xrightarrow{\circlearrowright} e'$ or $c\colon h_2\xrightarrow{\circlearrowright} e'$
\item Let $h,h'$ be halfedges of $\rgraph$. Let $v'$ be the vertex of $h'$. We write $c\colon h\overset{\circlearrowright}{\shortrightarrow} h'$ if $c\colon h\xrightarrow{\circlearrowright} e(h')$  and $c$ does not pass by the halfedge preceding $h'$ at $v'$. Stated differently, the latter means that $c$ stops before going one counterclockwise step around $v'$. If $h=h'$, we allow $c$ to be a constant curve
\item Let $e$ be an edge of $\rgraph$, with halfedges $h_1,h_2\in e$, and $h'$ be a halfedge of $\rgraph$. We write $c\colon e\overset{\circlearrowright}{\shortrightarrow} h'$ if $c$ is a trajectory such that either $c\colon h_1\overset{\circlearrowright}{\shortrightarrow} h'$ or $c\colon h_2\overset{\circlearrowright}{\shortrightarrow} h'$.
\end{enumerate}
We similarly write $c\colon x\overset{\circlearrowleft}{\shortrightarrow} y$, with $x=h,e$ and $y=h',e'$, if $c$ goes counterclockwise.

We consider all of the above trajectories up to homotopies that do not move the endpoints away from small neighborhoods of the corresponding edges.
\end{definition}

\begin{remark}
For two edges $e,e'$ of $\rgraph$, the set $\{c\colon e\xrightarrow{\circlearrowright} e'\}$ counts the number of times that the clockwise trajectory $\cwt_e$ starting at $e$ passes by $e'$ (these are either 0,1 or 2 times). For a vertex $v$, the set $\bigcup_{h\in H(v)} \{c\colon h\to e'\}$ counts the number of times that the clockwise web trajectory $\cwt_v$ passes by the edge $e'$ (these are at most $2n-1$ times, if $v$ is $n$-valent). Similar interpretations apply to unions of sets $\{c\colon h\xrightarrow{\circlearrowright} h'\}$ and $\{c\colon e\xrightarrow{\circlearrowright} h'\}$. 
\end{remark}

\subsection{Parametrized perverse schobers}\label{subsec:schobers}

We briefly recall the notion of a parametrized perverse schober, see \cite[Section 3]{CHQ23}, \cite[Sections 3,4]{Chr22}. 

For $n\in\mathbb{N}_{\geq 1}$, we let $\rgraph_{n}$ be the ribbon graph with a single vertex $v$ and $n$ incident external edges. We call $\rgraph_n$ the $n$-spider. 

\begin{definition}\label{def:schobernspider}
Let $n\geq 1$. A perverse schober on the $n$-spider consists of the following data: 
\begin{enumerate}
\item[(1)] If $n=1$, a spherical adjunction of stable $\infty$-categories
\[ F\colon \mathcal{V}\longleftrightarrow \mathcal{N}\cocolon G\,,\]
meaning an adjunction whose twist functor $C_{\mathcal{V}}=\on{cof}(\on{id}_{\mathcal{V}}\xrightarrow{\on{unit}}GF)\in \on{Fun}(\mathcal{V},\mathcal{V})$ and cotwist functor $C_{\mathcal{N}}=\on{fib}(FG\xrightarrow{\on{counit}}\on{id}_{\mathcal{N}})\in \on{Fun}(\mathcal{N},\mathcal{N})$ are autoequivalences (note that there exist different conventions for the twist and cotwists in the literature). We also call the functors $F,G$ spherical functors.
\item[(2)] If $n\geq 2$, a collection of adjunctions of stable $\infty$-categories
\[ (F_i\colon \mathcal{V}^n\longleftrightarrow \mathcal{N}_i\cocolon G_i)_{i\in \mathbb{Z}/n\mathbb{Z}}\]
satisfying that
\begin{enumerate}
    \item $G_i$ is fully faithful, which is equivalent to $F_iG_i\simeq \on{id}_{\mathcal{N}_i}$ via the counit,
    \item $F_{i}\circ G_{i+1}$ is an equivalence of $\infty$-categories,
    \item $F_i\circ G_j\simeq 0$ if $j\neq i,i+1$,
    \item $G_i$ admits a right adjoint $\on{radj}(G_i)$ and $F_i$ admits a left adjoint $\on{ladj}(F_i)$ and
    \item $\on{fib}(\on{radj}(G_{i+1}))=\on{fib}(F_{i})$ as full subcategories of $\mathcal{V}^n$.
\end{enumerate}
\end{enumerate}
We will consider a collection of functors $(F_i\colon \mathcal{V}^n\rightarrow \mathcal{N}_i)_{i\in \mathbb{Z}/n}$ a perverse schober on the $n$-spider if there exist adjunctions $(F_i\dashv \on{radj}(F_i))_{i\in \mathbb{Z}/n\mathbb{Z}}$ defining a perverse schober on the $n$-spider. 
\end{definition}

\begin{lemma}\label{lem:schobern-gonfromsphadj}
Let $n\geq 2$. A collection of adjunctions of stable $\infty$-categories
\[ (F_i\colon \mathcal{V}^n\longleftrightarrow \mathcal{N}_i\cocolon G_i)_{i\in \mathbb{Z}/n\mathbb{Z}}\]
defines a perverse schober on the $n$-spider if and only if they satisfy conditions (2).(a),(b),(c) from \Cref{def:schobernspider} and the functor 
\[
\prod_{i=1}^nF_i\colon \mathcal{V}^n\longrightarrow \prod_{i=1}^n \mathcal{N}_i
\]
is spherical.
\end{lemma}

\begin{proof}
If the functors $(F_i)_{1\leq i\leq n}$ define a perverse schober on the $n$-gon, the functor $\prod_{i=1}^nF_i$ is spherical by \cite[Lem.~3.8]{Chr22}.

Conversely, suppose conditions (a),(b),(c) are satisfied and the functor $\prod_{i=1}^nF_i$ is spherical. Then $\prod_{i=1}^nF_i$ admits repeated left and right adjoints, see \cite[Cor.~2.5.14]{DKSS21}, hence so do the functors $F_i$, showing condition (d). The cotwist functor $C_{\prod_i \N_i}$ of the adjunction $\prod_{i=1}^nF_i\dashv \prod_{i=1}^n G_i$ acts by (a),(b),(c) as a cyclic permutation on the components of $\prod_{i=1}^n\N_i$. The right adjoint $\prod_{i\in \mathbb{Z}/n}\on{radj}(G_{i+1})$ of $\prod_{i\in \mathbb{Z}/n}G_{i+1}$ is given by $C_{\prod_i \N_i}^{-1}\circ \prod_{i=1}^nF_{i+1}$, which is on the $i$-th component given $F_i$ composed with the equivalence $(F_{i}G_{i+1})^{-1}\colon \N_i\simeq \N_{i+1}$, from which condition (e) directly follows.
\end{proof}

Given a ribbon graph $\rgraph$, we denote by $\on{Exit}(\rgraph)$ its exit path category, whose objects are the vertices and edges of $\rgraph$ and non-identity morphisms go from the vertices to the edges according to incidence. Note that a functor $\on{Exit}(\rgraph_n)\to \on{LinCat}_R$ equivalently encodes a collection of functors $(F_i)_{i\in \mathbb{Z}/n}$ as in \Cref{def:schobernspider}.

\begin{definition}\label{def:schober}
Let $\rgraph$ be a ribbon graph. A functor $\mathcal{F}\colon \on{Exit}(\rgraph)\rightarrow \on{St}$ is called a $\rgraph$-parametrized perverse schober if for each vertex $v$ of $\rgraph$, the restriction of $\mathcal{F}$ to $\on{Exit}(\rgraph)_{v/}$ determines a perverse schober parametrized by the $n$-spider in the sense of \Cref{def:schobernspider}.
\end{definition}

\begin{definition}
Let $\rgraph$ be a ribbon graph and $\mathcal{F}$ a $\rgraph$-parametrized perverse schober. 
\begin{enumerate}[(1)]
\item The generic stalk of $\mathcal{F}$ is the equivalence class of the value of $\mathcal{F}$ at any edge of $\rgraph$. Note that the generic stalk does not depend on the choice of the edge. We typically denote the generic stalk by $\N$. 
\item Let $v$ be a vertex of $\rgraph$ of valency $n$ with incident edges halfedges $h_1,\dots,h_n$ and edges $e_1,\dots,e_n$. The $\infty$-category of vanishing cycles of $\mathcal{F}$ at $v$ is defined as the equivalence class of the fiber of the functor 
\[\prod_{i=1}^{n-1}\mathcal{F}(v\xrightarrow{h_i} e_i)\colon \mathcal{F}(v)\to \prod_{i=1}^{n-1}\mathcal{F}(e_i)\,.\]
The $\infty$-category of vanishing cycles is independent on the choice of ordering of the halfedges at $v$ and typically denoted by $\V$.
\item Let $v$ be a vertex as above. We call $v$ a singularity of $\mathcal{F}$ if the $\infty$-category of vanishing cycles at $v$ is non-trivial, i.e.~$\mathcal{V}\not \simeq 0$. 
\end{enumerate}
\end{definition}

\begin{remark}
If an $n$-valent vertex $v$ is not a singularity of a perverse schober $\mathcal{F}$, then $\mathcal{F}(v)\simeq \on{Fun}(\Delta^{n-2},\N)$ with $\N$ the generic stalk. If $v$ is a singularity, there exists a spherical functor $F\colon \V\to \N$ such that $\mathcal{F}(v)\simeq \V\overset{\rightarrow}{\times_F} \on{Fun}(\Delta^{n-2},\N)$ is the lax limit of $\V$ and $\on{Fun}(\Delta^{n-2},\N)$ along the functor $\V\xrightarrow{F} \N\xrightarrow{(\Delta^{\{0\}}\subset \Delta^{n-2})_*} \on{Fun}(\Delta^{n-2},\N)$, where the second functor is right Kan extension. This spherical functor $F\colon \V\to \N$ is unique in the appropriate sense and referred to as the spherical functor underlying $\mathcal{F}$ at $v$. The spherical functor underlying a non-singular vertex is given by $0\colon 0\to \N$. We refer to \cite{CHQ23} for proofs of the above statements.
\end{remark}

\begin{definition}\label{def:sections}
Let $\mathcal{F}$ be a $\rgraph$-parametrized perverse schober.
\begin{itemize}
    \item We denote by $\glsec(\rgraph,\mathcal{F})=\on{lim}(\mathcal{F})$ the limit of $\mathcal{F}$ in $\on{St}$. We call $\glsec(\rgraph,\mathcal{F})$ the $\infty$-category of global sections of $\mathcal{F}$. Note also that $\glsec(\rgraph,\mathcal{F})$ can be identified with the $\infty$-category of coCartesian sections of the Grothendieck construction $p\colon \Gamma(\mathcal{F})\to \on{Exit}(\rgraph)$ of $\mathcal{F}$, see \Cref{subsec:Grothendieckconstr}.
    \item We denote by $\losec(\rgraph,\mathcal{F})$ the $\infty$-category of all sections of the Grothendieck construction of $\mathcal{F}$ (which describes the lax limit of $\mathcal{F}$). We call $\losec(\rgraph,\mathcal{F})$ the $\infty$-category of lax sections of $\mathcal{F}$. 
\end{itemize}
Note that $\glsec(\rgraph,\mathcal{F})\subset \losec(\rgraph,\mathcal{F})$ is a full subcategory.
\end{definition}

Given a ribbon graph $\rgraph$, we denote by $\rgraph^\vee$ the ribbon graph whose underlying graph is $\rgraph$ and where the cyclic orderings of halfedges of $\rgraph^\vee$ is opposite to the cyclic orderings of the halfedges of $\rgraph$.

We denote by $(\mhyphen)^{\on{op}}\colon \on{St}\to \on{St}$ the autoequivalence that maps a stable $\infty$-category $\C$ to the opposite stable $\infty$-category $\C^{\on{op}}$.

\begin{lemma}\label{lem:dualschober}
Let $\mathcal{F}$ be a $\rgraph$-parametrized perverse schober. Then 
\[ \mathcal{F}^{\on{op}}\coloneqq (\mhyphen)^{\on{op}}\circ \mathcal{F}\colon \on{Exit}(\rgraph^\vee)=\on{Exit}(\rgraph)\to \on{St}\] 
is a $\rgraph^\vee$-parametrized perverse schober. Furthermore, there exists an equivalence of stable $\infty$-categories
\[ \mathcal{H}(\rgraph^\vee,\mathcal{F}^{\on{op}})\simeq \mathcal{H}(\rgraph,\mathcal{F})^{\on{op}}\,.\]
\end{lemma}

\begin{proof}
To show that $\mathcal{F}^{\on{op}}$ is a $\rgraph^\vee$-parametrized perverse schober, it suffices to show the perversity locally at each vertex of $\rgraph$. The passage to the opposite category exchanges left and right adjoints. At a $1$-valent vertex, the sphericalness of $\mathcal{F}^{\on{op}}$ thus follows from the fact that an adjunction $F\dashv G$ is spherical if and only if the adjunction $E\dashv F$ is spherical, which follows for instance from \cite[Cor.~2.5.16]{DKSS21}. At a vertex of valency $n\geq 2$, conditions (a),(b) and (c) from \Cref{def:schobernspider} are clearly satisfied by $\mathcal{F}^{\on{op}}$. Hence, it thus follows from \Cref{lem:schobern-gonfromsphadj} that $\mathcal{F}^{\on{op}}$ is also locally perverse. 

The equivalence $ \mathcal{H}(\rgraph^\vee,\mathcal{F}^{\on{op}})\simeq \mathcal{H}(\rgraph,\mathcal{F})^{\on{op}}$ follows from the fact that the autoequivalence $(\mhyphen)^{\on{op}}$ preserves limits.
\end{proof}

\noindent {\bf Transport of perverse schobers}

Let ${\bf S}$ be a marked surface with spanning ribbon graph $\rgraph$. We finally briefly recall the notion of transport $\mathcal{F}^\rightarrow(\gamma)\colon \mathcal{F}(e)\simeq \mathcal{F}(e')$ of a perverse schober $\mathcal{F}$ along a path $\gamma$ in ${\bf S}\backslash \rgraph_0$ starting at an edge $e$ and ending at an edge $e'$ of $\rgraph$ . We also consider $\gamma$ up to suitable homotopy. The notion of transport was introduced in \cite[Section 4.3]{Chr23}, as an intermediate step in the definition of the monodromy of perverse schobers. We refer to loc.~cit. for a full treatment of the notion of transport. 

We can write $\gamma$ as the composite of segments, meaning local paths, each of which is contained in a contractible neighborhood of a vertex of $\rgraph$. The transport is defined as the composite of the transports of these segments. A segment $\delta$ near a vertex $v$ with incident edges $e_1,\dots, e_n$ (ordered compatibly with their cyclic order) can be chosen to start at an edge $e_i$ and end at an edge $e_{i+j}$, going $j\in \mathbb{Z}$-steps counterclockwise. We can thus consider the segment $\delta$ as the composite of $|j|$ minimal segments which go exactly one step counterclockwise\footnote{Note that a $1$-step counterclockwise segment corresponds to a clockwise trajectory with respect to the line field of the ribbon graph, since the orientation of a trajectory is opposite to the direction it turns around a vertex.} ($j=1$) or clockwise ($j=-1$). We define the transport along the segment $\delta$ as the composite of these minimal segments. We thus suppose that $\delta$ is minimal. If $j=1$, we define the transport as the equivalence
\[ \mathcal{F}^{\rightarrow}(\delta)= \mathcal{F}(v\to e_{i+1})\circ \mathcal{F}(v\to e_i)^L\,.\]
If $j=-1$, we define the transport as the equivalence
\[ \mathcal{F}^{\rightarrow}(\delta)= \mathcal{F}(v\to e_{i-1})\circ \mathcal{F}(v\to e_i)^R\,.\]

\section{Induction functors for parametrized perverse schobers} 

Let $\rgraph$ be a spanning graph of a marked surface and let $\mathcal{F}$ be a $\rgraph$-parametrized perverse schober. Given an object $x\in \on{Exit}(\rgraph)$, meaning either $x=v$ a vertex or $x=e$ an edge, consider the functor $\on{ev}_x\colon \losec(\rgraph,\mathcal{F})\to \mathcal{F}(x)$ evaluating sections of the Grothendieck construction $p\colon \Gamma(\mathcal{F})\to \on{Exit}(\rgraph)$ of $\mathcal{F}$ at $x$. We will prove in this section that the functor $\on{ev}_x$ admits left and right adjoints, denoted $\indL_x,\indR_x$ and called the induction functors. We further describe the induction functors explicitly in terms of the clockwise and counterclockwise trajectories of \Cref{def:cwtrajectory} and the transport of $\mathcal{F}$. 

For the discussion of induction, we first consider the case that $\rgraph$ has no $1$-valent vertices. Note that one can always choose a spanning ribbon graph without $1$-valent vertices if the surface is not the $1$-gon. We generalize some of the results to general ribbon graphs in \Cref{subsec:1-valent}. In \Cref{subsec:restriction1}, we focus on a preliminary construction, which already allows to fully treat the case where $x$ is an external edge. In \Cref{subsec:restriction2}, we then discuss induction from edges and vertices. In \Cref{subsec:restriction3}, we further study the induction from sections supported on subgraphs/subsurfaces to global sections. Finally, we prove in \Cref{subsec:globaltwist} that left and right induction are related by spherical twists.

Note that by \Cref{lem:dualschober}, it will suffice to spell out the construction and study only for the left induction functors, as they then also follow for the right induction functors by passing to opposite $\infty$-categories.

\subsection{Boundary cap and cup functors}\label{subsec:restriction1}

We fix a marked surface ${\bf S}$ with a spanning graph $\rgraph$ without $1$-valent vertices and a $\rgraph$-parametrized perverse schober $\mathcal{F}$. This sections concerns a preliminary construction and study of a functor $U_h^L$ assigned to a halfedge of $\rgraph$, which will be used in the construction of the induction functors. In the case that $h$ is part of external edge $e$ of $\rgraph$, the functor $U_h^L$ already coincides with the induction functor $\indL_e$ from $e$.  

A first application of the construction and study of $U_h^L$ in this section will be a proof that the evaluation functors of a perverse schober at the boundary edges of any given boundary component themselves form a perverse schober on an $n$-gon, see \Cref{thm:globalsphericalness}. This perverse schober can be considered as the perverse pushforward of the original perverse schober along a map from the surface ${\bf S}$ to the $n$-gon. 

We turn to the construction of the functor $U_h^L$.

\begin{construction}\label{constr:Uh}
For $v$ a vertex of $\rgraph$ with incident edges $e_1,\dots,e_n$, the $\infty$-category $\mathcal{F}(v)$ is equivalent to the full subcategory of the $\infty$-category of lax sections $\mathcal{L}(\rgraph,\mathcal{F})$ consisting of sections $X\colon \on{Exit}(\rgraph)\to \Gamma(\mathcal{F})$ of the Grothendieck construction $p\colon \Gamma(\mathcal{F})\to \on{Exit}(\rgraph)$ of $\mathcal{F}$, which are $p$-relative left Kan extensions of their restriction to $v$, see \cite[Prop.~4.3.2.15]{HTT}. Concretely, being such a Kan extension amounts to the section $X$ satisfying:
\begin{itemize}
\item $X(v')\simeq 0$ for every vertex $v'\neq v$ and $X(e)\simeq 0$ for every edge $e\neq e_1,\dots,e_n$.
\item any morphism $X(v\to e_i)$ is $p$-coCartesian, meaning it describes an equivalence\newline  $\mathcal{F}(v\to e_i)(X(v))\simeq X(e_i)$. 
\end{itemize}
We denote the arising embedding of $\mathcal{F}(v)$ into $\mathcal{L}(\rgraph,\mathcal{F})$ by $\iota_v$. The right adjoint of $\iota_v$ is given by the functor $\on{ev}_v$ restricting sections to $v$, see \cite[Prop.~4.3.2.17]{HTT}. We similarly denote for every edge $f$ of $\rgraph$ the apparent embedding of $\mathcal{F}(f)$ into $\mathcal{L}(\rgraph,\mathcal{F})$, mapping an object to the corresponding section supported at $f$, by $\iota_f$. The right adjoint of $\iota_f$ is given by $\on{ev}_f$. 

We fix a choice of a halfedge $h$ of $\rgraph$, which determines an incident vertex $v$ and a corresponding edge $e$. \\

Consider the sequences of edges $e_1,\dots,e_m$ and vertices $v_1,\dots,v_{m-1}$ of $\rgraph$ obtained as follows. We will also later write $m(h)\coloneqq m$. We let $e_1=e$ and $v_1=v$. We recursively define $e_{i+1}$ as the edge incident to $v_i$ which follows $e_i$ in the counterclockwise order and $v_{i+1}\neq v_i$ as the other vertex incident to $e_{i+1}$. After finitely many steps, the external edge $e_m$ is reached ($e_m=e_1$ is permitted), by virtue of the assumption that every boundary component contains at least one marked point. Note that the edges $e_1,\dots,e_m$ trace along the clockwise trajectory $\cwt_h$.

 This external edge $e_m$ can be described as follows: If $e$ is an external edge lying on a boundary component $B$ of ${\bf S}$, then $e_m$ also lies on $B$ and follows $e$ in the clockwise orientation of $B$ by one step.
If $e$ is not external, consider the decomposition of ${\bf S}$ into polygons dual to $\rgraph$. Cutting  ${\bf S}$ along the edge dual to $e$ yields a new  marked surface $\on{cut}_e({\bf S})$, where $e$ is external. The edge $e_m$ is now obtained as before as the clockwise successor of $e$. 

\begin{figure}[ht]
\begin{center}
\begin{tikzpicture}[decoration={markings, 
	mark= at position 0.94 with {\arrow{stealth}}}]
\draw[color=ao][very thick] (0,0) circle(2);
\draw[color=ao][very thick] (0,0) circle(0.7);
\node (0) at (0.6,0.4){};
\node (1) at (2,0){};
\node (2) at (0,2){};
\node (3) at (0,-2){};
\node (4) at (-2,0){};

\node (a) at (-1.15,-1.15){};
\node (b) at (-1.15,1.15){};
\node (c) at (1.15,-1.15){};
\node (d) at (1.15,1.15){};
\node (e) at (-1.15,0){};
\node () at (-0.95,0.2){$e$};

\fill[orange] (0) circle(0.1);
\fill[orange] (1) circle(0.1);
\fill[orange] (2) circle(0.1);
\fill[orange] (3) circle(0.1);
\fill[orange] (4) circle(0.1);

\fill (a) circle(0.1);
\fill (b) circle(0.1);
\fill (c) circle(0.1);
\fill (d) circle(0.1);
\fill (e) circle(0.1);

\draw[very thick] (-1.15,0)--(-0.7,0) (-1.15,-1.15)--(-1.15,1.15)--(1.15,1.15)--(1.15,-1.15)--(-1.15,-1.15) (-1.15,-1.15)--(-1.412,-1.412) (1.15,1.15)--(1.412,1.412) (1.15,-1.15)--(1.412,-1.412) (-1.15,1.15)--(-1.412,1.412);
\draw[color=blue, very thick][postaction={decorate}]  plot [smooth] coordinates {(-0.7,0.05) (-0.75,0.9) (0.85,0.85) (0.85,-0.85) (-0.85,-0.85) (-0.7,-0.05)};
\end{tikzpicture}
\caption{An example of a ribbon graph where the external successor edge $e_m$ of the external edge $e$ coincides with $e$. The corresponding clockwise trajectory $\cwt_e$ is in blue.}
\end{center}
\end{figure}
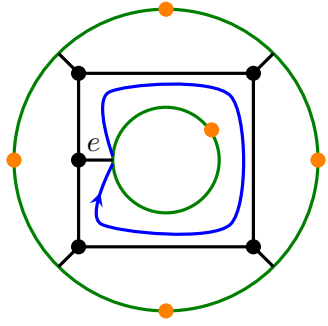

Consider the simplicial set $\Lambda^2_0$, given by the span with non-degenerate $1$-simplicies $0\to 1$ and $0\to 2$. For $i=1,2$, we denote by $\Delta^0\xrightarrow{\Delta^{\{i\}}}\Lambda^2_0$ the inclusion of the $0$-simplex $i$. We define $Z_m\in \on{Set}_\Delta$ as the coequalizer of the diagram
\[ 
\begin{tikzcd}[column sep=7em]
\coprod_{1<i<m}\Delta^0 \arrow[r, "\coprod_{1<i<m}\Delta^{\{1\}}\times \{i\}", shift left] \arrow[r, "\coprod_{1<i<m}\Delta^{\{2\}}\times \{i-1\}"', shift right] & \coprod_{1\leq i< m} \Lambda^2_0\times \{i\}
\end{tikzcd}
\]
informally describing a sequence of spans with identified endpoints.
For $1\leq i<m$, consider the apparent natural transformations
\[
\iota_{e_i}\to \iota_{v_i}\circ \mathcal{F}(v_i\to e_i)^L
\]
and 
\[
\iota_{e_{i+1}}\circ   \mathcal{F}(v_i\to e_{i+1}) \to  \iota_{v_i}
\]
which restrict at $e$ to the unit of $\mathcal{F}(v_i\to e_i)^L\dashv \mathcal{F}(v_i\to e_i)$ (which is an equivalence) and the identity 
on $\mathcal{F}(v_i\to e_{i+1})$, respectively. Using these natural transformations, we consider the functor $\underline{U}^L_h:Z^{\on{op}}_m\to \on{Fun}(\mathcal{F}(e),\losec(\rgraph,\mathcal{F}))$ which is uniquely determined by the property that it restricts on $(\Lambda^2_0\times \{i\})^{\on{op}}$, with $1\leq i<m$, to the composite of the diagram 
\begin{equation}\label{eq:evLdiag1}
\begin{tikzcd}
                       & \iota_{v_i}\circ \mathcal{F}(v_i\to e_i)^L &                                                                                            \\
\iota_{e_i} \arrow[ru] &                                            & \iota_{e_{i+1}}\circ  \mathcal{F}(v_i\to e_{i+1})\circ \mathcal{F}(v_i\to e_i)^L \arrow[lu]
\end{tikzcd}
\end{equation}
with the transport equivalence 
\begin{equation}\label{eq:transport} \mathcal{F}^\rightarrow(\cwt_i)= \mathcal{F}(v_{i-1}\to e_i)\circ \mathcal{F}(v_{i-1}\to e_{i-1})^L\circ\dots \circ \mathcal{F}(v_1\to e_2)\circ \mathcal{F}(v_1\to e_1)^L\,.
\end{equation} 

Passing to right adjoints yields the functor $\underline{U}_h\colon Z_m \to \on{Fun}(\losec(\rgraph,\mathcal{F}),\mathcal{F}(e))$, which restricts on $\Lambda^2_0\times \{i\}$, with $1\leq i<m$, to the composite of the diagram 
\begin{equation}\label{eq:evLdiag2}
\begin{tikzcd}
              & \mathcal{F}(v_i\to e_i)\circ \on{ev}_{v_i} \arrow[ld] \arrow[rd] &                             \\
\on{ev}_{e_i} &                                                                  &  \mathcal{F}(v_i\to e_i)\circ \mathcal{F}(v_i\to e_{i+1})^R \circ \on{ev}_{e_{i+1}}
\end{tikzcd}
\end{equation}
with $\mathcal{F}^\rightarrow(\cwt_i)^{-1}$.

We define 
\[ U_h^L\coloneqq \on{colim}_{Z^{\on{op}}_m}\underline{U}_h^L\colon \mathcal{F}(e)\to \mathcal{L}(\rgraph,\mathcal{F})\] 
and
\[ U_h\coloneqq \on{lim}_{Z_m}\underline{U}_h\colon \mathcal{L}(\rgraph,\mathcal{F})\to \mathcal{F}(e)\,.\] 
Note that the functors form an adjunction $U_h^L\dashv U_h$, since the passage to the adjoint is an exact functor.

By construction, there is a natural transformation $\eta\colon \iota_{e}\hookrightarrow U_h^L$ that restricts to an equivalence at $e\in \on{Exit}(\rgraph)$. Passing to right adjoints induces a natural transformation $\eta^R\colon U_h\to \on{ev}_e$.
\end{construction}

\begin{remark}
Each section in the image of $U_h^L\colon \mathcal{F}(e)\to \mathcal{L}(\rgraph,\mathcal{F})$ should be considered as associated with a decorated trajectory, namely the clockwise trajectory $\cwt_h$ decorated by an object in the generic stalk $\mathcal{F}(e)$. We will see below that the trajectory determines where the corresponding section is supported in $\on{Exit}(\rgraph)$.
\end{remark}

\begin{lemma}\label{lem:Uhproperties}
Let $h$ be a halfedge of $\rgraph$, contained in the edge $e$.
\begin{enumerate}[(1)]
\item The natural transformation $\eta^R\colon U_h\to \on{ev}_e$ 
from \Cref{constr:Uh} restricts on $\mathcal{H}(\rgraph,\mathcal{F})\subset \mathcal{L}(\rgraph,\mathcal{F})$ to a natural equivalence.
\item Given $0\not=X\in \mathcal{F}(e)$ and a non-degenerate $1$-simplex $\alpha\colon v'\to e'$ in $\on{Exit}(\rgraph)$, the $1$-simplex $U_h^L(X)(\alpha)$ 
is $p$-coCartesian, unless $\alpha\colon \sigma(\tau(h))\xrightarrow{\tau(h)} e$ arises from another halfedge $\tau(h)\neq h$ of $e$.
\item If $e$ is an external edge, the functor $U_h^L$ 
factors through $\mathcal{H}(\rgraph,\mathcal{F})\subset \mathcal{L}(\rgraph,\mathcal{F})$. In this case, $U_h^L$ restricts to define a left adjoint $\indL_e$ of $\on{ev}_e$.
\end{enumerate}
\end{lemma}

\begin{proof}
It follows from the characterization of coCartesian sections that, for all $1\leq i \leq m$, the natural transformation $\mathcal{F}(v_i\to e_i)\circ \on{ev}_{v_i}\to \on{ev}_{e_i}$ appearing on the left in \eqref{eq:evLdiag2} restricts on $\mathcal{H}(\rgraph,\mathcal{F})$ to a natural equivalence. Inspecting the diagram $\underline{U}_h$, and using that equivalences are preserved under pullback, we thus see that the inclusion $\eta^R$ of $\on{ev}_e$ into the limit $U_h$ restricts on $\mathcal{H}(\rgraph,\mathcal{F})$ to an equivalence, showing part (1).

For part (2), we first note that colimits in the $\infty$-category $\mathcal{L}(\rgraph,\mathcal{F})$ of sections of $p\colon \Gamma(\mathcal{F})\to \on{Exit}(\rgraph)$ are computed pointwise in $\on{Exit}(\rgraph)$, which follows from \cite[5.1.2.3, 4.3.1.10]{HTT} and the fact that a $p$-relative colimit in a fiber of $p$ is exactly a colimit in the fiber, see \cite[4.3.1.5.(2)]{HTT} for $\mathcal{E}=\ast$. Thus, $U^L_h(X)(\alpha)$ arises as the colimit of $\underline{U}^L(X)(\alpha)$. For $\alpha\colon \sigma(\tau(h))\xrightarrow{\tau(h)}e$, the morphism $U^L_h(X)(\alpha)$ contains a summand of the form $0\to X$, and is hence not coCartesian. Now suppose that $\alpha$ is a different non-degenerate $1$-simplex. The diagram $\underline{U}^L(X)(\alpha)$ vanishes for most $(l,i)\in Z^{\on{op}}$ with $l\in \Lambda^2_0$: it either vanishes for all elements $Z^{\on{op}}$ if $\alpha$ does not lie on the trajectory $\cwt_h$ or it vanishes for all except three or six elements of $Z^{\on{op}}$ if $\alpha$ lies on the trajectory $\cwt_h$. In the latter case, inspecting the diagram and using that the sections in the image of a functor $\iota_{v_i}$ are $p$-coCartesian near $v_i$, one readily sees that $U^L_h(X)(\alpha)$ is also $p$-coCartesian.

Part (2) shows that $U_h^L$ factors through $\mathcal{H}(\rgraph,\mathcal{F})\subset \mathcal{L}(\rgraph,\mathcal{F})$ in the case that $e$ is an external edge. It thus restricts on $\mathcal{H}(\rgraph,\mathcal{F})$ to the left adjoint of $U_h\simeq \on{ev}_e$, showing part (3).
\end{proof}

\begin{lemma}\label{lem:evUh}
Let $h$ be a halfedge of $\rgraph$, contained in the edge $e$.
\begin{enumerate}[(1)]
\item Let $f$ be an edge of $\rgraph$. Then 
\[ \on{ev}_f\circ U_h^L \simeq \bigoplus_{c\colon h\xrightarrow{\circlearrowright} f} \mathcal{F}^\rightarrow(c)\colon \mathcal{F}(e)\longrightarrow \mathcal{F}(f)\]
decomposes into the sum of the transport equivalences along clockwise trajectories $c\colon h\xrightarrow{\circlearrowright} f$, see \Cref{def:trajectory}. 
\item Let $v$ be a vertex of $\rgraph$. Then 
\[ \on{ev}_v\circ U_h^L\simeq \bigoplus_{h'\in H(v)} \bigoplus_{c\colon h\xrightarrow{\circlearrowright} h'} \mathcal{F}(h')^L\circ \mathcal{F}^\rightarrow(c)\colon \mathcal{F}(e)\longrightarrow \mathcal{F}(v)\]
decomposes into the sum of the transport equivalences along clockwise trajectories $c\colon h\xrightarrow{\circlearrowright} h'$ composed with $\mathcal{F}(v\xrightarrow{h'}e(h'))^L\colon \mathcal{F}(e')\to\mathcal{F}(v)$.
\item Suppose that the edge $e$ is external. The unit $\on{id}_{\mathcal{F}(e)}\to U_hU_h^L$ is a split inclusion.
\end{enumerate}
\end{lemma}

\begin{proof}
Parts (1) and (2) follow readily from an inspection of the construction of $U_h^L$, using that colimits of sections of $p\colon \Gamma(\mathcal{F})\to \on{Exit}(\rgraph)$ are computed pointwise in $\on{Exit}(\rgraph)$.

We turn to part (3). Let $h\in e$ be the halfedge of $e$. 
Consider the apparent diagram of natural transformations $\underline{u}\colon Z^{\on{op}}_{m(h)}\to \on{Fun}(\Delta^1\times \mathcal{F}(e),\mathcal{F}(e))$, of the form 
\[ z\mapsto (\on{id}_{\mathcal{F}(e)}\to \on{ev}_eU_h^L(z))\,,\] 
which to $(0,1)$ and $(0,0)$ assigns the equivalences $\on{id}_{\mathcal{F}(e)}\simeq \on{ev}_e\circ\iota_e$ and $\on{id}_{\mathcal{F}(e)}\simeq \on{ev}_e\circ \iota_{v_1}\circ \mathcal{F}(v_1\to e)^L$ and assigns to all other $z\in Z^{\on{op}}_{m(h)}$ the zero natural transformation. Passing to the colimit over $Z^{\on{op}}_{m(h)}$ yields the split inclusion $u\colon \on{id}_{\mathcal{F}(e)}\to \on{ev}_e\circ U_h^L\simeq U_h\circ U_h^L$, which we want to show to be the unit of the adjunction. For this, we next verify that the natural transformation induces the relevant equivalence between mapping spaces. 

Given a $1$-simplex $\alpha\colon z=(i,0)\to z'=(i,1)$ in $Z_{m(h)}$, $X\in \mathcal{F}(e)$ and $Y\in \mathcal{H}(\rgraph,\mathcal{F})$, consider the commutative diagram:
\begin{equation}\label{eq:diagequiv}
\begin{tikzcd}
{ \on{Map}_{\mathcal{L}(\rgraph,\mathcal{F})}(\underline{U}_h^L(z)(X),Y)} \arrow[d, "\simeq"] \arrow[r]   & { \on{Map}_{\mathcal{L}(\rgraph,\mathcal{F})}(\underline{U}_h^L(z')(X),Y)} \arrow[d, "\simeq"] \\
{ \on{Map}_{\mathcal{F}(e)}(X,\underline{\on{ev}}_{e}(z)(Y))} \arrow[r, "\simeq"] & { \on{Map}_{\mathcal{F}(e)}(X,\underline{\on{ev}}_{e}(z')(Y))}         
\end{tikzcd}
\end{equation}
The fact that the lower horizontal morphism is an equivalence amounts to the fact that the left vertical morphism in \eqref{eq:evLdiag2} is an equivalence on the global section $Y$. This shows that the upper morphism is an equivalence as well. Consider the diagram $\underline{M}\colon Z_{m(h)}\to \on{Fun}(\Delta^1,\mathcal{S})$, assigning to $z\in Z_{m(h)}$ the morphism of spaces
\[ \on{Map}_{\mathcal{L}(\rgraph,\mathcal{F})}(\underline{U}_h^L(z)(X),Y)\xrightarrow{\on{ev}_e} \on{Map}_{\mathcal{F}(e)}(\on{ev}_e\circ \underline{U}_h^L(z)(X),\on{ev}_e(Y))\xrightarrow{(\mhyphen)\circ \underline{u}_z} \on{Map}_{\mathcal{F}(e)}(X,\on{ev}_e(Y))\,.\]
The fact that both horizontal morphisms in \eqref{eq:diagequiv} are equivalences shows that $\underline{M}$ maps each morphism $(i,0)\to (i,1)$ in $Z_{m(h)}$ to an equivalence. Thus the limit of $\underline{M}$ is equivalent to its value at $(0,1)\in Z_{m(h)}$, where $\underline{u}$ evaluates to the unit of $\iota_e\dashv \on{ev}_e$, and $\underline{M}$ is hence an equivalence.
\end{proof}

\begin{remark}\label{rem:infiniteinduction}
If we drop the assumption that each boundary component of ${\bf S}$ contains a marked point, some of the clockwise trajectories will be of infinite length. Thus $U_h^L$ must in this case be built by gluing possibly infinitely many corresponding local functors (supposing that sufficient colimits exist). \Cref{lem:Uhproperties,lem:evUh} then generalize accordingly.
\end{remark}

We end this section with a discussion of the perversity of the boundary restriction functors. Recall that $\rgraph^\partial_1$ denotes the set of external edges of $\rgraph$. Given a boundary component $B$ of ${\bf S}$, we let $\rgraph_1^{\partial,B}$ be the set of external edges ending on $B$.

\begin{theorem}\label{thm:globalsphericalness}
Let ${\bf S}$ be a marked surface with spanning ribbon graph $\rgraph$ (possibly with $1$-valent vertices) and let $\mathcal{F}$ be a $\rgraph$-parametrized perverse schober. Consider a boundary component $B$ of ${\bf S}$ containing $n$ marked points. The functors 
\[ (\on{ev}_e\colon \glsec(\rgraph,\mathcal{F})\longleftrightarrow \mathcal{F}(e) \noloc \indR_e)_{e\in  \rgraph_1^{\partial,B}}\]
define a perverse schober on the $n$-spider. 
\end{theorem}

\begin{proof}
Suppose first that $n\geq 2$. Then ${\bf S}$ admits a spanning graph without $1$-valent vertices. Using contractions of ribbon graphs, see \cite[Lemma 3.16]{CHQ23}, we may thus equivalently prove the theorem in the case that $\rgraph$ has no $1$-valent vertices.  Conditions (2) (a),(b),(c) of \Cref{def:schobernspider} are proven in \Cref{lem:boundaryfunctor} below. Passing to $\on{Ind}$-completions, we first assume that $\mathcal{F}$ takes values in presentable stable $\infty$-categories. Condition (d) then follows from the description of the functors $\indL_e,\indR_e$ from part (4) of \Cref{lem:Uhproperties}: they are expressed as finite colimits of functors that preserve all limits and colimits and thus themselves preserve all limits and colimits, and thus admit adjoints by the adjoint functor theorem. Condition (e) of \Cref{def:schobernspider} is equivalent to the assertion that the closures under colimits of the image of $\on{ladj}(F_i)$ and $G_{i+1}$ coincide. This follows from inspecting \Cref{constr:Uh}: for an external edge $e$ ending on $B$ with $e'$ the clockwise successor external edge on $B$ and any object $X\in \mathcal{F}(e)$, there exists an essentially unique object $Y=\on{ev}_{e'}\indR_e(X)\in \mathcal{F}(e')$ with $\indR_e(X)\simeq \indL_{e'}(Y)$. The statement in the case that $\mathcal{F}$ takes values in small stable $\infty$-categories then follows from the presentable case using \Cref{lem:schobern-gonfromsphadj}, that the adjunction $\on{ev}_e\colon \glsec(\rgraph,\on{Ind}\mathcal{F})\leftrightarrow \on{Ind}\mathcal{F}(e)\noloc \indR_e$ restricts to the corresponding adjunction of $\mathcal{F}$ and that a restriction of a spherical adjunction along full subcategories remains a spherical adjunction.

Consider now the case $n=1$. We add an additional marked point to $B$ and a corresponding external edge $e'$ to $\rgraph$, yielding $\rgraph'$. Let $v$ be the vertex incident to $e'$. We can lift $\mathcal{F}$ to a $\rgraph'$-parametrized perverse schober $\mathcal{F}'$ such that $\mathcal{F}(v)\simeq \on{fib}\mathcal{F}'(v\to e')$. By the above, the Theorem holds for $\mathcal{F}'$ and it follows for $\mathcal{F}$ by \cite[Prop.~3.6]{CHQ23}.
\end{proof}

\begin{corollary}[{$\!\!$\cite[Thm.~5.2.5]{CDW23}}]\label{cor:globalsphadj}
Let ${\bf S}$ be a marked surface with spanning ribbon graph $\rgraph$ (possibly with $1$-valent vertices) and let $\mathcal{F}$ be a $\rgraph$-parametrized perverse schober. The adjunction 
\[ \prod_{e\in \rgraph_1^\partial} \on{ev}_e\colon \glsec(\rgraph,\mathcal{F})\longleftrightarrow \prod_{e\in \rgraph_1^\partial} \mathcal{F}(e)\noloc  \prod_{e\in \rgraph_1^\partial}\indR_e\] 
is spherical. 

In particular, the $\infty$-categorical exact structure on $\mathcal{H}(\rgraph,\mathcal{F})$ induced by $\prod_{e\in \rgraph_1^\partial} \on{ev}_e$ is Frobenius by \Cref{prop:Frob2CYexact}.
\end{corollary}

\begin{proof}
Using \Cref{thm:globalsphericalness} and \cite[Lem.~3.8]{Chr22}, we find that for any boundary component $B$ of ${\bf S}$, the adjunction 
\begin{equation}\label{eq:sphadjB}
\prod_{e\in \rgraph_1^{\partial,B}} \on{ev}_e\colon \glsec(\rgraph,\mathcal{F})\longleftrightarrow \prod_{e\in \rgraph_1^{\partial,B}} \mathcal{F}(e)\noloc \indR_e\end{equation}
is spherical. If $B\not= B'$ are two different boundary components of ${\bf S}$ and $e\in \rgraph_{1}^{\partial, B}$ and $f\in \rgraph_1^{\partial, B'}$, we have by \Cref{lem:evUh} that $\on{ev}_f\circ \indR_e\simeq 0$. Hence the twist and cotwist functors of the adjunction $ \prod_{e\in \rgraph_1^\partial} \on{ev}_e\dashv \prod_{e\in \rgraph_1^\partial} \indR_e$ are equivalent to the commuting products of the twist and cotwist functors of the adjunctions \eqref{eq:sphadjB}, over all boundary components $B$, and therefore equivalences.
\end{proof}

\begin{lemma}\label{lem:boundaryfunctor}
Suppose that $e$ is a boundary edge intersecting a boundary component $B$ of ${\bf S}$ containing at least two marked points. 
\begin{enumerate}[(1)]
\item The functors $\indL_e,\indR_e\colon \mathcal{F}(e)\to \mathcal{H}(\rgraph,\mathcal{F})$ are fully faithful.
\item If $f$ is the successor edge of $e$ in the counterclockwise ordering of external edges intersecting $B$, then $\on{ev}_{f}\circ \indL_e$ is an equivalence. If $f'\neq f,e$ is a different external edge, then $\on{ev}_f\circ \indL_e\simeq 0$.
\item If $f$ is the successor edge of $e$ in the clockwise ordering of external edges intersecting $B$, then $\on{ev}_{f}\circ \indR_e$ is an equivalence. If $f'\neq f,e$ is a different external edge, then $\on{ev}_f\circ \indR_e\simeq 0$.
\end{enumerate}
\end{lemma}

\begin{proof}
We apply \Cref{lem:evUh} to obtain the statements about $\indL_e$. Passing to opposite $\infty$-categories and using \Cref{lem:dualschober} yields the statements about $\indR_e$.
\end{proof}

\subsection{Induction from vertices and internal edges}\label{subsec:restriction2}

We continue our study of the induction functors started in \Cref{subsec:restriction1}. We fix a spanning graph $\rgraph$ of a marked surface ${\bf S}$, such that $\rgraph$ has no $1$-valent vertices. We also fix a $\rgraph$-parametrized perverse schober $\mathcal{F}$.

We build the functors $\indL_e$ and $\indL_v$ by gluing functors $U_h^L,\iota_e,\iota_v$ from \Cref{constr:Uh}.

\begin{lemma}\label{lem:indLe}
Let $e$ be an internal edge of $\rgraph$ with halfedges $h_1,h_2$. Then $\on{ev}_e\colon \mathcal{H}(\rgraph,\mathcal{F})\to \mathcal{F}(e)$ admits a left adjoint $\indL_e$. Further, there is a pushout diagram in $\on{Fun}(\mathcal{F}(e),\mathcal{L}(\rgraph,\mathcal{F}))$
\[
\begin{tikzcd}
\iota_e \arrow[r, "\eta"] \arrow[d, "\eta"] \arrow[rd, "\ulcorner", phantom] & U_{h_1}^L \arrow[d] \\
U_{h_2}^L \arrow[r]                                          & \iota\circ \indL_e   
\end{tikzcd}
\]
where $\iota$ is the inclusion $\mathcal{H}(\rgraph,\mathcal{F})\subset \mathcal{L}(\rgraph,\mathcal{F})$.
\end{lemma}

\begin{proof}
Let $P=U_{h_1}^L\amalg_{\iota_e} U_{h_2}^L$ be the pushout. It readily follows from part (2) of \Cref{lem:Uhproperties} that $P$ factors through $\iota\colon \mathcal{H}(\rgraph,\mathcal{F})\subset \mathcal{L}(\rgraph,\mathcal{F})$, hence $P\simeq \iota \circ P'$ for a functor $P'$. The pullback of right adjoints $P^R\coloneqq U_h\times_{\on{ev}_e}U_{h'}$ is right adjoint to $P$. Part (1) of \Cref{lem:Uhproperties} thus implies that $P^R\circ \iota$ is equivalent to $\on{ev}_e$. Since $P$ factors through $\iota$, the adjunction $P\dashv P^R$ restricts to the desired adjunction $\indL_e\coloneqq P'\dashv P^R\circ \iota\simeq \on{ev}_e$.
\end{proof}

\begin{proposition}\label{prop:eve}
Let $e$ be an edge of $\rgraph$.
\begin{enumerate}[(1)]
\item The unit $\on{id}_{\mathcal{F}(e)}\to \on{ev}_e\indL_e$ is a split inclusion and the counit $\on{ev}_e\indR_e\to \on{id}_{\mathcal{F}(e)}$ is a split projection.
\item Let $f$ be an edge of $\rgraph$. Then
\[ \on{ev}_f\circ \indL_e \simeq \bigoplus_{c:e\overset{\circlearrowright}{\shortrightarrow} f}\mathcal{F}^\rightarrow(c)\]
and
\[ 
\on{ev}_f\circ \indR_e \simeq \bigoplus_{c:e\overset{\circlearrowleft}{\shortrightarrow} f}\mathcal{F}^\rightarrow(c)
\]

see \Cref{def:trajectory} for the notation $c\colon e\xrightarrow{\circlearrowright} f$ and $c\colon e\xrightarrow{\circlearrowleft} f$. 
\item Let $v$ be a vertex of $\rgraph$. Then 
\[ \on{ev}_v\circ \indL_e\simeq \bigoplus_{h\in H(v)}\bigoplus_{c:e\overset{\circlearrowright}{\shortrightarrow}h} \mathcal{F}(v\xrightarrow{h}e(h))^L\circ \mathcal{F}^\rightarrow(c)\]
and 
\[ \on{ev}_v\circ \indR_e\simeq \bigoplus_{h\in H(v)}\bigoplus_{c:e\overset{\circlearrowleft}{\shortrightarrow}h} \mathcal{F}(v\xrightarrow{h}e(h))^R\circ \mathcal{F}^\rightarrow(c)\,.\]
\end{enumerate}
\end{proposition}

\begin{proof}
The statements about $\indR_e$ follow from the statements for $\indL_e$ by passing to the opposite perverse schobers via \Cref{lem:dualschober}, so that we only need to prove the latter statements.

Parts (2) and (3) follow readily from \Cref{lem:evUh} and \Cref{lem:indLe}, by using that the pushout of functors valued in sections in \Cref{prop:eve} is computed pointwise in $\on{Exit}(\rgraph)$.

The proof of part (1) is similar to the proof of part (3) in \Cref{lem:evUh}
\end{proof}

\begin{lemma}\label{lem:constrofevvL}
Let $v$ be a vertex of $\rgraph$. Then $\on{ev}_v\colon \mathcal{H}(\rgraph,\mathcal{F})\to \mathcal{F}(v)$ admits a left adjoint $\indL_v$. Further, there is a pushout diagram in $\on{Fun}(\mathcal{F}(e),\mathcal{L}(\rgraph,\mathcal{F}))$
\[
\begin{tikzcd}[column sep=80]
\prod_{h\in H(v)}\iota_{e(h)}\circ \mathcal{F}(v\xrightarrow{h}e(h)) \arrow[r, "\prod_h \eta\circ \mathcal{F}(v\xrightarrow{h}e(h))"] \arrow[d] \arrow[rd, "\ulcorner", phantom] & \prod_{h\in H(v)} U_h^L\circ \mathcal{F}(v\xrightarrow{h}e(h)) \arrow[d] \\
\iota_v \arrow[r]                                                                                  & \iota\circ \indL_v                     
\end{tikzcd}
\]
where $\iota$ is the inclusion $\mathcal{H}(\rgraph,\mathcal{F})\subset \mathcal{L}(\rgraph,\mathcal{F})$.
\end{lemma}

\begin{proof}
Analogous to the proof of \Cref{lem:indLe}.
\end{proof}

\begin{proposition}\label{prop:evv}
Let $v$ be a vertex of $\rgraph$.
\begin{enumerate}[(1)]
\item The unit $\on{id}_{\mathcal{F}(v)}\to \on{ev}_v\indL_v$ is a split inclusion.
\item Let $f$ an edge of $\rgraph$. Then
\[ \on{ev}_f\circ \indL_v\simeq \bigoplus_{h\in H(v),\,c:h\overset{\circlearrowright}{\shortrightarrow}f}\mathcal{F}^\rightarrow(c)\circ \mathcal{F}(h)
\,.\]
\item Let $v'$ be a vertex of $\rgraph$. Then 
\[ \on{ev}_{v'}\circ \indL_v\simeq  \begin{cases} \on{id}_{\mathcal{F}(v)}\oplus \bigoplus_{h,h'\in H(v),c:h\overset{\circlearrowright}{\shortrightarrow}h',c\not= \ast} \mathcal{F}(h')^L\circ \mathcal{F}^\rightarrow(c)\circ \mathcal{F}(h) & \text{if }v'=v\\ \bigoplus_{h\in H(v),h'\in H(v'),c:h\overset{\circlearrowright}{\shortrightarrow}h'} \mathcal{F}(h')^L\circ \mathcal{F}^\rightarrow(c)\circ \mathcal{F}(h) & \text{if }v'\not=v\,.\end{cases}\]
In the upper sum above,  if $h=h'$, the trajectory $c$ is assumed not to be the constant trajectory $\ast$ at $h$.
\end{enumerate}
Similar results as above also apply to $\indR_v$, as stated in \Cref{introprop:splitting}.
\end{proposition}

\begin{proof}
Analogous to the proof of \Cref{prop:eve}.
\end{proof}

\subsection{Induction for subgraphs/subsurfaces}\label{subsec:restriction3}

We next discuss the induction and restriction functors arising from subgraphs of the parametrizing ribbon graphs. We begin by specifying what we mean by subgraphs:

\begin{definition}
We call a graph $\srgraph$ a subgraph of a graph $\rgraph$, and write $\srgraph\subset \rgraph$, if 
\begin{itemize}
\item there is an inclusion of the set of vertices $\srgraph_0\subset \rgraph_0$,
\item for each vertex $v\in \srgraph_0$, there is an equality $H_{\srgraph}(v)=H_{\rgraph}(v)$ of the sets of incident halfedges,
\item the map $\sigma\colon H_{\srgraph}\to \srgraph_0$ arises as the restriction of the map $\sigma\colon H_{\rgraph}\to \rgraph_0$,
\item any internal edge of $\srgraph$ is also an internal edge of $\rgraph$. 
\end{itemize} 
Note that $\srgraph$ can thus have external edges which arise from internal halfedges of $\rgraph$.
\end{definition}

Note that a subgraph of a ribbon graph inherits a conical ribbon graph structure. A subgraph $\srgraph\subset \rgraph$ comes with a functor between the exit path $\infty$-categories $\on{Exit}(\srgraph)\to \on{Exit}(\rgraph)$. Note that this functor is not necessarily injective on objects, this happens if an internal edge of $\rgraph$ is cut into two external edges of $\srgraph$.

\begin{remark}
Given a spanning graph $\rgraph$ of a marked surface ${\bf S}$ together with a subgraph of $\rgraph$, we obtain a subsurface of ${\bf S}$ by thickening the subgraph. 
\end{remark}

\begin{lemma}
Let $\srgraph\subset \rgraph$ be a subgraph and $i:\on{Exit}(\srgraph)\to \on{Exit}(\rgraph)$ the induced functor. If $\mathcal{F}\colon \on{Exit}(\rgraph)\to \on{St}$ is a $\rgraph$-parametrized perverse schober, then 
\[ \mathcal{F}_{\srgraph}\coloneqq i^*\mathcal{F}
\] 
is an $\srgraph$-parametrized perverse schober.
\end{lemma}

\begin{proof}
Immediate.
\end{proof}

Given a subgraph $\srgraph\subset \rgraph$, there are apparent functors
\[
\on{res}_{\srgraph}\colon \mathcal{H}(\rgraph,\mathcal{F})\to \mathcal{H}(\srgraph,\mathcal{F}_\srgraph)\] and 
\[ \on{res}_\srgraph^{\on{lax}}\colon \mathcal{L}(\rgraph,\mathcal{F})\to \mathcal{L}(\srgraph,\mathcal{F}_\srgraph)\,,\]
given by pulling back sections of the Grothendieck construction along $\on{Exit}(\srgraph)\to \on{Exit}(\rgraph)$. We show below that $\on{res}_{\srgraph}$ admits left and right adjoints $\indL_{\srgraph}$ and $\indR_{\srgraph}$.

The left adjoint $\iota_\srgraph$ of the functor $\on{res}_\srgraph^{\on{lax}}$ is given by left Kan extension and can concretely be described on objects as follows. Consider the functor between the Grothendieck constructions $\Gamma(\mathcal{F}_{\srgraph})\to \Gamma(\mathcal{F})$ arising from pulling back the coCartesian fibration $p\colon \Gamma(\mathcal{F})\to \on{Exit}(\rgraph)$ along the functor $\on{Exit}(\srgraph)\to \on{Exit}(\rgraph)$. Given a section $X\colon \on{Exit}(\srgraph)\to \Gamma(\mathcal{F}_{\srgraph})$, the section $\iota_\srgraph(X)\colon \on{Exit}(\rgraph)\to \Gamma(\mathcal{F})$ is obtained by extending $X$ on the complement of $\on{Exit}(\srgraph)$ in $\on{Exit}(\rgraph)$ with zero objects. If the functor $\on{Exit}(\srgraph)\to \on{Exit}(\rgraph)$ is not fully faithful, for each pair of two external edges $e',e''\in \srgraph_1$ that are mapped to the same internal edge $e\in \rgraph_1$, we further have $\iota_\srgraph(X)(e)\simeq X(e')\oplus X(e'')$. 

\begin{example}\label{ex:surface_vertex_induction}
Let $v\in \rgraph_0$ be a vertex. Then there exists a unique subgraph $\srgraph_v\subset \rgraph$ with $(\srgraph_v)_0=\{v\}$. Given a $\rgraph$-parametrized perverse schober $\mathcal{F}$, the $\srgraph_v$-perverse schober $\mathcal{F}_{\srgraph_v}$ satisfies 
\[
\mathcal{H}(\srgraph_v,\mathcal{F}_{\srgraph_v})\simeq \mathcal{F}(v)\,.
\]
Under this equivalence, the functor $\on{res}_{\srgraph_v}\colon \mathcal{H}(\rgraph,\mathcal{F})\to \mathcal{H}(\srgraph_v,\mathcal{F}_{\srgraph_v})$ identifies with $\on{ev}_v$, and the functor $\iota_\srgraph$ identifies with $\iota_v$.
\end{example}

Given a subgraph $\srgraph\subset \rgraph$, we denote by $E_{\srgraph}\subset \srgraph_1^\partial$ the subset of external edges that are mapped under the functor $\on{Exit}(\srgraph)\to \on{Exit}(\rgraph)$ to internal edges. We denote the unique halfedge of an external edge $e\in E_\srgraph$ by $h(e)$.

We fix a marked surface ${\bf S}$ with spanning graph $\rgraph$ and a $\rgraph$-parametrized perverse schober $\mathcal{F}$. As in \Cref{subsec:restriction2}, we obtain:

\begin{lemma}\label{lem:constrofresVL}
Let $\srgraph\subset \rgraph$ be a subgraph. Then $\on{res}_{\srgraph}$ admits a left adjoint $\indL_\rgraph$ and there is a pushout diagram in $\on{Fun}(\mathcal{H}(\srgraph,\mathcal{F}_\srgraph),\mathcal{L}(\rgraph,\mathcal{F}))$:
\[
\begin{tikzcd}[column sep=45pt]
\prod_{e\in E_\srgraph} \iota_e\circ \on{ev}_e \arrow[r, "\prod_{E_\srgraph} \eta\circ \on{ev}_e"] \arrow[d] \arrow[rd, "\ulcorner", phantom] & \prod_{e\in E_{\srgraph}} U_{h(e)}^L\circ \on{ev}_e \arrow[d] \\
\iota_\srgraph \arrow[r]                                               & \iota\circ \indL_{\srgraph}            
\end{tikzcd}
\]
\end{lemma}

\begin{proposition}\label{prop:ressurface}
Let $\srgraph\subset \rgraph$ be a subgraph. 
\begin{enumerate}[(1)]
\item The unit $\on{id}_{\mathcal{F}(v)}\to \on{res}_{\srgraph}\indL_\srgraph$ is a split inclusion and the counit $\on{res}_\srgraph\indR_\srgraph\to \on{id}_{\mathcal{H}(\srgraph,\mathcal{F}_\srgraph)}$ is a split projection.
\item Let $f$ an edge of $\rgraph$. If $f$ does not lie in the image of $\on{Exit}(\srgraph)\to \on{Exit}(\rgraph)$, then
\[ \on{ev}_f\circ \indL_\srgraph \simeq \bigoplus_{e\in E_\srgraph,c:h(e)\overset{\circlearrowright}{\shortrightarrow} f} \mathcal{F}^\rightarrow(c)\circ \on{ev}_e\,.\]
If $f$ lies in the image of $\on{Exit}(\srgraph)\to \on{Exit}(\rgraph)$, with unique preimage $f'$, then
\[ \on{ev}_f\circ \indL_\srgraph \simeq \on{ev}_{f'}\oplus \bigoplus_{f'\not= e\in E_\srgraph,c:h(e)\overset{\circlearrowright}{\shortrightarrow} f} \mathcal{F}^\rightarrow(c)\circ \on{ev}_e\,.\]\\
If $f$ lies in the image of $\on{Exit}(\srgraph)\to \on{Exit}(\rgraph)$, with two preimages $f',f''$, then $f',f''\in E_\srgraph$, and thus
\[ \on{ev}_f\circ \indL_\srgraph \simeq \bigoplus_{e\in E_\srgraph,c:h(e)\overset{\circlearrowright}{\shortrightarrow} f} \mathcal{F}^\rightarrow(c)\circ \on{ev}_e\,.\]
\item Let $v$ be a vertex of $\rgraph$. Then 
\[ \on{ev}_{v}\circ \indL_\srgraph\simeq  \begin{cases} \on{res}_{v}\oplus \bigoplus_{e\in E_\srgraph,h'\in H(v),c:h(e)\overset{\circlearrowright}{\shortrightarrow}h'} \mathcal{F}(v\xrightarrow{h'}e(h'))^L\circ \mathcal{F}^\rightarrow(c)\circ \on{ev}_e & \text{if }v\in \srgraph_0\\ \bigoplus_{e\in E_\srgraph,h'\in H(v),c:h(e)\overset{\circlearrowright}{\shortrightarrow}h'} \mathcal{F}(v\xrightarrow{h'}e(h'))^L\circ \mathcal{F}^\rightarrow(c)\circ \on{ev}_e & \text{if }v\not\in \srgraph_0\,.\end{cases}\]
\end{enumerate}
Similar results apply to $\indR_\srgraph$.
\end{proposition}

\begin{remark}
The results of \Cref{prop:ressurface} geometrically characterize when the functor $\indL_\srgraph\colon \mathcal{H}(\srgraph,\mathcal{F}_\srgraph)\to \mathcal{H}(\rgraph,\mathcal{F})$ is fully faithful. This is the case if and only if the clockwise trajectories starting at the boundary of the thickening ${\bf S}_{\srgraph}$ of $\srgraph$ in ${\bf S}$ do not pass into the interior of ${\bf S}_{\srgraph}$.
\end{remark}

\subsection{Induction for spanning graphs with \texorpdfstring{$1$}{1}-valent vertices}\label{subsec:1-valent}

\begin{lemma}\label{lem:induction_1_valent}
Let ${\bf S}$ be a marked surface with spanning graph $\rgraph$ and let $\mathcal{F}$ be a $\rgraph$-parametrized perverse schober.
\begin{enumerate}[(1)]
\item Let $e$ be an edge of $\rgraph$. Then the functor $\on{ev}_e\colon \mathcal{H}(\rgraph,\mathcal{F})\to \mathcal{F}(e)$ admits left and right adjoints $\indL_e,\indR_e$. 
\item Let $v$ be a vertex of $\rgraph$. Then the functor $\on{ev}_v\colon \mathcal{H}(\rgraph,\mathcal{F})\to \mathcal{F}(v)$ admits left and right adjoints $\indL_v,\indR_v$.
\item Let $\srgraph\subset \rgraph$ be a subgraph. Then the functor $\on{res}_{\srgraph}\colon \mathcal{H}(\rgraph,\mathcal{F})\to \mathcal{H}(\srgraph,\mathcal{F}_{\srgraph})$ admits left and right adjoints $\indL_{\srgraph},\indR_{\srgraph}$. 
\end{enumerate}
\end{lemma}

\begin{proof}
We only prove part (2), the proofs of part (1) and (3) are very similar.
Let $\tilde{\rgraph}$ be the ribbon graph obtained from $\rgraph$ by adding an external edge to each $1$-valent vertex. We denote by $E\subset \tilde{\rgraph}_1^\partial$ the subset of these added edges.
We can lift $\mathcal{F}$ to a $\tilde{\rgraph}$-parametrized perverse schober $\tilde{\mathcal{F}}$ with the property that for each $1$-valent vertex $v$ of $\rgraph$ with additional external edge $e$ in $\tilde{\rgraph}$, we have $\mathcal{F}(v)\simeq \on{fib}(\tilde{\mathcal{F}}(v)\to \tilde{\mathcal{F}}(e))$. The $\infty$-category of global sections $\mathcal{H}(\rgraph,\mathcal{F})\subset \mathcal{H}(\tilde{\rgraph},\tilde{\mathcal{F}})$ describes the full subcategory of global sections of $\tilde{\mathcal{F}}$ which evaluate trivially at each edge $e\in E$. The perverse schober $\tilde{\mathcal{F}}$ admits by \Cref{lem:indLe,lem:constrofevvL} left induction functors, which we denote by $\tilde{\on{ind}}_x^L$. 

Fix a vertex $v$ of $\rgraph$.  Let $h\in H_{\tilde{\rgraph}}(v)$. Consider the clockwise trajectory $\cwt_h$ (with respect to the line field $\nu_{\tilde{\rgraph}}$) starting at $h$. Denote by $t(h)$ the external edge of $\tilde{\rgraph}$ near which $\cwt_h$ ends. Consider the natural transformation 
$\alpha\colon \tilde{\on{ind}}^L_{t(e)}\circ \tilde{\mathcal{F}}^\rightarrow(\cwt_h)\circ \mathcal{F}(h)\to\tilde{\on{ind}}^L_v$ adjoint to the inclusion $\tilde{\mathcal{F}}^\rightarrow(\cwt_h)\circ \mathcal{F}(h)\subset \on{ev}_{t(e)}\circ \tilde{\on{ind}}^L_v$ from \Cref{prop:evv}. We pass to the cofiber $\on{cof}(\alpha)$ and distinguish two cases: the first case is that the trajectory $\cwt_{t(h)}$ ends at an external edge in $\tilde{\rgraph}_1^\partial \backslash E=\rgraph_1^\partial$. In this case, we proceed in the same way with the next halfedge of $v$. The second case is that the trajectory $\cwt_{t(h)}$ ends at an external edge $e_2$ in $E$. We then pass to the cofiber of the natural transformation 
\[\tilde{\on{ind}}^L_{e_2}\circ \tilde{\mathcal{F}}^\rightarrow(\cwt_{t(h)})\circ \tilde{\mathcal{F}}^\rightarrow(\cwt_h)\circ \mathcal{F}(h)[1]\to \on{cof}(\alpha)\,.\]
If $\cwt_{e_2}$ ends near an external edge in $E$, we again pass to the next cofiber, until eventually an external edge in $\rgraph_1^\partial$ is reached from the composite of clockwise trajectories. Note that this always happens after finitely many steps for topological reasons, the composite of these clockwise trajectories recovers the clockwise trajectory starting at $h$ with respect to the line field induced by $\rgraph$ (which loops around all $1$-valent vertices).  

After passing to cofibers as above for all halfedges $h\in H(v)$, the resulting functor $\tilde{\mathcal{F}}(v)\to  \mathcal{H}(\tilde{\rgraph},\tilde{\mathcal{F}})$ takes values in the full subcategory $\mathcal{H}(\rgraph,\mathcal{F})\subset \mathcal{H}(\tilde{\rgraph},\tilde{\mathcal{F}})$. If $v$ is $1$-valent, we further compose the functor with the inclusion $\mathcal{F}(v)\subset \tilde{\mathcal{F}}(v)$, denoting the result by $\indL_v$. By construction, the functor $\indL_v$ is left adjoint to a functor that restricts on $\mathcal{H}(\rgraph,\mathcal{F})$ to $\on{ev}_v$.  
\end{proof} 

\begin{remark}
The proof of \Cref{lem:induction_1_valent} shows that the support of objects in the image of $\indL_v$ is described by the clockwise web trajectory $\cwt_v$. While the splitting results from \Cref{prop:eve,prop:evv,prop:ressurface} can fail in the presence of $1$-valent vertices, they can be replaced by descriptions in terms of extensions of the given direct summands.
\end{remark}

\begin{example}
In the case where the marked surface ${\bf S}$ is the $1$-gon, the formalism of parametrized perverse schobers can be seen as an algebraic reformulation and generalization of the formalism of Fukaya--Seidel categories of \cite{Sei08}. As spanning ribbon graph, we consider the linear graph $\rgraph$:
\[
\begin{tikzcd}
\times \arrow[r, no head] & \times \arrow[r, no head] & \dots \arrow[r, no head] & \times \arrow[r, no head] & {}
\end{tikzcd}
\]
Thimbles then arise from induction from the vertices as follows: Let $\mathcal{F}$ be a $\rgraph$-parametrized perverse schober. Let $v$ be a $2$-valent vertex of $\rgraph$, $e^L$ the edge to its left and $e^R$ the edge to its right. Let $F\colon \V\to \N$ be the spherical functor underlying $\mathcal{F}$ at $v$. The case of Fukaya--Seidel categories corresponds to $\V=\D^{\on{perf}}(k)$ and $\N$ the perfect derived $\infty$-category of the wrapped Fukaya category of the fiber of the Lefschetz fibration. 

Then $\mathcal{F}(v)\simeq \V\times^{\rightarrow}_F\N$ is equivalent to the lax sum, containing the $\infty$-category of vanishing cycles $\V$ as a full subcategory. A vanishing cycle $C\in \V\subset \mathcal{F}(v)$ evaluates trivially at $e^L$ and as $F(C)$ at $e^R$. Thus, the support of the global section $\indL_v(C)$ is not described by the entire web trajectory $\cwt_v$, but only by the right-pointed part of the web trajectory. This amounts to the vanishing path that starts at $v$ and passes below the vertices to the right of $v$. The object $\indL_v(C)$ can be considered as the thimble associated with this vanishing path with vanishing cycle $C$. We depict these vanishing paths in \Cref{fig:thimbles}. Similarly, right induction gives rise to thimbles that pass along the top.

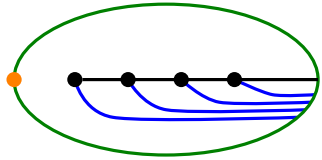
\begin{figure}[ht!]
\begin{center}
\begin{tikzpicture}
\draw[color=blue, very thick]  plot [smooth] coordinates {(-1.2,0) (-0.7,-0.5) (1.73,-0.5)};
\draw[color=blue, very thick]  plot [smooth] coordinates {(-0.5,0) (0,-0.4) (1.83,-0.4)};
\draw[color=blue, very thick]  plot [smooth] coordinates {(0.2,0) (0.7,-0.3) (1.9,-0.3)};
\draw[color=blue, very thick]  plot [smooth] coordinates {(0.9,0) (1.4,-0.2) (1.95,-0.2)};	
	
\draw[color=ao][very thick] (0,0) ellipse(2 and 1);

\node (4) at (-2,0){};
\node (a) at (-1.2,0){};
\node (c) at (-0.5,0){};
\node (d) at (0.2,0){};
\node (e) at (0.9,0){};

\fill[orange] (4) circle(0.1);

\fill (a) circle(0.1);
\fill (c) circle(0.1);
\fill (d) circle(0.1);
\fill (e) circle(0.1);

\draw[very thick] (-1.1,0)--(-0.5,0)--(0,0)--(0.5,0)--(1,0)--(2,0);

\end{tikzpicture}
\caption{The vanishing paths associated with thimbles arising from left induction are described by clockwise trajectories.}\label{fig:thimbles}
\end{center}
\end{figure}
\end{example}

\subsection{Global twisting and left and right induction}\label{subsec:globaltwist}

Let ${\bf S}$ be a marked surface with spanning graph $\rgraph$ and $\mathcal{F}$ a $\rgraph$-parametrized perverse schober. By \Cref{cor:globalsphadj}, the adjunction arising from boundary restriction of global sections
\[ \prod_{e\in \rgraph_1^\partial} \on{ev}_e\colon \glsec(\rgraph,\mathcal{F})\longleftrightarrow \prod_{e\in \rgraph_1^\partial} \mathcal{F}(e)\noloc  \prod_{e\in \rgraph_1^\partial}\indR_e\] 
is spherical. We denote by $C_{\bf S}\colon \glsec(\rgraph,\mathcal{F})\to \glsec(\rgraph,\mathcal{F})$ the corresponding twist functor. 

Given a vertex $v$ of $\rgraph$, we denote by $C_v\colon \mathcal{F}(v)\to \mathcal{F}(v)$ the twist functor of the spherical adjunction 
\[
\prod_{h\in H(v)} \mathcal{F}(v\xrightarrow{h}h(e))\colon \mathcal{F}(v)\longleftrightarrow \prod_{h\in H(v)} \mathcal{F}(e(h))\noloc  \prod_{h\in H(v)} \mathcal{F}(v\xrightarrow{h}h(e))^R\,.
\]

\begin{proposition}\label{prop:twist_indL_indR}~
\begin{enumerate}[(1)]
\item Let $e$ be an edge of $\rgraph$. Then there exists an equivalence of functors
\[
C_{\bf S}\circ \indL_e\simeq \indR_e\,.
\]
\item Let $v$ be a vertex of $\rgraph$. Then there exists an equivalence of functors
\[
C_{\bf S}\circ \indL_v\simeq \indR_e\circ C_v\,.
\]
\item Let $\srgraph\subset \rgraph$ be a subgraph. The thickening of $\srgraph$ defines a marked surface denoted ${\bf S}_{\srgraph}$.
Then there exists an equivalence of functors
\[
C_{\bf S}\circ \indL_{\srgraph}\simeq \indR_{\srgraph}\circ C_{{\bf S}_{\srgraph}}\,.
\]
\end{enumerate}
\end{proposition}

\begin{proof}
Part (1) follows from part (2) using the equivalences of functors $\indL_e\simeq \indL_v\circ \mathcal{F}(h)^L$ and $\indR_e\simeq \indR_v\circ \mathcal{F}(h)^R$ for any edge $e$ with halfedge $h\in e$ lying at a vertex $v$. Part (2) is in turn a special case of part (3) by \Cref{ex:surface_vertex_induction}. 

We turn to the proof of part (3). To simplify the discussion, we will restrict ourselves to the case that $\rgraph$ has no $1$-valent vertices. The general case follows from a similar argument. The functor $\indR_\srgraph\circ C_{{\bf S}_{\srgraph}}$ appears as the cofiber of the right vertical morphism in the following diagram with horizontal fiber and cofiber sequences:
\[
\begin{tikzcd}
\prod_{e\in E_{\srgraph}}\iota_e\circ \on{ev}_e \arrow[r] \arrow[d] & \prod_{e\in E_{\srgraph}}U_{h(e)}^R\circ \on{ev}_e \oplus \iota_{\srgraph} \arrow[r] \arrow[d, "\alpha"] & \indR_{\srgraph} \arrow[d]                                                                \\
0 \arrow[r]                                                        & \prod_{e\in \srgraph_1^\partial} \indR_{e}\circ \on{ev}_e \arrow[r, "\simeq"]                   & {\indR_{\srgraph}\circ \prod_{e\in \srgraph_1^\partial} \indR_{e,\srgraph}\circ \on{ev}_e}
\end{tikzcd}
\]
Here and below, we denote for $e\in E_\srgraph$  by $h(e)\in e$ its unique halfedge. The functor $U_{h}^R$ denotes the variant of the functor $U_{h}^L$ obtained by performing \Cref{constr:Uh} in the opposite perverse schober. The functor $U_{h}^R$ is thus the right adjoint to $U_{h}$. The functor $\indR_{e,\srgraph}$ denotes the right induction functor from $e$ to global sections of $\mathcal{F}_{\srgraph}$. 

Similarly, the functor $T_{\bf S} \indL_{\srgraph}$ appears as the cofiber of the right vertical morphism in the following diagram with horizontal fiber and cofiber sequences:
\[
\begin{tikzcd}[column sep=5]
\underset{e\in E_{\srgraph}}{\prod}\iota_e\circ \on{ev}_e \arrow[r] \arrow[d] & \underset{e\in E_{\srgraph}}{\prod}U_{h(e)}^L\circ \on{ev}_e \oplus \iota_{\srgraph} \arrow[r] \arrow[d, "\beta"]                                                                                                        & \indL_{\srgraph} \arrow[d]                                                     \\
0 \arrow[r]                                                         & \underset{e\in E_{\srgraph}}{\prod} \indR_{t(e)}\circ \mathcal{F}^{\rightarrow}(\cwt_{h(e)})\circ \on{ev}_e \oplus \underset{e\in H_1^\partial \backslash E_{\srgraph}}{\prod} \indR_e\circ \on{ev}_e \arrow[r, "\simeq"] & \underset{e\in \rgraph_1^\partial}{\prod} \indR_{e}\circ \on{ev}_e\circ \indL_{\srgraph}
\end{tikzcd}
\]
We denote for $e\in E_{\srgraph}$ by $t(e)\in \rgraph_1^\partial$ the external edge where $\cwt_{h(e)}$ ends. The lower equivalence follows from	 the observation that for an external edge $f\in \rgraph_1^\partial$, we have 
\[ \on{ev}_f\circ \indL_{\srgraph}\simeq \begin{cases} \bigoplus_{e\in E_{\srgraph}, t(e)=f} \mathcal{F}^{\rightarrow}(\cwt_{h(e)})\circ \on{ev}_e & f\not \in \srgraph_1^\partial \\ 
\on{ev}_f\oplus \bigoplus_{e\in E_{\srgraph}, t(e)=f} \mathcal{F}^{\rightarrow}(\cwt_{h(e)})\circ \on{ev}_e & f\in \srgraph_1^\partial\,.
\end{cases}
\,.\]

Consider the splitting
\[ \prod_{e\in \srgraph_1^\partial} \indR_{e}\circ \on{ev}_e\simeq \prod_{e\in E_{\srgraph}} \indR_{e}\circ \on{ev}_e\oplus \prod_{e\in \srgraph_1^\partial\backslash E_{\srgraph}} \indR_{e}\circ \on{ev}_e\,.\] 
With respect to this splitting, the morphism $\alpha$ is of an upper triangular form, containing the apparent inclusion of sections $i\colon \prod_{e\in E_{\srgraph}}U_{h(e)}^R\circ \on{ev}_e \hookrightarrow \prod_{e\in E_{\srgraph}} \indR_{e}\circ \on{ev}_e$. Similarly, the morphism $\beta$ is also upper triangular and restricts on the first summands to the apparent inclusion $i'\colon \prod_{e\in E_{\srgraph}}U_{h(e)}^L\circ \on{ev}_e\hookrightarrow \prod_{e\in E_{\srgraph}} \indR_{t(e)}\circ \mathcal{F}^{\rightarrow}(\cwt_{h(e)})\circ \on{ev}_e$. Unraveling the constructions of these functors, we observe that the cofibers of $i$ and $i'$ are equivalent, they can also be described as the cofiber of the inclusion $\iota_e\hookrightarrow U^R_{\tau(h(e))}$, where $\tau(h(e))$ is the halfedge in $\rgraph$ opposite to $h(e)$. The lower right components of the upper triangular morphisms $\alpha$ and $\beta$ agree, we denote the corresponding morphim by $j\colon i_{\srgraph}\to \prod_{e\in H_1^\partial \backslash E_{\srgraph}} \indR_e\circ \on{ev}_e$. We obtain fiber and cofiber sequences
\[
\on{cof}(j)[-1]\to \on{cof}(i)\to \on{cof}(\alpha)
\]
and 
\[
\on{cof}(j)[-1]\to \on{cof}(i')\to \on{cof}(\beta)\,.
\]
Under the equivalence $\on{cof}(i)\simeq \on{cof}(i')$, the two first morphisms above are identified, they both arise from the morphism $\iota_{\srgraph}\to U^R_{\tau(h(e))}$. We thus find $\on{cof}(\alpha)\simeq \on{cof}(\beta)$. The induced inclusions of section of $\prod_{e\in E_{\srgraph}}\iota_e\circ \on{ev}_e[1]$ into both of these cofibers also agree, hence so do the cofibers along these inclusions. In total, this shows the desired equivalence $C_{\bf S}\circ \indL_{\srgraph}\simeq \indR_{\srgraph}\circ C_{{\bf S}_{\srgraph}}$.
\end{proof}

\section{Gluing cluster tilting subcategories}\label{subsec:gluingCTO}

The goal of this section is to prove the following Theorem, stating that a collection of cluster tilting subcategories of a perverse schober glues to a global cluster tilting subcategory. For this, we use the $\infty$-categorical Frobenius exact structure of the $\infty$-category of global sections of a perverse schober $\mathcal{H}(\rgraph,\mathcal{F})$ induced by the spherical functor $\prod_{e\in \rgraph_1^\partial} \on{ev}_e\colon \glsec(\rgraph,\mathcal{F})\rightarrow \prod_{e\in \rgraph_1^\partial} \mathcal{F}(e)$, see \Cref{prop:Frob2CYexact} and \Cref{cor:globalsphadj}. Throughout the section, we will assume that the spanning graph $\rgraph$ has no $1$-valent vertices, this assumption is needed to have the splitting results of \Cref{prop:eve,prop:evv,prop:ressurface}.

\begin{theorem}\label{thm:clustertiltingschober}
Let $\rgraph$ be a spanning ribbon graph of a marked surface without $1$-valent vertices. Let $\mathcal{F}$ be a $k$-linear $\rgraph$-parametrized perverse schober.  For each vertex $v$ with incident halfedges $h_1,\dots,h_n$ and edges $e_1,\dots,e_n$, consider $\mathcal{F}(v)$ as equipped with the $\infty$-categorical Frobenius exact structure induced by the spherical functor 
\[ \prod_{1\leq i\leq n}\mathcal{F}(v\xrightarrow{h_i}e_i) \colon \mathcal{F}(v)\to \prod_{1\leq i\leq n}\mathcal{F}(e_i)\,.\]
Given for every vertex $v\in \rgraph_0$ a cluster tilting subcategory $\mathcal{T}_v\subset \mathcal{F}(v)$, the additive closure of
\[ \bigcup_{v\in \rgraph_0} \indL_v(\mathcal{T}_v)\] 
is a cluster tilting subcategory of the Frobenius exact $\infty$-category $\mathcal{H}(\rgraph,\mathcal{F})$.
\end{theorem}

\begin{remark}
Applying \Cref{thm:clustertiltingschober} to the opposite perverse schober $\mathcal{F}^{\on{op}}$ from \Cref{lem:dualschober}, we find that $\bigcup_{v\in \rgraph_0} \indR_v(\mathcal{T}_v)$ is also a cluster tilting subcategory of $\mathcal{H}(\rgraph,\mathcal{F})$. Using \Cref{prop:twist_indL_indR}, we find that in the stable $\infty$-category of the Frobenius exact $\infty$-category $\mathcal{H}(\rgraph,\mathcal{F})$, see \Cref{rem:stablecategory}, the two cluster tilting subcategories $\bigcup_{v\in \rgraph_0} \indL_v(\mathcal{T}_v)$ and $\bigcup_{v\in \rgraph_0} \indR_v(C_v(\mathcal{T}_v))$ differ by a shift. In favorable situations, the corresponding clusters of a corresponding cluster algebra thus differ by the action of the Donaldson--Thomas transformation.
\end{remark}

We split the proof of \Cref{thm:clustertiltingschober} in two parts, discussed in \Cref{subsec:2-termglue,subsec:rigidglue}.

\subsection{Gluing rigid objects}\label{subsec:rigidglue}

The goal of this section is to prove the following.

\begin{proposition}\label{prop:rigidglue} 
Let ${\bf S}$ be a marked surface. Let $\rgraph$ be a spanning graph of ${\bf S}$ without $1$-valent vertices and $\mathcal{F}$ a $\rgraph$-parametrized perverse schober. For each vertex $v$ with incident halfedges $h_1,\dots,h_n$ and corresponding edges $e_1,\dots,e_n$, consider $\mathcal{F}(v)$ as equipped with the $\infty$-categorical Frobenius exact structure induced by the spherical functor 
\[ \prod_{1\leq i\leq n}\mathcal{F}(v\xrightarrow{h_i}e_i) \colon \mathcal{F}(v)\to \prod_{1\leq i\leq n}\mathcal{F}(e_i)\,.\]
Given a collection of rigid additive subcategories $\{\mathcal{T}_v\subset \mathcal{F}(v)\}_{v\in \rgraph_0}$, the additive closure of 
\[ \bigcup_{v\in \rgraph_0} \indL_v(\mathcal{T}_v)\subset \mathcal{H}(\rgraph,\mathcal{F})\] 
is also rigid with respect to the $\infty$-categorical exact structure induced by $\prod_{e\in \rgraph_1^\partial} \on{ev}_e$.
\end{proposition}

\begin{lemma}\label{lem:extensions_between_indutions_1}
Let ${\bf S}$ be a marked surface with  a spanning graph $\rgraph$ without $1$-valent vertices. Let $\mathcal{F}$ be a $\rgraph$-parametrized perverse schober. Let $T_v,T_v'\in \mathcal{F}(v)$ for some vertex $v\in \rgraph_0$. The functor $\indL_v$ induces an isomorphism
\[ \on{Ext}^{1,\on{ex}}_{\mathcal{F}(v)}(T_v,T_v')\simeq \on{Ext}^{1,\on{ex}}_{\mathcal{H}(\rgraph,\mathcal{F})}(\indL_v(T_v),\indL_v(T_v'))\]
between the abelian groups of exact extensions. In particular, $T_v$ is rigid if and only if $\indL_v(T_v)$ is rigid.
\end{lemma}

\begin{proof}
By adjunction, we have 
\begin{equation}\label{eq:adjindev}  \on{Ext}^1(\indL_{v}(T_v),\indL_{v}(T_v'))\simeq \on{Ext}^1(T_v,\on{ev}_{v}\indL_{v}(T_v'))\,.\end{equation}
By \Cref{prop:evv}, $\on{ev}_{v}\indL_{v}(T_v')$ splits into the direct sum of $T_v'$ and injective-projective objects in $\mathcal{F}(v)$, which we denote by $P_1,\dots,P_m$. Thus 
\[
\on{Ext}^1(\indL_{v}(T_v),\indL_{v}(T_v'))\simeq \on{Ext}^1(T_v,T_v')\oplus \bigoplus_{i=1}^m \on{Ext}^1(T_v,P_i)\,.
\]
 Let $P_l$ be any such indecomposable injective-projective direct summand, it evaluates non-trivially at exactly two consecutive edges $e_{i_l-1},e_{i_l}$ (in the counterclockwise orientation) incident to $v$ (as follows from properties (2).(a)-(c) of \Cref{def:schobernspider}). We denote $\on{res}_{e_{i_l}}=\mathcal{F}(v\xrightarrow{h_{i_l}}e_{i_l})$ for brevity. 

Note that $P_l\simeq \on{res}_{e_{i_l}}^R\circ \on{res}_{e_{i_l}}(P_l)$. We thus have 
\[ \on{Ext}^1_{\mathcal{F}(v)}(T_v,P_l)\simeq \on{Ext}^{1}_{\mathcal{F}(e_{i_l})}(\on{res}_{e_{i_l}}(T_v),\on{res}_{e_{i_l}}(P_l))\,.\]

The (inverse of the) equivalence \eqref{eq:adjindev} arises from applying $\on{ind}_{v}^L$ and composing with the counit $\eta_{v}$ of $\on{ind}_{v}^L\dashv \on{res}_{v}$. Inspecting the construction of $\on{ind}_{v}^L$, one finds canonical pointwise split inclusions of sections of the Grothendieck construction (which extend to natural transformations)
\begin{equation}\label{eq:incl1}
U_{h_{i_l}}^L(P_l)\hookrightarrow \on{ind}_{v}^L(T_v')\,,
\end{equation}
\begin{equation}\label{eq:incl2}
U_{h_{i_l}}^L(P_l)\hookrightarrow  \on{ind}_{v}^L(P_l)\,,
\end{equation}
and
\[
U_{h_{i_l}}^L\circ \on{res}_{e_{i_l}}(T_v)\hookrightarrow \on{ind}_{v}^L(T_v)\,.
\]
One readily sees that the map $\indL_v(P_l)\subset \on{ind}_{v}^L\circ \on{res}_{v}\circ \on{ind}_{v}^L(T_v')\xrightarrow{\eta_v} \on{ind}_{v}^L(T_v')$ arising from the counit restricts to an equivalence on the images of the pointwise inclusions \eqref{eq:incl1} and \eqref{eq:incl2}.

Consider an extension $\alpha_l\colon \indL_{v}(T_v)\to \indL_{v}(T_v')[1]$ arising as the image of 
\[
\beta_l\in \on{Ext}^{1}_{\mathcal{F}(e_{i_l})}(\on{res}_{e_{i_l}}(T_v),\on{res}_{e_{i_l}}(P_l))\subset \on{Ext}^1(\indL_{v}(T_v),\indL_{v}(T_v'))\,.
\]
By the above, there is a commutative diagram of sections of the Grothendieck construction, with horizontal pointwise split inclusions, as follows:
\[
\begin{tikzcd}
{U_{h_{i_l}}^L\circ \on{res}_{e_{i_l}}(T_v)} \arrow[r, hook] \arrow[d,"{U_{h_{i_l}}^L(\beta_l)}"] & \indL_{v}(T_v) \arrow[d, "\alpha_l"] \\
{U_{h_{i_l}}^L\circ \on{res}_{e_{i_l}}(P_l)[1]} \arrow[r, hook]          & {\indL_{v}(T_v')[1]}                      
\end{tikzcd}
\]
Let $e_{B_l}$ be the external edge of $\rgraph$ at which the trajectory $\cwt_{h_{i_l}}$ ends. Then $\on{ev}_{e_{B_l}}(\alpha_l)\simeq \on{ev}_{e_{B_l}}U_{h_{i_l}}^L(\beta_l)\simeq \mathcal{F}^\rightarrow(\cwt_{h_{i_l}})(\beta_l)$. Since the transport $\mathcal{F}^\rightarrow(\cwt_{h_{i_l}})$ is an equivalence and $\alpha_l$ is exact, we see that $\alpha_l$ is exact if and only if $\beta_l\simeq 0$. 

A similar argument shows that an extension $\alpha_0\colon \indL_v(T_v)\to \indL_v(T_v')[1]$ arising as the image of 
\[
\beta_0\in \on{Ext}^{1}_{\mathcal{F}(v)}(T_v,T_v')\subset \on{Ext}^1(\indL_{v}(T_v),\indL_{v}(T_v'))
\]
is exact if and only if $\beta_0$ is exact with respect to the exact structure of $\mathcal{F}(v)$.

To conclude the argument, we observe that an arbitrary morphism 
\[ \alpha=\sum_{0\leq l\leq m}\alpha_l\colon \indL_{v}(T_v)\to \indL_{v}(T_v')\,,\] 
split as a sum of morphisms as above, is exact if and only if $\alpha_l$ is exact for all $0\leq l\leq m$. For this, it suffices to note the splitting
\begin{align*}
\bigoplus_{e\in \rgraph_1^\partial} \on{Ext}^1(\on{res}_e\indL_v(T_v),\on{res}_e\indL_v(T_v')) \simeq &\bigoplus_{h\in H(v)}\!\! \on{Ext}^1(\mathcal{F}^\rightarrow(\gamma^\circlearrowright_{h})\circ \mathcal{F}(h)(T_v),\mathcal{F}^\rightarrow(\gamma^\circlearrowright_{h})\circ\mathcal{F}(h)(T_v'))\\
&\oplus \bigoplus_{l=1}^n \on{Ext}^1(\mathcal{F}^\rightarrow(\gamma^\circlearrowright_{h_{i_l}})\circ \mathcal{F}(h_{i_l})(T_v),\mathcal{F}^\rightarrow(\gamma^\circlearrowright_{h_{i_l}})(P_l))\,,
\end{align*}
which implies that the boundary evaluation of $\alpha$ vanishes if and only if the boundary evaluation of each $\alpha_l$ vanishes. 
\end{proof}

\begin{lemma}\label{lem:extensions_between_indutions_2}
Let ${\bf S}$ be a marked surface with  a spanning graph $\rgraph$ without $1$-valent vertices. Let $\mathcal{F}$ be a $\rgraph$-parametrized perverse schober. Let $v,v'\in \rgraph_0$ be two distinct vertices, $T_v\in \mathcal{F}(v)$ and $T_{v'}\in \mathcal{F}(v')$. Then 
\[ \on{Ext}^{1,\on{ex}}_{\mathcal{H}(\rgraph,\mathcal{F})}(\indL_v(T_v),\indL_{v'}(T_{v'}))\simeq 0\,.\]
\end{lemma}

\begin{proof}
Analogous to the proof of \Cref{lem:extensions_between_indutions_1}.
\end{proof}

\begin{proof}[Proof of \Cref{prop:rigidglue}]
Combine \Cref{lem:extensions_between_indutions_1} and \Cref{lem:extensions_between_indutions_2}.
\end{proof}

\begin{remark}
Note that the marked surface in \Cref{thm:clustertiltingschober} is assumed to have a marked point on each boundary component. Without this assumptions, objects in $\on{ind}_v^L(\mathcal{T}_v)$ arise from curves of infinite length, see \Cref{rem:infiniteinduction}. The rigidity statement of \Cref{lem:extensions_between_indutions_1} typically fails in this more general setting. The gluing of $2$-term resolutions as in \Cref{prop:glue2termres} however works the same. 
\end{remark}

\subsection{Gluing \texorpdfstring{$2$}{2}-term resolutions}\label{subsec:2-termglue}

We prove the gluing of $2$-term resolutions by an iterative argument. We first show that right $2$-term resolutions glue along right induction, or equivalently (via passage to opposite $\infty$-categories) that the left $2$-term resolution property glues along left induction. The case of gluing left $2$-term resolutions along right induction (or right $2$-term resolutions along left induction) does not seem to follow from the same argument, and we will deduce this using \Cref{prop:twist_indL_indR}.

Consider a commutative diagram in $\on{St}$ of the following form
\begin{equation}\label{eq:pbdiag}
\begin{tikzcd}
\D \arrow[r, "j_1"] \arrow[d, "j_2"'] \arrow[rd, "k"] & \D_1 \arrow[d, "i_1"] \arrow[r, "f_{1}"] & \C_1 \\
\D_2 \arrow[d, "f_{2}"] \arrow[r, "i_2"']                                               & \C_2                       &      \\
\C_3                                                                   &                          &     
\end{tikzcd}
\end{equation}
such that the square is pullback, and such that $i_i,j_i,k_i$ admit right adjoints $i_i^R,j_i^R,k_i^R$. We will further assume that the $\infty$-categorical exact structures on $\D_1,\D_2,\D$ induced by the functors $\D_i\to \C_i\times \C_{i+1}$, for $i=1,2$, and $\D\to \C_1\times \C_3$, see \Cref{ex:inducedexstr}, are Frobenius. This follows for instance if these functors are spherical. 

The following Lemma is the central ingredient for the gluing of $2$-term resolutions.

\begin{lemma}\label{lem:unitsquare}
Under the above assumptions, the units of the adjunctions assemble into a biCartesian square in $\on{Fun}(\D,\D)$:
\[
\begin{tikzcd}
\on{id}_\D \arrow[r] \arrow[d] \arrow[rd, "\square", phantom] & j_1^Rj_1 \arrow[d] \\
j_2^Rj_2 \arrow[r]                                            & k^Rk              
\end{tikzcd}
\]
\end{lemma}

\begin{proof}
See the proof of \cite[Lem.~3.10]{Chr23}.
\end{proof}

For $i=1,2$, let $\mathcal{T}_i\subset \D_i$ be an additive subcategory containing the essential image of $i_i^R$. We define $\T\coloneqq j_1(\mathcal{T}_1)\cup  j_2(\mathcal{T}_2)\subset \D$.

\begin{lemma}\label{lem:gluing2termresolutions}
Suppose that for $i=1,2$ and $X\in \D_i$, the image under $f_i$ of the counit map $j_ij_i^R(X)\to X$ is a split deflation. If $\T_1\subset \D_1$ and $\T_2\subset\D_2$ satisfy the right $2$-term resolution property, then $\T\subset \D$ also satisfies the right $2$-term resolution property.
\end{lemma}

\begin{proof}[Proof of \Cref{lem:gluing2termresolutions}]
Let $X\in \D$ and pick exact fiber and cofiber sequences for $i=1,2$
\begin{equation}\label{eq:coresolution} j_i(X)\to T_0^i\to T_1^i\,,\end{equation}
with $T_0^i,T_1^i\in \T_i$. Note that the exactness amounts to the vanishing of the images of the morphisms $T_1^i[-1]\to j_i(X)$ under $i_i$ and $f_i$. Rotating these fiber and cofiber sequences, applying $j_i^R$ and taking their direct sum yields the middle horizontal fiber and cofiber sequence in the following diagram:
\begin{equation}\label{eq:3x3diag}
\begin{tikzcd}
{\bigoplus_{i=1,2} j_i^R(T_0^i)[-1]} \arrow[r] \arrow[d, "\on{id}"] & {\bigoplus_{i=1,2} j_i^R(T_1^i)[-1] \oplus k^Rk(X)[-1]} \arrow[r] \arrow[d] & X \arrow[d]                               \\
{\bigoplus_{i=1,2} j_i^R(T_0^i)[-1]} \arrow[r] \arrow[d] & {\bigoplus_{i=1,2} j_i^R(T_1^i)[-1]} \arrow[r] \arrow[d, "0"]               & {\bigoplus_{i=1,2} j_i^Rj_i(X)} \arrow[d] \\	
0 \arrow[r]                                        & k^Rk(X) \arrow[r, "\on{id}"]                                                     & k^Rk(X)                                  
\end{tikzcd}
\end{equation}
The rightmost vertical sequence arises from the units of the adjunctions $j_i\dashv j_i^R$, $k\dashv k^R$ and is a fiber and cofiber sequence by \Cref{lem:unitsquare}. The naturality of the unit of the adjunction $i_i\dashv i_i^R$ implies that the following diagram commutes.
\[
\begin{tikzcd}
{j_i^R(T_1^i)[-1]} \arrow[d] \arrow[r] & j_i^Rj_i(X) \arrow[d]                              &         \\
{j_i^Ri_i^Ri_i(T_1^i)[-1]} \arrow[r]   & j_i^Ri_i^Ri_ij_i(X) \arrow[r, Rightarrow, no head] & k^Rk(X)
\end{tikzcd}
\]
The bottom arrow in this diagram vanishes by the exactness of the fiber and cofiber sequences \eqref{eq:coresolution}. This shows that the arising vertical morphism $\bigoplus_{i=1,2} j_i^R(T_0^i)\to k^Rk(X)$ in the diagram \eqref{eq:3x3diag} vanishes. Since we thus find that all the other sequences are fiber and cofiber sequences, the top horizontal sequence in the above diagram is also one.

Rotating the top fiber and cofiber sequence yields the fiber and cofiber sequence
\begin{equation}\label{eq:X2termres}
X\longrightarrow \bigoplus_{i=1,2} j_i^R(T_0^i) \longrightarrow \bigoplus_{i=1,2} j_i^R(T_1^i) \oplus k^Rk(X)\,.
\end{equation}
It remains to show that the above sequence is exact, for which we prove that the first morphism is an inflation. Fix $l\in \mathbb{Z}/2\mathbb{Z}$. The morphism $\beta\colon X\to j_l(T_0^l)$ contained in the first morphism above factors by the commutativity of the diagram \eqref{eq:3x3diag} as $X\xrightarrow{\on{unit}} j_l^Rj_l(X) \xrightarrow{j_l^R(\alpha)} j_l^R(T_0^l)$, where the second morphism is the image under $j_l^R$ of the inflation $\alpha\colon j_l(X)\to T_0^l$ appearing in \eqref{eq:coresolution}.

The image of the morphism $\beta\colon X\to j_l^R(T_0^l)$ under $f_l\circ j_l$ composes by the commutativity of the following diagram
\[
\begin{tikzcd}[column sep=huge]
f_lj_l(X) \arrow[r, "f_lj_l\circ \on{unit}"']\arrow[rd, "{f_lj_l(\beta)}"', bend right=10] \arrow[rr, "\on{id}", bend left=15] & f_lj_lj_l^Rj_l(X) \arrow[r, "f_l\circ \on{cu}\circ j_l"'] \arrow[d, "f_lj_lj_l^R(\alpha)"'] & f_lj_l(X) \arrow[d, "f_l(\alpha)"] \\
                                                                                   & f_lj_lj_l^R(T_0^l) \arrow[r, "f_l\circ \on{cu}"]                                                & f_l(T_0^l)                          
\end{tikzcd}
\]
with the counit $f_l\circ \on{cu}\colon f_lj_lj_l^R(T_0^l)\to f_l(T_0^l)$ to the split inflation $f_l(\alpha)\colon f_lj_l(X)\to f_l(T_0^l)$ in $\C_l$. By the assumption that $f_l\circ \on{cu}\colon f_lj_lj_l^R(T_0^l)\to f_l(T_0^l)$ is a split deflation, we see that $f_lj_l(X)\to f_lj_lj_l^R(T_0^l)$ is also a split inflation, hence so is 
$f_lj_l(X)\to f_lj_lj_l^R(T_0^l)\oplus f_lj_lj_{l+1}^R(T_0^{l+1})$. This shows that $X\to \bigoplus_{i=1,2} j_i^R(T_0^i)$ is an inflation, as desired.
\end{proof}

\begin{proposition}\label{prop:glue2termres}
Let ${\bf S}$ be a marked surface with  a spanning graph $\rgraph$ without $1$-valent vertices. Let $\mathcal{F}$ be a $\rgraph$-parametrized perverse schober.
\begin{enumerate}[(1)] 
\item Given additive subcategories $\mathcal{T}_v\subset \mathcal{F}(v)$ for all $v\in \rgraph_0$ with the left $2$-term resolution property, the additive closure of 
\[ \bigcup_{v\in \rgraph_0} \indL_v(\mathcal{T}_v)\] 
in $\mathcal{H}(\rgraph,\mathcal{F})$ also has the left $2$-term resolution property. 
\item Given additive subcategories $\mathcal{T}_v\subset \mathcal{F}(v)$ for all $v\in \rgraph_0$ with the right $2$-term resolution property, the additive closure of 
\[ \bigcup_{v\in \rgraph_0} \indL_v(\mathcal{T}_v)\] 
in $\mathcal{H}(\rgraph,\mathcal{F})$ also has the right $2$-term resolution property. 
\end{enumerate}
Similar statements also hold for right induction.
\end{proposition}

\begin{proof}
For part (1), we repeatedly apply \Cref{lem:gluing2termresolutions}, gluing in one vertex of $\rgraph$ at a time. The assumptions for \Cref{lem:gluing2termresolutions} are satisfied by part (1) of \Cref{prop:ressurface}. 

We turn to the proof of part (2). The twist autoequivalence $C_{\mathcal{H}(\rgraph,\mathcal{F})}\colon \mathcal{H}(\rgraph,\mathcal{F})\to \mathcal{H}(\rgraph,\mathcal{F})$ of the spherical adjunction 
\[ \prod_{e\in \rgraph_1^\partial} \on{ev}_e\colon \glsec(\rgraph,\mathcal{F})\longleftrightarrow \prod_{e\in \rgraph_1^\partial} \mathcal{F}(e)\noloc \prod_{e\in \rgraph_1^\partial} \indR_e\] defines an exact functor with respect to the $\infty$-categorical exact structure of $\mathcal{H}(\rgraph,\mathcal{F})$ induced by $\prod_{e\in \rgraph_1^\partial} \on{ev}_e$: this simply follows from the fact that (co)twist functors commute with their spherical functors, see \cite[Lem.~2.2]{Chr20}. Similarly, for every vertex $v$ of $\rgraph$, the autoequivalence $C_v\colon\mathcal{F}(v)\to \mathcal{F}(v)$ is an exact functor. An exact autoequivalence preserves the right $2$-term resolution property of additive subcategories. The opposite categorical version of part (1) shows that the additive hull of $\bigcup_{v\in \rgraph_0} \indR_v\circ C_v^{-1}(\mathcal{T}_v)$ has the right $2$-term resolution property, hence so does the additive hull of $\bigcup_{v\in \rgraph_0} C_{\bf S}\circ \indR_v\circ C_v^{-1}(\mathcal{T}_v)\simeq \bigcup_{v\in \rgraph_0} \indL_v(\mathcal{T}_v)$, see \Cref{prop:twist_indL_indR}.
\end{proof}

Finally, we note that \Cref{thm:clustertiltingschober} is a direct consequence of the above: 

\begin{proof}[Proof of \Cref{thm:clustertiltingschober}.]
Combine \Cref{prop:glue2termres,prop:rigidglue}.
\end{proof}

\subsection{Ice quivers from cluster tilting objects and amalgamation}\label{subsec:amalgamation}

In this section, we describe the ice quivers of the cluster tilting objects constructed in \Cref{thm:clustertiltingschober} via the amalgamation of ice quivers. Recall that an ice quiver consists of a quiver $Q$ together with a (not necessarily full) subquiver $F\subset Q$, consisting of so-called frozen vertices and frozen arrows between frozen vertices.

We fix a marked surface ${\bf S}$ with a spanning graph $\rgraph$ without $1$-valent vertices. Consider a $\rgraph$-parametrized perverse schober $\mathcal{F}$ taking values in $k$-linear stable $\infty$-categories for a field $k$. Suppose we are given for all vertices $v\in \rgraph_0$ a cluster tilting object $T_v\in \mathcal{F}(v)$ whose discrete endomorphism algebra $\on{Ext}^0(T_v,T_v)$ is finite dimensional over $k$. The additive hull $\mathcal{T}_v=\on{Add}(T_v)$ is the corresponding cluster tilting subcategory. 

For each vertex $v$, we denote by $I_v\subset T_v$ the maximal injective-projective summand in $\mathcal{F}(v)$ and by $T_v^\circ$ its complement. We denote for each edge $e$ by $T_e\in \mathcal{F}(e)$ the sum of the indecomposable objects. 

By \Cref{thm:clustertiltingschober}, there is a cluster tilting object in the $\infty$-category of global sections, given by the maximal basic direct summand of $\bigoplus_{v\in \rgraph_0}\on{ind}_v^L(T_v)\in \mathcal{H}(\rgraph,\mathcal{F})$. Consider the summand
\[ 
T^{\on{ind}}=\bigoplus_{v\in \rgraph_0}\on{ind}_v^L(T_v^\circ)\oplus \bigoplus_{e\in \rgraph_1} \on{ind}_e^L(T_e)\subset \bigoplus_{v\in \rgraph_0}\on{ind}_v^L(T_v)\in \mathcal{H}(\rgraph,\mathcal{F})\,.
\]
The summand $T^{\on{ind}}$ is basic unless there are different edges $e,e'$ for which $\on{ind}_e^L(T_e),\on{ind}_{e'}^L(T_{e'})$ share a non-zero summand. We note that this cannot happen if $\mathcal{F}$ has only vertices of valency $3$ or higher (due to the zero relation in \Cref{def:schobernspider}.(2).(c)). We will assume in the following discussion that $T^{\on{ind}}$ is basic, and hence cluster tilting.

Any finite dimensional algebra gives rise to a quiver. Given a vertex $v$, we denote the quiver of the cluster tilting object $T_v$, meaning of the finite dimenisonal algebra $\on{Ext}^0(T_v,T_v)$, by $Q_{T_v}$. Recall that the indecomposable summands of $T_v$ give rise to the vertices of $Q_{T_v}$ and arrows arise from irreducible morphisms, roughly meaning morphisms which are not the composite of two non-trivial morphisms (counted in a suitable $k$-linear sense). The quiver does not encode the relations between composites of irreducible morphisms.  

We now define a frozen subquiver of $Q_{T_v}$, turning $Q_{T_v}$ into an ice quiver.
We declare a vertex of $Q_{T_v}$ to be frozen if the corresponding indecomposable object is injective-projective. Let $e_1,\dots,e_m$ be the edges incident to $v$. Note that injective-projective objects in $\mathcal{F}(v)$ are exactly images of the indecomposable objects under the functor $\prod_{i=1}^m\mathcal{F}(v\to e_i)^L\colon \prod_{i=1}^m\mathcal{F}(e_i)\to \mathcal{F}(v)$. An arrow between two frozen vertices is declared frozen if the corresponding morphism lies in the image of the functor $\prod_{i=1}^m\mathcal{F}(v\to e_i)^L$. Since this functor is componentwise fully faithful, we see that given two frozen vertices, either all or none of their connecting arrows are frozen. There are $m$ connected component of the frozen quiver, corresponds to the $m$ edges incident to $v$.

There is a similar ice quiver associated with $T^{\on{ind}}$, and ice quiver $Q_{T_e}$ associated with $T_e$, for any $e\in \rgraph_1$, where all vertices and arrows in $Q_{T_e}$ are declared to be frozen.

There can be irreducible endomorphisms of $T_{e_i}$ which do not extend to irreducible endomorphisms of $T_v$, see for instance \Cref{ex:A2}. Hence, the quiver $Q_{T_{e_i}}$ of $T_{e_i}$ is not necessarily a subquiver of the $Q_{T_v}$. Instead, the functor $\mathcal{F}(v\to e_i)^L$ gives rise to a morphism of quivers $Q_{T_{e_i}}\to Q_{T_v}$ (in the sense below) that is injective on vertices and full but not faithful on arrows. The image of $Q_{T_{e_i}}\to Q_{T_v}$ is the frozen component of $Q_{T_v}$ corresponding to $e_i$. \\

We next define the amalgamation ice quiver $\tilde{Q}$. Let $\on{Quiv}$ be the $1$-category of quivers, meaning 
\begin{itemize}
\item the objects of $\on{Quiv}$ are quivers, which we can consider as additive $1$-categories with finitely many objects (the vertices) and morphisms freely generated by a finite collection of morphisms (the arrows).
\item morphisms of quivers can be defined as functors between the corresponding $1$-categories that map arrows to arrows or zero morphisms. 
\end{itemize}
The quiver underlying $\tilde{Q}$ is defined as the colimit of the diagram of quivers 
\[
\on{Exit}(\rgraph)^{\on{op}}\longrightarrow \on{Quiv},\quad  x\mapsto Q_{T_x}
\]
that maps an incidence $e\to v$ to the morphism $Q_{T_e}\to Q_{T_v}$ described above. Concretely, this colimit can be described as follows.
\begin{itemize}
\item The set of vertices of $\tilde{Q}$ is the union of the sets of vertices the ice quivers $Q_{T_v}$, $v\in \rgraph_0$, where for each internal edge $e$ incident to vertices $v_1,v_2$, we identify the vertices lying the images of the morphisms $Q_{T_e}\to Q_{T_{v_1}},Q_{T_{v_2}}$. 
\item Each non-frozen arrow of $Q_{T_v}$ yields an arrow in $\tilde{Q}$. 
\item Each frozen arrow of $Q_{T_v}$ lying in a frozen component corresponding to an external edge of $\rgraph$ yields an arrow in $\tilde{Q}$.
\item Let $e$ be an internal edge incident to vertices $v_1,v_2$. An arrow in $Q_{T_e}$ yields an arrow in $\tilde{Q}$ if the images under $Q_{T_{e}}\to Q_{T_{v_1}},Q_{T_{v_2}}$ are non-zero.
\end{itemize}
The frozen vertices and morphisms of $\tilde{Q}$ are declared to be those which lie in the image of a morphism $Q_{T_e}\to \tilde{Q}$ from the colimit diagram, with $e$ an external edge of $\rgraph$. 

\begin{proposition}\label{prop:amalgamationCTO}
The ice quiver of the cluster tilting object $T^{\on{ind}}\in \mathcal{H}(\rgraph,\mathcal{F})$ is given by the amalgamation ice quiver $\tilde{Q}$.
\end{proposition}

\begin{proof} 
It is apparent that the vertices of the ice quiver of $T^{\on{ind}}$ match the vertices of $\tilde{Q}$. We must thus identify the irreducible endomorphisms of $T^{\on{ind}}$ with the arrows in $\tilde{Q}$.

Consider two vertices $v_1,v_2$ of $\rgraph$. We first consider the case $v_1\not =v_2$. All morphisms $\on{ind}_{v_1}^L(T_{v_1}^\circ)\to\on{ind}_{v_2}^L(T_{v_2}^\circ)$ factor through a sum of objects of the form $\on{ind}_{e}^L(T_e)$ with $e$ incident to $v_2$, which can be seen as follows: The object $\on{ev}_{v_1}\indL_{v_2}(T_{v_2}^\circ)$ splits into the direct sum of injective-projective objects $P_1,\dots,P_m\in \mathcal{F}(v)$, where each $P_i$ is equivalent to $\mathcal{F}(v\to e)^L\circ \on{res}_e(P_i)$ for some edge $e$ incident to $v$. The inclusion $\on{res}_e(P_i)\subset \on{res}_e\on{ind}_{v_2}^L(T_2^\circ)$ induces a morphism $\on{ind}_{v_1}^L(P_i)\to \on{ind}_{v_2}^L(T_2^\circ)$, through which any morphism in 
\[ \on{Ext}^1(T_1^\circ,P_i)\subset \on{Ext}^1(\on{ind}_{v_1}^L(T_1^\circ),\on{ind}_{v_2}^L(T_2^\circ))\]
factors. There are hence no irreducible morphisms $\on{ind}_{v_1}^L(T_{v_1}^\circ)\to\on{ind}_{v_2}^L(T_{v_2}^\circ)$. 

Consider now the case $v_1=v_2$. By \Cref{prop:evv}, we have an inclusion 
\[ \on{Ext}^0(T_{v_1}^\circ,T_{v_1}^\circ)\subset \on{Ext}^0(\on{ind}_{v_1}^L(T_{v_1}^\circ),\on{ind}_{v_1}^L(T_{v_1}^\circ))\,.\] 
A similar argument as above shows that the irreducible endomorphisms of $\on{ind}_{v_1}^L(T_{v_1}^\circ)$ which extend to irreducible morphisms of $T^{\on{ind}}$ arise from irreducible endomorphisms of the summand $\on{Ext}^0(T_{v_1}^\circ,T_{v_1}^\circ)$.

Let $v\in \rgraph_0$ and $e\in \rgraph_1$. If $e$ is not incident to $v$, then no morphisms $\on{ind}_v^LT_v^\circ\to \on{ind}_e^L(T_e)$ or $\on{ind}_e^L(T_e)\to \on{ind}_v^LT_v^\circ$ extend to irreducible endomorphisms of $T^{\on{ind}}$. If $e$ is incident to $v$, then the irreducible morphisms contributing to the ice quiver of $T^{\on{ind}}$ are exactly those arising from irreducible morphisms $T_v^\circ \to \mathcal{F}(v\to e)^L(T_e)$ or $\mathcal{F}(v\to e)^L(T_e)\to T_v^\circ$ in $\mathcal{F}(v)$. A similar argument applies to morphisms $\on{ind}_{e_1}^L(T_{e_1})\to \on{ind}_{e_2}^L(T_{e_2})$ for $e_1,e_2\in \rgraph_1$ with $e_1\not= e_2$. 

It remains to describe which endomorphisms of $T_e$, for $e\in \rgraph_1$, extend to irreducible endomorphisms of $T^{\on{ind}}$. Note that there is by \Cref{prop:eve} an inclusion 
\[ \on{Ext}^0(T_e,T_e)\subset \on{Ext}^0(\on{ind}_e^L(T_e),\on{ind}_e^L(T_e))\,.\]
Let $v_1,v_2$ be the two vertices incident to $e$ and consider the quiver morphisms $Q_{T_{e}}\to Q_{T_{v_1}},Q_{T_{v_2}}$. Consider an arrow in $Q_{T_e}$ with corresponding endomorphism $\alpha\colon T_e\to T_e$. If its image in $Q_{T_{v_1}}$ or $Q_{T_{v_2}}$ vanishes, the morphism $\alpha$ becomes reducible in $\mathcal{F}(v_1)$ or $\mathcal{F}(v_2)$, and thus also induces a reducible endomorphism of $T^{\on{ind}}$. If its image in $Q_{T_{v_i}}$, $i=1,2$, does not vanish, then $\mathcal{F}(v_i\to e)^L(\alpha)$ induces an irreducible endomorphism of $T_{v_i}$. We claim that if this is the case for both $i=1,2$, then $\on{ind}_{e}^L(\alpha)$ induces an irreducible endomorphism of $T^{\on{ind}}$. First, suppose that $\on{ind}_e^L(\alpha)$ factors through an object $X=X_1\oplus X_2$ with $X_1$ in the additive hull of $\on{ind}_{v_1}^L(T_{v_1})$ and $X_2$ in the additive hull of $\on{ind}_{v_2}^L(T_{v_2})$. Then restricting the global sections to $\mathcal{F}(v_1)$ yields a factorization of $\mathcal{F}(v_1\to e)^L(\alpha)$ through an object in the additive hull of $T_{v_1}$. This factorization must be trivial by assumption that the frozen arrow of $Q_{T_e}$ persists in $Q_{T_{v_1}}$. It thus follows that $\on{res}_{v_1}(X)$ lies in the additive hull of $\mathcal{F}(v_1\to e)^L(T_e)$, hence $X_1$ lies in the additive hull of $\on{ind}_e^L(T_e)$. A similar argument shows that $X_2$ lies in the additive hull of $\on{ind}_e^L(T_e)$. It then follows that the factorization already exists in $\mathcal{F}(e)$, which contradicts the assumption that $\alpha$ is irreducible. 
 
Morphisms from $\on{ind}_e^L(T_e)$ to summands of $T^{\on{ind}}$ arising from vertices other than $v_1,v_2$ always factor through the additive hull of $\on{ind}_{v_1}^L(T_{v_1})\oplus \on{ind}_{v_1}^L(T_{v_2})$, so that we may restrict to the above case. 
\end{proof}

\begin{example}\label{ex:A2}
Let $\Pi_2(A_2)$ be the $2$-Calabi--Yau completion of the $A_2$-quiver over a field $k$, i.e.~the $2$-Calabi--Yau Ginzburg algebra, and 
\[ \C_{A_2}\coloneqq \on{CoSing}(\Pi_2(A_2))=\D^{\on{perf}}(\Pi_2(A_2))/\D^{\on{fin}}(\Pi_2(A_2))\] 
the cosingularity category. Then $\C_{A_2}$ has two indecomposable objects up to equivalence, which we denote by $1,2$. Up to $k$-linear multiple, there are unique non-zero morphisms $1\to 2$ and $2\to 1$. The stable $\infty$-category $\C_{A_2}$ is $2$-periodic and there are fiber and cofiber sequences $1\to 2\to 1$ and $2\to 1 \to 2$. 

The trivial spherical adjunction $0\leftrightarrow \C_{A_2}$ gives rise to a perverse schober $\mathcal{F}$ parametrized by the $3$-spider, with vertex $v$ and incident external edges $e_1,e_2,e_3$, with
\[
\mathcal{F}(v)\simeq \on{Fun}(\Delta^1,\C_{A_2}))
\]
and 
\[
\mathcal{F}(e_i)\simeq \C_{A_2}\,.
\]

As one can check with a direct computation, see also \cite{KL25} for a more general result, that a cluster tilting object $T_v$ in $\mathcal{F}(v)$ is given by the direct sum of the injective-projective objects
\[
1\to 0,\quad 2\to 0,\quad 1\simeq 1,\quad 2\simeq 2,\quad 0\to 2,\quad 0\to 2
\]
and the object 
\[ 
1\to 2\,.
\]
We can depict the summands of this cluster tilting object in terms of decorated curve and web trajectories (in blue) in the $3$-gon (with the $3$-spider in black), where boundary intersections correspond to non-trivial restrictions of these objects to $\mathcal{F}(e_i)$, $i=1,2,3$.

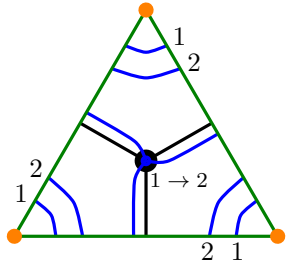
\begin{center}
\begin{tikzpicture}[rotate=-30]

\fill (0,0) circle (0.15);
\foreach \n in {0,1,2}
            {	          	
\draw[color=blue, very thick] plot [smooth] coordinates {(360/3*\n-10:1.55) (360/3*\n:1.44) (360/3*\n+10:1.55)};
\draw[color=blue, very thick] plot [smooth] coordinates {(360/3*\n-20:1.3) (360/3*\n:1.1) (360/3*\n+20:1.3)};

\draw[very thick] (0,0)--(360/3*\n+60:1);
            }
            
            \foreach \n in {0,1,2}
            {	
            \draw[color=blue, very thick] plot [smooth] coordinates {(0,0) (360/3*\n+30:0.3) (360/3*\n+50.6:1)};
            }

\draw[color=blue, very thick] (0,0) circle (0.05);

\foreach \n in {0,1,2}
            {	
            	\coordinate (x) at (360/3*\n:2);
            	\coordinate (y) at (360/3*\n+120:2);
            	\draw[color=ao, very thick] (x)--(y);
            }
            
            \foreach \n in {0,1,2}
            {	
            	\coordinate (x) at (360/3*\n:2);
            	\coordinate (y) at (360/3*\n+120:2);
            	\fill[color=orange] (x) circle (0.1);
            }
\node() at (105:1.7){$1$};
\node() at (94:1.45){$2$};
\node() at (225:1.7){$1$};
\node() at (214:1.45){$2$};
\node() at (345:1.7){$1$};
\node() at (334:1.45){$2$};
\node() at (0.5,0){\scriptsize $1\to 2$};
\end{tikzpicture}
\end{center}

The ice quiver of $T_v$ can be depicted as follows:
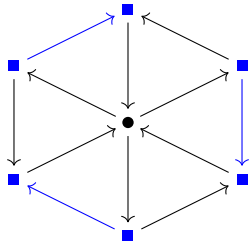
\begin{center}

\begin{tikzpicture}[scale=1.5]
\draw node[frvertex] (1) at (0,2) {};
\draw node[frvertex] (3) at (-1,1.5) {};
\draw node[frvertex] (2) at (1,1.5) {};
\draw node[vertex] (4) at (0,1) {};
\draw node[frvertex] (7) at (-1,0.5) {};
\draw node[frvertex] (5) at (1,0.5) {};
\draw node[frvertex] (8) at (0,0) {};

\draw[->] (2)--(1);
\draw[->] (1)--(4);
\draw[->] (4)--(2); 
\draw[->] (4)--(8);
\draw[->] (7)--(4);
\draw[->] (5)--(4);
\draw[->] (4)--(3);
\draw[->] (3)--(7);
\draw[->, blue] (8)--(7);
\draw[->] (8)--(5);
\draw[->, blue] (3)--(1);
\draw[->, blue] (2)--(5);
\end{tikzpicture}
\end{center}
The ice quiver of $1\oplus 2\in \C_{A_2}$ is given by 
\begin{center}
\begin{tikzpicture}[scale=1.5]
\draw node[frvertex] (1) at (-0.5,2) {};
\draw node[frvertex] (2) at (0.5,2) {};
\draw[->, blue] (2) to [bend right=25] (1);
\draw[->, blue]  (1) to [bend right=25] (2);
\end{tikzpicture}
\end{center}
The morphism $Q_{T_{e_i}}\to Q_{T_v}$ thus forgets one of the two frozen arrows. In the endomorphism algebras, the frozen arrow is sent to the composite of two non-frozen arrows.

Next, we glue two of the above triangle to a $4$-gon. We obtain a cluster tilting object in $\on{Fun}(\Delta^1,\C)\times_{\C}\on{Fun}(\Delta^1,\C)\simeq \on{Fun}(\Delta^2,\C)$ by left induction, which we can graphically represent as follows (omitting the decorations). 

\begin{figure}[h!]
\begin{center}
\begin{tikzpicture}[scale=0.8]
            
\begin{scope}[yscale=1.5]
\foreach \n in {0,1,2}
            {	          	
\draw[color=blue, very thick] plot [smooth] coordinates {(360/3*\n-10:1.56) (360/3*\n:1.45) (360/3*\n+10:1.56)};
\draw[color=blue, very thick] plot [smooth] coordinates {(360/3*\n-20:1.31) (360/3*\n:1.05) (360/3*\n+20:1.31)};

\draw[very thick] (0,0)--(360/3*\n+60:1);
            }
            
            \foreach \n in {0,1,2}
            {	
            \            \draw[color=blue, very thick] plot [smooth] coordinates {(0,0) (360/3*\n+30:0.3) (360/3*\n+50.6:1)};
            }

\foreach \n in {0,2}
            {	
            	\coordinate (x) at (360/3*\n:2);
            	\coordinate (y) at (360/3*\n+120:2);
            	\draw[color=ao, very thick] (x)--(y);
            }

\draw[color=blue, very thick] plot [smooth] coordinates {(-171.2:1.05) (-140:0.5) (280:1.07)};
\end{scope}

\begin{scope}[xshift=-2cm, rotate=180, yscale=1.5]
\foreach \n in {0,1,2}
            {	          	
\draw[color=blue, very thick] plot [smooth] coordinates {(360/3*\n-10:1.55) (360/3*\n:1.45) (360/3*\n+10:1.55)};
\draw[color=blue, very thick] plot [smooth] coordinates {(360/3*\n-20:1.3) (360/3*\n:1.05) (360/3*\n+20:1.3)};

\draw[very thick] (0,0)--(360/3*\n+60:1.01);
            }
            
            \foreach \n in {0,1,2}
            {	
                        \draw[color=blue, very thick] plot [smooth] coordinates {(0,0) (360/3*\n+30:0.3) (360/3*\n+50.6:1)};
            }

\foreach \n in {0,2}
            {	
            	\coordinate (x) at (360/3*\n:2);
            	\coordinate (y) at (360/3*\n+120:2);
            	\draw[color=ao, very thick] (x)--(y);
            }

\draw[color=blue, very thick] plot [smooth] coordinates {(-171.2:1.05) (-140:0.5) (280:1.07)};
\end{scope}

\fill (0,0) circle (0.15);
\draw[color=blue, very thick] (0,0) circle (0.05);
\begin{scope}[xshift=-2cm, rotate=180]
\fill (0,0) circle (0.15);
\draw[color=blue, very thick] (0,0) circle (0.05);
\end{scope}

\fill[color=orange] (2,0) circle (0.1);
\fill[color=orange] (-4,0) circle (0.1);
\fill[color=orange] (-1,2.55) circle (0.1);
\fill[color=orange] (-1,-2.55) circle (0.1);

\end{tikzpicture}\caption{A collection of curve and web trajectories in the $4$-gon corresponding to a cluster tilting object in $\on{Fun}(\Delta^2,\on{CoSing}(\Pi_2(A_2)))$.}\label{fig:induction}
\end{center} 
\end{figure}
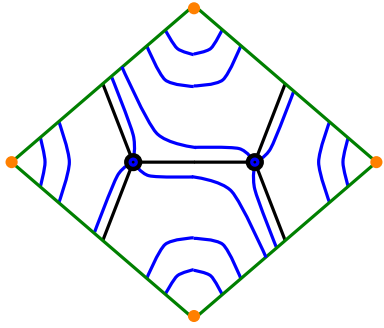

The corresponding ice quiver arises from amalgamation and is given as follows:

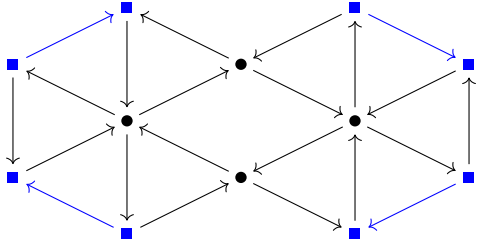
\begin{center}
\begin{tikzpicture}[scale=1.5]
\draw node[frvertex] (1) at (0,2) {};
\draw node[frvertex] (3) at (-1,1.5) {};
\draw node[vertex] (2) at (1,1.5) {};
\draw node[vertex] (4) at (0,1) {};
\draw node[frvertex] (7) at (-1,0.5) {};
\draw node[vertex] (5) at (1,0.5) {};
\draw node[frvertex] (8) at (0,0) {};

\draw node [frvertex] (10) at (2,2){};
\draw node [frvertex] (11) at (3,1.5){};
\draw node [vertex] (12) at (2,1){};

\draw node[frvertex] (9) at (3,0.5) {};
\draw node[frvertex] (6) at (2,0) {};

\draw[->] (2)--(1);
\draw[->] (1)--(4);
\draw[->] (4)--(2); 
\draw[->] (4)--(8);
\draw[->] (7)--(4);
\draw[->] (5)--(4);
\draw[->] (4)--(3);
\draw[->] (3)--(7);
\draw[->, blue] (8)--(7);
\draw[->] (8)--(5);
\draw[->, blue] (3)--(1);
\draw[->] (2)--(12);
\draw[->] (12)--(5);
\draw[->] (10)--(2);
\draw[->, blue] (10)--(11);
\draw[->] (12)-- (10);
\draw[->] (11)--(12);
\draw[->] (12)--(9);
\draw[->] (6)--(12);
\draw[->]  (9)--(11);
\draw[->, blue] (9)--(6);
\draw[->] (5)--(6);
\end{tikzpicture}
\end{center}
\end{example}

\subsection{Example: marked surfaces with punctures} \label{subsec:puncturedsurf}

A marked surface with punctures (or also punctured marked surface) $({\bf S},M,P)$ consists of a marked surface $({\bf S},M)$ together with a set of points $P$ in the interior of ${\bf S}$ (the punctures). An arc in a punctured marked surface ${\bf S}$ is an embedded curve $\gamma\colon [0,1]\to {\bf S}\backslash M$ in ${\bf S}$ with endpoints in $\partial {\bf S}\backslash M$ or $P$, considered up to homotopies relative $\partial {\bf S}\backslash M$ and $P$. Arcs are assumed to be non-constant. A tagging of an arc refers to an additional binary labeling of the endpoints at the punctures, subject to certain conditions. A tagged triangulation is a maximal collection of compatible tagged arcs. For every punctured marked surface, there is an associated cluster algebra, whose clusters are in bijection with tagged triangulations. Choosing the boundary arcs as coefficients (i.e.~as frozen cluster variables), the tagged arcs become in bijection with the cluster variables of the cluster algebra. We refer to \cite{FST08,FT12} for background. 

Additive categorifications of cluster algebras of punctured marked surfaces without coefficients are well understood. They arise from the quivers equipped with the non-degenerate potentials defined in \cite{Lab09,GLS16} and were studied for instance in \cite{QZ17,AP21}. 

In the case that the coefficients are the boundary arcs, \cite{Wu25} proves that the corresponding Higgs category arises via the $\mathbb{Z}/2$-quotient of the Higgs category of the unpunctured marked double cover of the punctured surface. 

In this section, we discuss how to construct additive categorifications of the cluster algebras of marked surfaces with coefficients the boundary arcs using perverse schobers. We expect the resulting exact $\infty$-categories of global sections to coincide with the Higgs categories, we hope to return to this question in future work.\\

We begin with the basic building block, given by a once-punctured $2$-gon. For $i\in \mathbb{Z}$ let $k[t_i^\pm]$ be the graded Laurent algebra with generator in degree $|t_i|=i$. Consider the morphism of dg algebra $\phi\colon k[t_2^\pm]\xrightarrow{t_2\mapsto t_1^2} k[t_1^\pm]$.

\begin{lemma}
The functor $\phi_!\colon \D^{\on{perf}}(k[t_2^\pm])\to \D^{\on{perf}}(k[t_1^\pm])$ is spherical. 
\end{lemma}

\begin{proof}
Let $\phi^*$ be the right adjoint of $\phi_!$. The twist functor of this adjunction is equivalent to $[1]$. The functor $\phi^*$ is monadic and $[1]$ commutes with the unit of the monad, hence the adjunction is spherical by \cite[Prop.~4.5]{Chr20}.
\end{proof}

We next consider a perverse schober $\mathcal{F}$ parametrized by the $2$-spider $\rgraph_2$ with underlying spherical functor $\phi_!$. Let $v$ be the unique vertex of $\rgraph_2$. Then $\mathcal{F}(v)\simeq \mathcal{V}^2_{\psi_!}=\D^{\on{perf}}(k[t_2^\pm])\times^{\rightarrow}_{\psi_!}\D^{\on{perf}}(k[t_1^\pm])$ is given by the lax sum, whose objects are triples $(a,b,\eta)$ with $a\in \D^{\on{perf}}(k[t_2^\pm]$, $b\in\D^{\on{perf}}(k[t_1^\pm])$ and $\eta\colon \phi_!(a)\to b$.  

There are six indecomposable objects of $\mathcal{V}^2_{\psi_!}$ up to equivalence. We list them, identifying each object with a tagged arc in the once-punctured $2$-gon:

\begin{itemize}
\item The two objects 
\[ (k[t_2^\pm],0,0),\quad (k[t_2^\pm][1],0,0)\]
corresponding to the tagged arcs
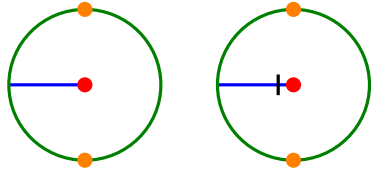
\begin{center}\begin{tikzpicture}[anchor=base, baseline]
\draw[color=blue, very thick] (-1,0)--(0,0);
\fill[color=red] (0,0) circle (0.1);
\draw[color=ao, very thick] (0,0) circle (1);
\fill[color=orange] (0,1) circle (0.1);
\fill[color=orange] (0,-1) circle (0.1);
\end{tikzpicture}\quad\quad
\begin{tikzpicture}[anchor=base, baseline]
\draw[color=blue, very thick, decoration={markings, mark=at position 0.82 with {\arrow[black]{|}}}, postaction={decorate}] (-1,0)--(0,0);
\fill[color=red] (0,0) circle (0.1);
\draw[color=ao, very thick] (0,0) circle (1);
\fill[color=orange] (0,1) circle (0.1);
\fill[color=orange] (0,-1) circle (0.1);
\end{tikzpicture}
\end{center}
\item The two objects 
\[\ (k[t_2^\pm],k[t_1^\pm],\phi_!(k[t_2^\pm])\simeq k[t_1^\pm]),\quad (k[t_2^\pm][1],k[t_1^\pm],\phi_!(k[t_2^\pm][1])\simeq k[t_1^\pm]) \]
corresponding to the tagged arcs
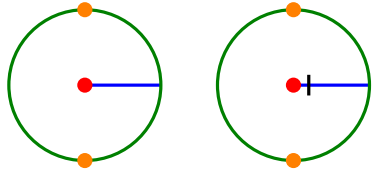
\begin{center}\begin{tikzpicture}[anchor=base, baseline]
\draw[color=blue, very thick] (1,0)--(0,0);
\fill[color=red] (0,0) circle (0.1);
\draw[color=ao, very thick] (0,0) circle (1);
\fill[color=orange] (0,1) circle (0.1);
\fill[color=orange] (0,-1) circle (0.1);
\end{tikzpicture}\quad\quad
\begin{tikzpicture}[anchor=base, baseline]
\draw[color=blue, very thick, decoration={markings, mark=at position 0.82 with {\arrow[black]{|}}}, postaction={decorate}] (1,0)--(0,0);
\fill[color=red] (0,0) circle (0.1);
\draw[color=ao, very thick] (0,0) circle (1);
\fill[color=orange] (0,1) circle (0.1);
\fill[color=orange] (0,-1) circle (0.1);
\end{tikzpicture}
\end{center}
\item The two objects
\[ (0,k[t_1^\pm],0),\ (k[t_2^\pm]\oplus k[t_2^\pm][1],k[t_1^\pm],\on{counit}_{\phi_!\dashv \phi^*})\]
corresponding to the (trivially tagged) boundary arcs 
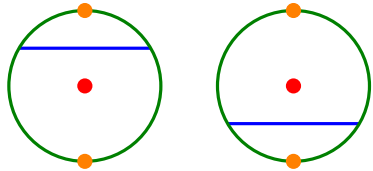
\begin{center}\begin{tikzpicture}[anchor=base, baseline]
\draw[color=blue, very thick] (150:1)--(30:1);
\fill[color=red] (0,0) circle (0.1);
\draw[color=ao, very thick] (0,0) circle (1);
\fill[color=orange] (0,1) circle (0.1);
\fill[color=orange] (0,-1) circle (0.1);
\end{tikzpicture}\quad\quad
\begin{tikzpicture}[anchor=base, baseline]
\draw[color=blue, very thick] (210:1)--(330:1);
\fill[color=red] (0,0) circle (0.1);
\draw[color=ao, very thick] (0,0) circle (1);
\fill[color=orange] (0,1) circle (0.1);
\fill[color=orange] (0,-1) circle (0.1);
\end{tikzpicture}
\end{center}
\end{itemize}

Let $e_1,e_2$ be the two external edges of $\rgraph_2$. Up to equivalence of perverse schobers, we have that $\mathcal{F}(v\to e_1)\simeq \varrho_1$ and $\mathcal{F}(v\to e_2)\simeq \varrho_2$ with $\varrho_1,\varrho_2\colon \mathcal{V}^2_{\psi_!}\to \D^{\on{perf}}(k[t_1^\pm])$ the apparent functors given on an object $(a,b,\eta)\in \mathcal{V}^2_{\psi_!}$by
\[
\varrho_1(a,b,\eta)=b
\]
and 
\[
\varrho_2(a,b,\eta)=\on{cof}(\eta)\,.
\]

\begin{lemma}\label{lem:CTO2-gon}
Consider $\mathcal{V}^2_{\psi_!}$ as equipped with the $\infty$-categorical Frobenius exact structure, induced by 
\[(\varrho_1,\varrho_2)\colon \mathcal{V}^2_{\psi_!}\to \D^{\on{perf}}(k[t_1^\pm])^{\times 2}\,.\]  
Then there is a canonical bijection between cluster tilting objects in $\mathcal{V}^2_{\psi_!}$ and tagged triangulation of the once-punctured $2$-gon. 
\end{lemma}

\begin{proof}
We first list the four possible tagged triangulations of the once-punctured $2$-gon. Recall that two distinct tagged arcs with coinciding endpoints can appear in a tagged triangulation if their endpoints are tagged identically, whereas identical arcs intersecting the boundary must have different taggings. 

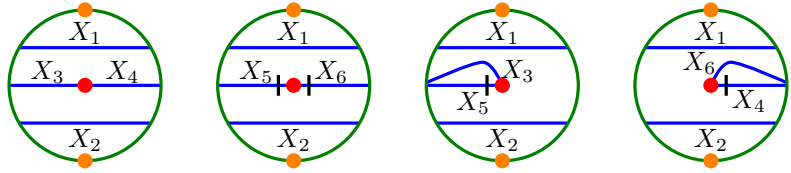
\begin{center}
\begin{tikzpicture}[anchor=base, baseline]
\draw[color=blue, very thick] (-1,0)--(0,0);
\draw[color=blue, very thick] (1,0)--(0,0);
\draw[color=blue, very thick] (210:1)--(330:1);
\draw[color=blue, very thick] (150:1)--(30:1);
\fill[color=red] (0,0) circle (0.1);
\draw[color=ao, very thick] (0,0) circle (1);
\fill[color=orange] (0,1) circle (0.1);
\fill[color=orange] (0,-1) circle (0.1);
\node () at (170:0.5){$X_3$};
\node () at (10:0.5){$X_4$};
\node () at (90:0.6){$X_1$};
\node () at (270:0.82){$X_2$};
\end{tikzpicture}\quad\quad
\begin{tikzpicture}[anchor=base, baseline]
\draw[color=blue, very thick, decoration={markings, mark=at position 0.82 with {\arrow[black]{|}}}, postaction={decorate}] (-1,0)--(0,0);
\draw[color=blue, very thick, decoration={markings, mark=at position 0.82 with {\arrow[black]{|}}}, postaction={decorate}] (1,0)--(0,0);
\draw[color=blue, very thick] (210:1)--(330:1);
\draw[color=blue, very thick] (150:1)--(30:1);
\fill[color=red] (0,0) circle (0.1);
\draw[color=ao, very thick] (0,0) circle (1);
\fill[color=orange] (0,1) circle (0.1);
\fill[color=orange] (0,-1) circle (0.1);
\node () at (170:0.5){$X_5$};
\node () at (10:0.5){$X_6$};
\node () at (90:0.6){$X_1$};
\node () at (270:0.82){$X_2$};
\end{tikzpicture}\quad\quad
\begin{tikzpicture}[anchor=base, baseline]
\draw[color=blue, very thick] plot [smooth] coordinates {(178:1) (130:0.4) (0,0)};
\draw[color=blue, very thick, decoration={markings, mark=at position 0.82 with {\arrow[black]{|}}}, postaction={decorate}] (180:1)--(0,0);
\draw[color=blue, very thick] (210:1)--(330:1);
\draw[color=blue, very thick] (150:1)--(30:1);
\fill[color=red] (0,0) circle (0.1);
\draw[color=ao, very thick] (0,0) circle (1);
\fill[color=orange] (0,1) circle (0.1);
\fill[color=orange] (0,-1) circle (0.1);
\node () at (220:0.5){$X_5$};
\node () at (0.2,0.1){$X_3$};
\node () at (90:0.6){$X_1$};
\node () at (270:0.82){$X_2$};
\end{tikzpicture}\quad\quad
\begin{tikzpicture}[anchor=base, baseline]
\draw[color=blue, very thick] plot [smooth] coordinates {(2:1) (50:0.4) (0,0)};
\draw[color=blue, very thick, decoration={markings, mark=at position 0.82 with {\arrow[black]{|}}}, postaction={decorate}] (0:1)--(0,0);
\draw[color=blue, very thick] (210:1)--(330:1);
\draw[color=blue, very thick] (150:1)--(30:1);
\fill[color=red] (0,0) circle (0.1);
\draw[color=ao, very thick] (0,0) circle (1);
\fill[color=orange] (0,1) circle (0.1);
\fill[color=orange] (0,-1) circle (0.1);
\node () at (-0.15,0.2){$X_6$};
\node () at (0.5,-0.3){$X_4$};
\node () at (90:0.6){$X_1$};
\node () at (270:0.82){$X_2$};
\end{tikzpicture}
\end{center}

It amounts to a standard computation to match the cluster tilting objects with the above four tagged triangulations, but spelling out the details may be instructive to the reader.

We first check that the four above tagged triangulations indeed give rise to cluster tilting objects. The first cluster tilting object is given by the direct sum of 
\begin{align*} 
X_1&=(0,k[t_1^\pm],0)\\
X_2&=  (k[t_2^\pm]\oplus k[t_2^\pm][1],k[t_1^\pm],\on{counit}_{\phi_!\dashv \phi^*})\\
X_3&=(k[t_2^\pm],0,0)\\
X_4&=(k[t_2^\pm],k[t_1^\pm],\phi_!(k[t_2^\pm])\simeq k[t_1^\pm])\,.
\end{align*} 
The objects $X_1, X_2$ are injective-projective as they arise from induction from the two external edges $e_1,e_2$. We have 
\[ \on{Ext}^i_{\mathcal{V}^2_{\psi_!}}(X_3,X_4)\simeq \on{Ext}^i_{\mathcal{V}^2_{\psi_!}}(X_4,X_3)\simeq \begin{cases} k& i\text{ is even}\\ 0 & i\text{ is odd}\,,\end{cases}
\] so that $T_1=X_1\oplus X_2\oplus X_3\oplus X_4$ is rigid. 
Let 
\begin{align*} 
X_5&= (k[t_2^\pm][1],0,0)\\
X_6&= (k[t_2^\pm][1],k[t_1^\pm],\phi_!(k[t_2^\pm][1])\simeq k[t_1^\pm])\,.
\end{align*}
There is a fiber and cofiber sequence
\begin{equation}\label{eq:exseq}
X_6\to X_2\to X_3
\end{equation}
which evaluates under $\varrho_1$ to the split fiber and cofiber sequence
\[
k[t_1^\pm]\xrightarrow{\on{id}} k[t_1^\pm]\to 0
\]
and under $\varrho_2$ to the split fiber and cofiber sequence
\[
0\rightarrow k[t_1^\pm]\xrightarrow{\on{id}} k[t_1^\pm]\,.
\]
Hence, the sequence \eqref{eq:exseq} is exact. Similarly, there are exact sequences
\[ X_3 \to X_1 \to X_6\]
\[ X_5\to X_1 \to X_4\]
\[ X_4 \to X_2 \to X_5\]
showing that $T_1$ has the $2$-term resolution property. Thus $T_1$ is cluster tilting. 

Note that $T_2= T_1[1]$ corresponding to the second tagged triangulation, since suspension reveres the tagging. Thus $T_2$ is also cluster tilting.

We next check that $T_3=X_1\oplus X_2\oplus X_3\oplus X_5$ is rigid. We have 
\[ \on{Ext}^i_{\mathcal{V}^2_{\psi_!}}(X_3,X_5)\simeq \on{Ext}^i_{\mathcal{V}^2_{\psi_!}}(X_5,X_3)\simeq \begin{cases} 0& i\text{ is even}\\ k & i\text{ is odd}\,,\end{cases}\] 
so there are degree $1$ extensions. However, these are not exact: any non-trivial morphism 
\[ X_3=(k[t_2^\pm],0,0) \to X_5[1]\simeq (k[t_2^\pm],0,0)\]
evaluates under $\varrho_2$ to an equivalence $k[t_1^\pm]\simeq k[t_1^\pm]$, which is no-zero. Similarly, any non-trivial morphism $X_5\to X_3[1]$ also evaluates under $\varrho_2$ to an equivalence. We thus see that $T_3$ is rigid. The $2$-term resolution property is checked as for $T_1$, there are exact sequences
\[
X_4\to X_2 \to X_5
\]
\[ 
X_6\to X_2\to X_3
\]
\[ X_5\to X_1 \to X_4\]
\[ X_3\to X_1 \to X_6\,.\]
The object $T_4$ arises from $T_3$ under the $\mathbb{Z}/2$-rotation symmetry, given by the twist functor $C_{\mathcal{V}^2_{\phi_!}}\colon\mathcal{V}^2_{\phi_!}\simeq \mathcal{V}^2_{\phi_!}$ of the adjunction $(\varrho_1,\varrho_2)\dashv (\varrho_1^R,\varrho_2^R)$. Thus $T_4$ is also cluster tilting. 

To conclude the argument, it remains to check that there can be no cluster tilting objects containing both $X_3$ and $X_6$ or both $X_4$ and $X_5$. This is because $\on{Ext}^1_{\mathcal{V}^2_{\psi_!}}(X_3,X_6)\simeq k$, and the degree $1$ extension $X_6 \to X_2 \to X_3$ is exact (as explained above), showing that $X_3\oplus X_6$ is not rigid. Similarly, $X_4\oplus X_6$ is not rigid.
\end{proof}

\begin{remark}\label{rem:icequiv}
The ice quivers of the cluster tilting objects $T_1$ and $T_2$ from the proof of \Cref{lem:CTO2-gon} are given by 
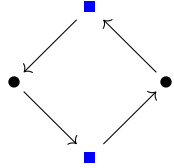
\begin{center}
\begin{tikzpicture}
\draw node[frvertex] (1) at (0,1) {};
\draw node[frvertex] (4) at (0,-1) {};
\draw node[vertex] (2) at (-1,0) {};
\draw node[vertex] (3) at (1,0) {};

\draw[->] (1)--(2);
\draw[->] (2)--(4);
\draw[->] (4)--(3); 
\draw[->] (3)--(1);
\end{tikzpicture}
\end{center}
and the ice quivers of the cluster tilting objects $T_3,T_4$ are given by 
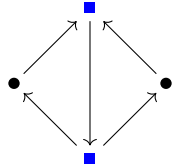
\begin{center}
\begin{tikzpicture}
\draw node[frvertex] (1) at (0,1) {};
\draw node[frvertex] (4) at (0,-1) {};
\draw node[vertex] (2) at (-1,0) {};
\draw node[vertex] (3) at (1,0) {};

\draw[->] (1)--(4);
\draw[->] (4)--(2);
\draw[->] (2)--(1);
\draw[->] (4)--(3); 
\draw[->] (3)--(1);
\end{tikzpicture}
\end{center}
\end{remark}

\begin{remark}
The two object $k[t_2^\pm],k[t_2^\pm][1]\in \D^{\on{perf}}(k[t_2^\pm])$ also define a cluster tilting object with respect to the exact structure on $ \D^{\on{perf}}(k[t_2^\pm])$ induced by $\phi_!$. 
\end{remark}

\begin{construction}\label{constr:taggedtrianglfromgraph}
We suppose that ${\bf S}$ is not a punctured monogon. We can choose a spanning graph $\rgraph$ of ${\bf S}$ with $2$-valent vertices at the punctures and all other vertices $3$-valent. We further suppose that we made a choice for each $2$-valent vertex of $\rgraph$ one of the four tagged triangulations of the once-punctured $2$-gon. There is an induced tagged triangulation $\mathcal{T}$ of ${\bf S}$ whose tagged arcs are obtained as follows:

Near each $3$-valent vertex of $\rgraph$, there are three local trajectories, which we extend to clockwise trajectories in ${\bf S}$. We can consider each of these trajectories  as a tagged arc (with trivial tagging), so that these three arcs form an ideal triangle. 

Similarly, at each $2$-valent vertex of $\rgraph$, the two tagged arcs in the once-punctured $2$-gon ending at the puncture in the chosen tagged triangulation can be extended to tagged arcs in ${\bf S}$ as clockwise trajectories (away from the puncture). 
\end{construction}

\begin{remark}
The tagged triangulation obtained in \Cref{constr:taggedtrianglfromgraph} gives rise to an ice quiver, used for the seed of the corresponding cluster algebra. This ice quiver is the amalgamation of the ice quivers from \Cref{rem:icequiv} associated with the $2$-valent vertices of $\rgraph$ and the ice quiver 
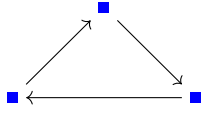
\begin{center}
\begin{tikzpicture}[scale=1.2]
\draw node[frvertex] (1) at (0,1) {};
\draw node[frvertex] (2) at (-1,0) {};
\draw node[frvertex] (3) at (1,0) {};

\draw[->] (2)--(1);
\draw[->] (3)--(2);
\draw[->] (1)--(3); 
\end{tikzpicture}
\end{center}
associated with the $3$-valent vertices of $\rgraph$. A similar observation appears in \cite[Rem.~4.2]{FST08}.
\end{remark}

\begin{proposition}\label{prop:puncturedsurfaceCTO}
Let ${\bf S}$ be a punctured marked surface which is not a punctured monogon. Consider a spanning graph $\rgraph$ and tagged triangulation $\mathcal{T}$ of ${\bf S}$ as in \Cref{constr:taggedtrianglfromgraph}. Let $\mathcal{F}$ be a $\rgraph$-parametrized perverse schober, such that 
\begin{itemize}
\item  the generic stalk of $\mathcal{F}$ is the $1$-periodic perfect derived $\infty$-category $\D^{\on{perf}}(k[t_1^\pm])$,
\item at each $3$-valent vertex of $\rgraph$, the spherical functor underlying $\mathcal{F}$ is trivial, meaning that $\mathcal{F}$ is locally constant at that vertex, and
\item at each $2$-valent vertex of $\rgraph$ the spherical functor underlying $\mathcal{F}$ is given by 
\[\phi_!\colon \D^{\on{perf}}(k[t_2^\pm])\to \D^{\on{perf}}(k[t_1^\pm])\,.\]
\end{itemize}
Then the Frobenius exact $\infty$-category $\mathcal{H}(\rgraph,\mathcal{F})$ admits a cluster tilting object whose corresponding ice quiver arises from the tagged triangulation.
\end{proposition}

\begin{proof}
This follows from combining \Cref{thm:clustertiltingschober}, \Cref{prop:amalgamationCTO} and \Cref{lem:CTO2-gon}. We also point out that there is an ambiguity in the construction of the cluster tilting object without fixing the specific perverse schober, since the shift functor $[1]$ acts by reversing the taggings at punctures.
\end{proof}

\bibliography{biblio} 
\bibliographystyle{alpha}

\textsc{Mathematisches Institut, Universit\"at Bonn, Endenicher Allee 60, 53115 Bonn, Germany}

\textit{Email address:} \texttt{christ@math.uni-bonn.de}

\end{document}